\documentclass[reqno,11pt]{amsart}

\usepackage{amssymb}
\usepackage{amssymb, latexsym,bm}
\usepackage{hyperref}
\usepackage{amsmath}
\usepackage{amscd}
\usepackage{enumerate}
\usepackage{amsfonts}
\usepackage{graphicx}
\usepackage[all]{xypic}
\usepackage{mathrsfs}
\usepackage{tikz}

\usepackage{setspace}
\usepackage[titletoc]{appendix}

\allowdisplaybreaks

\newcommand{\C}{{\mathbb{C}}}
\newcommand{\R}{{\mathbb{R}}}

\theoremstyle{plain}
\newtheorem{theorem}{Theorem}
\newtheorem{proposition}[theorem]{Proposition}
\newtheorem{lemma}[theorem]{Lemma}
\newtheorem{corollary}[theorem]{Corollary}

\theoremstyle{definition}

\newtheorem{remark}[theorem]{Remark}

\newtheorem{claim}{Claim}

\usepackage{setspace}

\numberwithin{equation}{section}
\numberwithin{theorem}{section}

\newcommand{\step}[1]{\textit{#1}.~~}

\renewcommand{\R}{\mathbb{R}}
\renewcommand{\i}{i}
\newcommand{\set}[1]{~~\left\{{#1}\right\}~~}

\newcommand{\bigo}[1]{\mathrm{O}\left(~{#1}~\right)}
\newcommand{\abs}[1]{\left|{#1}\right|}
\newcommand{\sts}[1]{\left({#1}\right)}
\newcommand{\ltl}[1]{~\left\{{#1}\right\}}

\newcommand{\bilin}[2]{\left(\,{#1}\, ,\,{#2}\,\right)}
\newcommand{\action}[2]{\left({#1}\, ,\,{#2}\right)}

\newcommand{\normhone}[1]{\left\|#1\right\|_{H^1(\R)}}

\newcommand{\normto}[1]{\|#1\|_{2}}
\newcommand{\dx}{\mathrm{d}x}
\newcommand{\ess}{\mathrm{ess}}
\newcommand{\spann}{\mathrm{span}}
\newcommand{\mass}[1]{M\left({#1}\right)}
\newcommand{\momentum}[1]{P\left({#1}\right)}

\makeatletter

\newcommand{\Rmnum}[1]{\expandafter\@slowromancap\romannumeral #1@}

\makeatother
\begin{document}

\onehalfspacing

\title[Stability of the traveling waves of DNLS]
{Stability of the traveling waves for the derivative Schr\"{o}dinger
equation in the energy space }

\author[]{Changxing Miao}
\address{\hskip-1.15em Changxing Miao:
\hfill\newline Institute of Applied Physics and Computational
Mathematics, \hfill\newline P. O. Box 8009,\ Beijing,\ China,\
100088.}  \email{miao\_changxing@iapcm.ac.cn}

\author[]{Xingdong Tang}
\address{\hskip-1.15em Xingdong Tang \hfill\newline Beijing Computational Science Research Center, \hfill\newline No. 10 West Dongbeiwang Road, Haidian
District, Beijing, China, 100193,}
\email{xdtang202@163.com}

\author[]{Guixiang Xu}
\address{\hskip-1.15em Guixiang Xu \hfill\newline Institute of
Applied Physics and Computational Mathematics, \hfill\newline P. O.
Box 8009,\ Beijing,\ China,\ 100088.}
\email{xu\_guixiang@iapcm.ac.cn}



\begin{abstract}
In this paper, we continue the study of the dynamics of the
traveling waves for nonlinear Schr\"{o}dinger equation with derivative (DNLS) in the energy space. Under some technical
assumptions on the speed of each traveling wave, the stability of
the sum of two traveling waves for DNLS is obtained in the energy
space by Martel-Merle-Tsai's analytic approach in
\cite{MartelMT:Stab:gKdV, MartelMT:Stab:NLS}. As a by-product, we
also give an alternative proof of the stability of the single
traveling wave in the energy space in \cite{ColinOhta-DNLS}, where
Colin and Ohta made use of the concentration-compactness argument.
\end{abstract}

\maketitle


\section{Introduction}
In this paper, we consider the Cauchy problem for nonlinear
Schr\"{o}dinger equation with derivative (DNLS)
\begin{equation} \label{DNLS-a}
\left\lbrace \aligned
    &
    \i\partial_t v + \partial^2_x v + \i\partial_x\sts{\abs{v}^2 v}=0,\; t\in \R, \\
    &
    v\sts{0,x}=v_0\sts{x}\in H^1(\R).
\endaligned
\right.
\end{equation}
\eqref{DNLS-a} appears in plasma physics \cite{MOMT-PHY, M-PHY,
SuSu-book}. There are many equivalent forms under the gauge
transformation. For instance, if we take the following gauge
transformation,
\begin{align*}v(t,x)\mapsto u(t,x)=G_{3/4}(v)(t,x) \triangleq e^{\i \frac34
\int^x_{-\infty}|v(t,\eta)|^2\;d\eta}v(t,x), \end{align*}
 then
\eqref{DNLS-a} is equivalent to the following equation (DNLS)
\begin{equation} \label{DNLS}
\left\lbrace \aligned
    &
    \i\partial_t u + \partial^2_x u  + \frac{1}{2}\i \abs{u}^2\partial_x u - \frac{1}{2}\i u^2\partial_x\overline{u}+\frac{3}{16}\abs{u}^4 u  =0,\; t\in \R\\
    &
    u\sts{0,x}=u_0\sts{x}\in H^1(\R),
\endaligned
\right.
\end{equation}
which is $L^2$-critical NLS with derivative for the fact that the scaling
symmetry
\begin{align*}u(t,x)\mapsto u_{\lambda}(t,x)=\lambda^{1/2}u(\lambda^2t, \lambda x)
\end{align*}
leaves both \eqref{DNLS} and the mass invariant. The mass, momentum
and energy are defined as following
\begin{align*}
M(u)(t)= & \frac12 \int |u(t,x)|^2 \;dx, \\
P(u)(t)=& -\frac12 \Im \int \left(\bar{u}\,\partial_x u\right)(t,x)
dx + \frac18 \int
|u(t,x)|^4\; dx,\\
E(u)(t)=& \frac12 \int |\partial_x u(t,x)|^2 \;dx -\frac{1}{32} \int
|u(t,x)|^6\; dx.
\end{align*}
They are conserved under the flow \eqref{DNLS} according to the
phase rotation invariance, spatial translation invariance and time
translation invariance respectively. Compared with the $L^2$-critical
NLS, \eqref{DNLS-a} or \eqref{DNLS} doesn't enjoy the Galilean
invariance and pseudo-conformal invariance any more.

Local well-posedness result for \eqref{DNLS} in the energy space has
been worked out by Hayashi and Ozawa \cite{HaOz-94, Oz-96}. They
combined the fixed point argument with $L^4_IW^{1,\infty}(\R)$
estimate to construct the local-in-time solution with arbitrary data
in the energy space. For other kinds of local well-posedness
results, we can refer to \cite{Ha-93, HaOz-92}. Since \eqref{DNLS}
is $\dot H^1$-subcritical, the maximal lifespan interval only depends on
the $H^1$ norm of initial data. More precisely, we have
\begin{theorem}\label{Thm:LWP}\cite{HaOz-94, Oz-96} For any $u_0 \in H^1(\R)$ and $t_0 \in \R$, there
exists a unique maximal-lifespan solution $u:I\times \R \rightarrow
\C$ to \eqref{DNLS} with $u(t_0)=u_0$, the map $u_0\rightarrow u$ is continuous from $H^1(\R)$ to $C(I, H^1(\R))\cap L^4_{loc}(I; W^{1,\infty}(\R))$.  Moreover, the solution also
has the following properties:
\begin{enumerate}
\item $I$ is an open neighborhood of $t_0$.

\item The mass, momentum and energy are conserved, that is, for all
$t\in I$,
\begin{align*}
M(u)(t)=M(u)(t_0), \;\; P(u)(t)=P(u)(t_0),\;\; E(u)(t)=E(u)(t_0).
\end{align*}

\item If\; $\sup(I)<+\infty$,  $(\text{or}\, \inf(I)>-\infty)$ , then
\begin{align*}
\lim_{t\rightarrow\sup(I)}\big\|\partial_x
u(t)\big\|_{L^2}=+\infty,\;\; \left(
\lim_{t\rightarrow\inf(I)}\big\|\partial_x u(t)\big\|_{L^2}=+\infty,
respectively.\right)
\end{align*}

\item If $\big\| u_0 \big\|_{H^1}$ is sufficiently small,
then $u$ is a global solution.
\end{enumerate}
\end{theorem}

The sharp local well-posedness result in $H^s, s\geq 1/2$ is due to
Takaoka \cite{Ta-99} by using Bourgain's space. The sharpness is
shown in \cite{Ta-01} in the sense that nonlinear evolution
$u(0)\mapsto u(t)$ fails to be $C^3$ or even uniformly $C^0$ in this
topology, even when $t$ is arbitrarily close to zero and $H^s$ norm
of the data is small (see also Biagioni-Linares
\cite{BiLi-01-Illposed-DNLS-BO}).

After the sharp local wellposedness is obtained, there are two
aspects of global solution of \eqref{DNLS} to be concerned. One is
about the global wellposedness in the lower regularity space $H^s(\R)$ for some $s<1$,
another one is about the dynamics of the traveling waves in the
energy space $H^1(\R)$.

On one hand,  the global well-posedness is obtained for \eqref{DNLS}
in the energy space in \cite{Oz-96} under the smallness condition
\begin{align}\label{Cond:smalldata}
\|u_0\|_{L^2} < \sqrt{2 \pi},
\end{align}
the argument is based on the energy method (conservation of mass and
energy) together with the sharp Gagliardo-Nirenberg inequality
\cite{Wein}. This is improved by Takaoka \cite{Ta-01}, who proved
global well-posedness in $H^s$ for $s>32/33$ under the condition
\eqref{Cond:smalldata}. His argument is based on Bourgain's
restriction method, which separated the evolution of low frequencies
and of high frequencies of initial data and noticed that the
nonlinear evolution has $H^1$ regularity effect even for the rough
solution $u\in H^s$, $s<1$. In \cite{CKSTT-01, CKSTT-02}, I-team
used the ``I-method'' to show global well-posedness in $H^s, s>1/2$
under \eqref{Cond:smalldata}, I-team defined $Iu$ as a modified
 function, whose energy is nearly conserved in time by
capturing nonlinear cancellation in frequency space under the flow
\eqref{DNLS}. Later, Miao, Wu and Xu \cite{MiaoWX-2011} showed the
sharp global well-posedness in $H^{1/2}$ under
\eqref{Cond:smalldata} by using I-method together with the refined
resonant decomposition. For the global result, we can also refer to
the recent paper \cite{LiuPS:DNLS:ISM, Pelin:DNLS} by the inverse scattering
method.

On the other hand, it is known in \cite{KaupN:DNLS:soliton,
MiaoTX-2015, SaarloosH:PhyD} that \eqref{DNLS} has a two-parameter
family of the traveling waves with the form:
\begin{align}\label{def:tw:a}u(t,x) \triangleq & \; e^{i\omega t } \varphi_{\omega,c}(x-ct) \\
\triangleq &\; e^{i\omega t + \i
\frac{c}{2}(x-ct)}\phi_{\omega,c}(x-ct) \label{def:tw struct:a}
\end{align}where $(\omega,c)\in \R^2, \; c^2<4\omega$ and
\begin{align}\label{tw:rsol:a}\phi_{\omega,c}(x)=\left[\frac{\sqrt{\omega}}{4\omega-c^2}
\left\{\cosh\big(\sqrt{4\omega-c^2}x\big)-\frac{c}{2\sqrt{\omega}}\right\}\right]^{-1/2},
\end{align}
which is a positive solution of
\begin{align}\label{eq:rsol:a}\left(\omega-\frac{c^2}{4}\right)\phi - \partial^2_x \phi
-\frac{3}{16}\left|\phi\right|^4\phi = -\frac{c}{2}
\left|\phi\right|^2\phi.
\end{align}
In fact, by \eqref{DNLS}and \eqref{def:tw:a}, we know that the
solitary solution $ \varphi_{\omega,c}$ should satisfy the following
equation
\begin{align}\label{eq:tw}
\omega\varphi - \partial^2_x \varphi
-\frac{3}{16}\left|\varphi\right|^4\varphi = -ic\partial_x \varphi+
\frac12 i |\varphi|^2\partial_x \varphi - \frac12 i \varphi^2
\partial_x \overline{\varphi}.
\end{align}
In \cite{MiaoTX-2015}, the authors characterized all solutions to
\eqref{eq:tw} in the energy space, which corresponds to all
traveling waves to \eqref{DNLS}.

\begin{theorem}\label{Thm:Exist:TW}Let
\begin{equation*}
\phi_{\omega,c}(x)= \left\lbrace \aligned
\left[\frac{\sqrt{\omega}}{4\omega-c^2}
\left\{\cosh\big(\sqrt{4\omega-c^2}x\big)-\frac{c}{2\sqrt{\omega}}\right\}\right]^{-1/2},
&\; c^2 < 4\omega
\\
 2\sqrt{c}\cdot \left(c^2x^2+1\right)^{-1/2}, &\; c^2=4\omega,\; c>0.
\endaligned
\right.
\end{equation*}
Then the following results hold
\begin{enumerate}
\item For the subcritical case $c^2<4\omega$. $\varphi_{\omega,c}(x)=e^{i\frac{c}{2}x} \phi_{\omega,c}(x) $ is the unique
solution of \eqref{eq:tw} among nontrivial solution in $H^1(\R)$, up
to the phase rotation and spatial translation symmetries of
\eqref{eq:tw}.

\item For the critical case $c^2=4\omega, c>0$. $\varphi_{\omega,c}(x)=e^{i\frac{c}{2}x} \phi_{\omega,c}(x)$ is the unique
solution of \eqref{eq:tw} among nontrivial solution in $H^1(\R)$, up
to the phase rotation and spatial translation symmetries of
\eqref{eq:tw}.

\item  For the critical case $c^2=4\omega, c\leq0$ and the supercritical case $c^2>4\omega$. \eqref{eq:tw}
has no nontrivial solution in $H^1(\R)$.
\end{enumerate}
\end{theorem}
\begin{remark}The key observation to the above result is to make use of
the structure of solution to \eqref{eq:tw} according to
\eqref{def:tw struct:a}, it is equivalent to solve a semilinear ODE
with the Nehari manifold argument and the non-increasing
rearrangement technique \cite{AmbMal:NonAnal, BerestLions:NLS:GS,
LiebL:book, Willem:MinMax}. As a consequence of the variational
characterization of $\varphi_{\omega,c}(x)$, a sufficient condition
on the global wellposedness of the solution $u(t,x)$ to \eqref{DNLS}
was shown in $H^1$ in \cite{MiaoTX-2015}. That is,  if the initial
data $u_0\in H^1(\R)$ satisfies\footnote{Here the minimum action
functional $J_{\omega,c}$ and the scaling derivative functional
$K_{\omega,c}$ are defined in $H^1(\R)$ as following
\begin{align*}
J_{\omega,c}(\varphi)\triangleq &\; E(\varphi)+\omega M(\varphi)+c P(\varphi),\\
K_{\omega,c}\sts{\varphi}\triangleq &\;
\int\sts{\abs{\varphi_x}^2-\frac{3}{16}\abs{\varphi}^6+\omega\abs{\varphi}^2
- c\Im\sts{\overline{\varphi}\varphi_x}+\frac{c}{2}\abs{\varphi}^4
  }dx.
\end{align*}}
$ J_{\omega,c}\big(u_0\big) <  J_{\omega,c}(\varphi_{\omega,c}), \;
K_{\omega,c}\sts{u_0} \geq 0 $ for some $(\omega, c)$ with $c^2<
4\omega$ or $c^2=4\omega, c>0$, then the solution $u(t)$ to
\eqref{DNLS} exists globally in time. While there is no blowup result for
the solution $u(t)$ with $J_{\omega,c}\big(u_0\big) <
J_{\omega,c}(\varphi_{\omega,c}), \; K_{\omega,c}\sts{u_0} < 0 $ in
$H^1(\R)$ because of the lack of the effective Virial identity.
\end{remark}

In this paper, we will focus on the study of the stability of the
traveling waves in the energy space. For the subcritical case $c^2<
4\omega$, Colin and Ohta made use of the concentration compactness
argument to show the stability of the single traveling wave for
\eqref{DNLS} in \cite{ColinOhta-DNLS}, which extended the result in
\cite{GuoWu:orbitalStab}. It is noticed that Martel, Merle and Tsai
developed some powerful analytic approach to show the stability of
the multi-soliton wave for subcritical gKdV and NLS in
\cite{MartelMT:Stab:gKdV, MartelMT:Stab:NLS}, which is based on the
modulation analysis \cite{Wein:stab:SJMA, Wein:stab:CPAM},
perturbation theory, monotonicity formulas and the conservation
laws. For the orbital stability results, we can also refer to
\cite{Caz:NLS:book, CaL:NLS:stable, DikaM:Stab:BBM,
GrillSS:Stable:87, munoz:stable, Pave:book, Tao:survey}. Here we
will apply this analytic method to \eqref{DNLS} and obtain the
following results. First of all, we revisit the stability of the
single traveling wave for \eqref{DNLS} in the energy space.

\begin{theorem}\label{thm:stab:asln} For any $(\omega^0, c^0)\in \R^2$ with
$\sts{c^0}^2 < 4\omega^0 $, the traveling wave solution $
e^{i\omega^0 t } \varphi_{\omega^0,c^0}(x-c^0t)$ to \eqref{DNLS} is
orbitally stable in the energy space. That is, for any $\epsilon>0$,
there exists $\delta>0$ such that if $u_0\in H^1(\R)$ satisfies
$$\big\|u_0(\cdot)-
\varphi_{\omega^0,c^0}(\cdot-x^0)e^{i\gamma^0
}\big\|_{H^1(\R)}<\delta$$ for some $(x^0, \gamma^0)\in \R^2$, then
the solution $u(t)$ of \eqref{DNLS} exists globally in time and
satisfies
\begin{align*}
\sup_{t\geq0}\inf_{(y, \gamma)\in \R^2}\big\|u(t,\cdot)-
\varphi_{\omega^0,c^0}(\cdot-y)e^{i\gamma
}\big\|_{H^1(\R)}<\epsilon.
\end{align*}

\end{theorem}

\begin{remark}
\begin{enumerate}
\item About the dynamics of the solution to \eqref{DNLS} around the traveling waves, the above result
extends the results in \cite{MiaoTX-2015, Wu-DNLS}.

\item As for the
critical parameters $c^2=4\omega, c>0$, we don't know whether the
corresponding traveling waves for \eqref{DNLS} in \cite{MiaoTX-2015}
is stable or not since there is in lack of the spectral gap for the
spectrum of the linearized operator around the traveling waves.

\end{enumerate}
\end{remark}

Secondly, we can show the stability of the sum of two traveling
waves for \eqref{DNLS} in the energy space when the centers of two
traveling waves are always far away (that is, weak interaction) from
each other. That is,

\begin{theorem}\label{thm:stab:tsln}
Let $(\omega^0_k, c^0_k)\in \R^2$, $k=1,2$ satisfy
\begin{enumerate}
\item[$(a)$] Nonlinear stability of each wave: $\left(c^0_k\right)^2 < 4
\omega^0_k$ for $k=1, 2$.

\item[$(b)$] Forward propagation of each wave: $0<c^0_1<c^0_2$.

\item[$(c)$] Relative speed: $ \max(c^0_1, 0) <
2\;\frac{\omega^0_2-\omega^0_1}{c^0_2-c^0_1}<c^0_2$.
\end{enumerate}
Then there exists $C, \delta_0, \theta_0, L_0>0$, such that if
$0<\delta<\delta_0,\; L>L_0$ and
\begin{align*}
\left\|u_0(\cdot)-\sum^2_{k=1}\varphi_{\omega^0_k,
c^0_k}(\cdot-x^0_k)e^{i\gamma^0_k}\right\|_{H^1(\R)} \leq \delta,
\end{align*}
with $ x^0_2-x^0_1>L$, then the solution $u(t)$ of \eqref{DNLS}
exists globally in time and there exist $\mathcal{C}^1$ functions
$x_k(t)$ and $\gamma_k(t)$, $k=1,2$ such that for any $t\geq 0$,
\begin{align*}
\left\|u(t,\cdot)-\sum^2_{k=1}\varphi_{\omega^0_k,
c^0_k}(\cdot-x_k(t))e^{i\gamma_k(t)}\right\|_{H^1(\R)} \leq
C\left(\delta+e^{-\theta_0 \frac{L}{2}}\right).
\end{align*}
\end{theorem}

\begin{remark}
\begin{enumerate}
\item Since the traveling wave with $\left(c^0_k\right)^2<4\omega^0_k$ is orbital stable
by Theorem \ref{thm:stab:asln}, we call the condition
$\left(c^0_k\right)^2<4\omega^0_k$ {\it nonlinear stable condition}.

\item
About assumption $(b)$: Forward propagation of each wave. It is a
special case for the two-traveling wave without collisions, the
arguments can be extended to other cases without collisions: $(I):
c^0_1<0<c^0_2$ (The right one propagates forward, the left one
propagate backward). From the view of physics, it is more stable
since two traveling waves propagate in different directions. $(II):
c^0_1<c^0_2<0$ (Backward propagation of each wave). It is the same
as the case we consider by symmetry. In fact, it is more easier than
the case we consider since we needn't modulate the speed parameter
$c_k$ in the modulation analysis.

\item
About assumption $(c)$: Relative speed. Compared with the single traveling wave case, the
multi-traveling wave case is more technical. In fact, after the
modulation analysis of the solution around the sum of two traveling
waves, we use the technical assumption $(c)$ for some monotonicity
formulas to show the localized action functional
$\mathfrak{E}\sts{u\sts{t}}$ is almost conserved, and obtain the
refined estimates of the parameters
$|\omega_k(t)-\omega_k(0)|+|c_k(t)-c_k(0)|$, $k=1, 2$. Please refer
Section \ref{sect:stab:tsln} to the details.

\item
About the the $k$-soliton case ($k\geq 3$),  Le Coz and Wu \cite{LeWu:DNLS} obtained the
stability of a $k$-soliton solution of (DNLS) independently a few months after we submitted our paper. The main difference is that Le Coz and Wu considered the $\frac12$-gauge transformed version of \eqref{DNLS-a}, whereas our
paper is based on the form \eqref{DNLS}, which corresponds to the $\frac34$-gauge transformed version of \eqref{DNLS-a}.
\end{enumerate}
\end{remark}

At last, the paper is organized as following.  In Section
\ref{sect:lin oper},  we introduce the linearized operator around
single traveling wave, show the coercivity property of the
linearized energy and obtain the geometric decomposition of the
solution around single traveling wave. In Section
\ref{sect:stab:asln}, inspired by the ideas in
\cite{MartelMT:Stab:NLS}, on one hand, we introduce a conserved
functional, which is related to the mass, momentum and energy, to
obtain a refined estimate about the remainder term in the modulation
analysis (geometric decomposition) of the solution. On the other
hand, we use the conservation laws of mass and momentum to refine
the estimate of the parameters $|\omega(t)-\omega(0)|+|c(t)-c(0)|$.
Together with the continuity argument, these refined estimates imply
Theorem \ref{thm:stab:asln}. In Section \ref{sect:decomp:tsln}, we
give the modulation analysis of the solution around the sum of two
traveling waves with weak interactions. In Section \ref{sect:mono
forms}, we introduce some extra monotonicity formulas and their
dynamics. In Section \ref{sect:stab:tsln}, on one hand, we introduce
a localized action functional, which is almost conserved by the
monotonicity formula and the conservation laws of mass, momentum and
energy, to refine the estimate about the remainder term in the
modulation analysis of the solution. On the other hand, we use some
monotonicity formulas to refine the estimates of
$|\omega_k(t)-\omega_k(0)|+|c_k(t)-c_k(0)|$, $k=1, 2$ besides of the
conservation laws of mass and momentum. These refined estimates also
imply Theorem \ref{thm:stab:tsln} together with the continuity
argument.

In Appendix \ref{app:coer:arf}, we give the coercivity of the
quadratic term (i.e. the linearized energy) under the orthogonal
structures. In Appendix \ref{app:tsln:acfln}, we use the
perturbation theory to linearize the action functional of the
solution around the sum of two traveling waves. In Appendix
\ref{app:tsln:coer}, we show the coercivity properties of the
localized quadratic term under the orthogonal structures.


\section{Properties of the linearized operator around the traveling
wave}\label{sect:lin oper} In this section, we will describe some
basic properties of the traveling wave with the subcritical parameters
$c^2<4\omega$ for \eqref{DNLS}.

When $c^2 < 4\omega$, it is well known that \eqref{DNLS} has the
following kinds of traveling wave solutions (for instance, see
\cite{KaupN:DNLS:soliton, MiaoTX-2015, SaarloosH:PhyD})
\begin{align}\label{def:tw}
    u_{\omega,c}(t,x)
        \triangleq &  e^{i\omega t } \varphi_{\omega,c}(x-ct) \\
        \triangleq &  e^{i\omega t + \i \frac{c}{2}(x-ct)}\phi_{\omega,c}(x-ct)
\label{def:tw struct}
\end{align}
where
\begin{align}\label{tw:rsol}\phi_{\omega,c}(x)=\left[\frac{\sqrt{\omega}}{4\omega-c^2}
\left\{\cosh\big(\sqrt{4\omega-c^2}x\big)-\frac{c}{2\sqrt{\omega}}\right\}\right]^{-1/2},
\end{align}
which is the unique positive solution up to the symmetries of
\eqref{eq:rsol:a}.\footnote{For $(\omega,c)\in \R^2$, we have
characterized all solutions for \eqref{eq:tw} in Theorem
\ref{Thm:Exist:TW}.} By the concentration-compactness argument in
\cite{ColinOhta-DNLS}, we know that the traveling waves
\eqref{def:tw} with $c^2 < 4\omega$ are orbitally stable. It is
worth noticing that the condition $c^2 < 4 \omega$ implies the
following non-degenerate condition \cite{ColinOhta-DNLS}
\begin{align}\label{nondeg}
  \det d''\sts{\omega,~c} &<0,
\end{align}
where $ d\sts{\omega,~c}
        \triangleq
    J_{\omega,c}\sts{\varphi_{\omega,c}}
       =     E\sts{\varphi_{\omega,c}}+\omega M\sts{\varphi_{\omega,c}}+c P\sts{\varphi_{\omega,c}}$.
In fact, since $\varphi_{\omega,c}$ is the critical point of the
action functional $ J_{\omega,c}(\varphi)$ in \cite{MiaoTX-2015}, we
have
\begin{align*}
    \partial_{\omega}d\sts{\omega, c} =\mass{ \varphi_{\omega,c} },
    \quad
    \partial_{c}d\sts{\omega,c} =\momentum{ \varphi_{\omega,c} },
\end{align*}
and
\begin{align*}
    d''\sts{\omega, c} = \begin{pmatrix}
                              \partial_{\omega}\mass{ \varphi_{\omega,c} } & \partial_{c}\mass{ \varphi_{\omega,c} } \\
                              \partial_{\omega}\momentum{ \varphi_{\omega,c} } & \partial_{c}\momentum{ \varphi_{\omega,c} } \\
                            \end{pmatrix}.
\end{align*}

\subsection{Linearized operator and its coercivity}\label{sect:coer}
We now consider the linearized operator around the traveling waves
and its coercivity in $H^1(\R)$ in this subsection.

Let $c^2 < 4 \omega$, we define the operators for any real valued
functions $\eta_1, \eta_2 \in H^1(\R)$
\begin{align*}
  \mathcal{L}_{+}\eta_1&\triangleq -\frac12 \partial^2_{x}\eta_1+\frac12\sts{ \omega-\frac{c^2}{4} } \eta_1+\frac{3}{4}c\,\phi_{\omega,c}^2\eta_1 - \frac{15}{32}\phi_{\omega,c}^4\eta_1,
\\
  \mathcal{L}_{-}\eta_2&\triangleq -\frac12\partial^2_{x}\eta_2+\frac12\sts{ \omega-\frac{c^2}{4} } \eta_2+\frac{1}{4}c\,\phi_{\omega,c}^2\eta_2 - \frac{3}{32}\phi_{\omega,c}^4\eta_2,
\end{align*}
and the quadratic form for any complex valued functions $\eta=\eta_1
+ \i \eta_2 \in H^1(\R)$ \begin{align*} \widetilde{
\mathcal{H}}_{\omega,c} \bilin{\eta}{\eta} \triangleq &
\bilin{\mathcal{L}_{+} \eta_1}{\eta_1}+ \bilin{\mathcal{L}_{-}
\eta_2}{\eta_2}
\\ =& \int
    \frac12\abs{\partial_x \eta}^2+\frac12\sts{\omega-\frac14
    c^2}\abs{\eta}^2\; \dx \\
    & + \int \left(
    -\frac{3}{32}\phi^4_{\omega,c}\abs{\eta}^2 -\frac38 \phi^4_{\omega,c}\abs{\eta_1}^2
    + \frac14 c  \phi^2_{\omega,c}\abs{\eta}^2 +\frac12 c
    \phi^2_{\omega,c}\abs{\eta_1}^2 \right)\; \dx.
\end{align*}

The next result shows that  for the subcritical case, the coercivity
of the linearized operator around the traveling wave holds under
some orthogonal conditions in the energy space.
\begin{proposition} \label{prop:coer:arf} Assume that $c^2 < 4 \omega$, then the following results hold
\begin{enumerate}
\item If $\psi\in H^{1}\sts{\R}$ satisfies
$
    \bilin{\psi}{\phi_{\omega,c}}             = 0,
$
then there exists $C_1>0$ such that
  \begin{align*}
    \bilin{\mathcal{L}_{-} \psi}{\psi}\geq C_1 \normto{\psi}^2.
  \end{align*}
\item If $\psi\in H^{1}\sts{\R}$ satisfies
$
    \bilin{\psi}{\phi_{\omega,c}}             = 0,\;
    \bilin{\psi}{\phi_{\omega,c}^3}           = 0,\text{~and~}
    \bilin{\psi}{\partial_x\phi_{\omega,c}}   = 0,
$
then there exists $C_2>0$ such that
  \begin{align*}
    \bilin{\mathcal{L}_{+} \psi}{\psi}\geq C_2 \normto{\psi}^2.
  \end{align*}
  \end{enumerate}
\end{proposition}

The proof of Proposition \ref{prop:coer:arf} follows from a standard
variational argument. We refer to Appendix \ref{app:coer:arf} for
the proof. As a consequence, we have
\begin{proposition}\label{prop:coer:asln}
Assume that $c^2 < 4 \omega$. Let $\eta\in H^{1}(\R)$ be such that $
    \bilin{\Im \eta}{\phi_{\omega,c}}             = 0,
$
and
  \begin{align*}
      \bilin{\Re \eta }{\phi_{\omega,c}}             = 0,\quad
    \bilin{\Re \eta}{\partial_x\phi_{\omega,c}}   = 0,\quad
    \frac{1}{2}\bilin{\Re \eta}{\phi_{\omega,c}^3}
        +\bilin{\Im \eta}{\partial_x\phi_{\omega,c}}        = 0,
  \end{align*}
then there exists $C>0$ such that
  \begin{align}
    \label{coer:asln}
   \widetilde{ \mathcal{H}}_{\omega, c}\bilin{\eta}{\eta}\geq C  \normhone{ \eta }^2.
  \end{align}
\end{proposition}
\begin{proof}
Let $\eta_1 = \Re \eta$ and  $\eta_2 = \Im \eta$. In order to prove
\eqref{coer:asln}, it suffices to show that
\begin{align}\label{coer:a}
  \widetilde{  \mathcal{H}}_{\omega, c}\bilin{\eta}{\eta}\geq C \normto{ \eta }^2.
\end{align}

On one hand, by Proposition \ref{prop:coer:arf} and
  \begin{align*}
    \bilin{\eta_1}{\phi_{\omega,c}}             = 0,\quad
    \bilin{\eta_1}{\partial_x\phi_{\omega,c}}   = 0,
  \end{align*}
there exist constants $c, \widetilde{C}_1>0$ such that
  \begin{align}
    \label{eq:1}
    \normto{\eta_{1}}^{2}
        \leq&
        \frac{1}{c }\bilin{\mathcal{L}_{+}\eta_{1}}{\eta_{1}}
        +\widetilde{C}_1 \bilin{\eta_{1}}{\phi_{\omega,c}^{3}}^{2}.
  \end{align}
 Now inserting
\begin{align*}
    \frac{1}{2}\bilin{\eta_1}{\phi_{\omega,c}^3}
        +\bilin{\eta_2}{\partial_x\phi_{\omega,c}}        = 0
\end{align*}
into \eqref{eq:1} and by the Cauchy-Schwarz inequality, we have
\begin{align}
    \normto{\eta_{1}}^{2}
        \leq&
        \frac{1 }{ c } \bilin{\mathcal{L}_{+}\eta_{1}}{\eta_{1}}
        +4 \widetilde{C}_1 \bilin{\eta_{2}}{ \partial_{x}\phi_{\omega,c} }^{2}
        \notag\\
    \leq& \frac{1 }{ c } \bilin{\mathcal{L}_{+}\eta_{1}}{\eta_{1}}
        +\widetilde{C}_2 \|\eta_2\|^2_{2}. \label{eq:2}
\end{align}

On the other hand, by Proposition \ref{prop:coer:arf} and
$\bilin{\eta_2}{\phi_{\omega,c}}  = 0,$ there exists a constant
$C_1>0$ such that
\begin{align}\label{eq:3}
 C_1 \; \normto{\eta_2}^2\leq \bilin{\mathcal{L}_{-}
 \eta_2}{\eta_2}.
\end{align}
Without loss of generality, we can choose $0<c\ll1$ with $c\left(1+\widetilde{C}_2\right)\leq C_1 $ by Proposition \ref{prop:coer:arf}. By \eqref{eq:2} and \eqref{eq:3}, we have

\begin{align*}
 \normto{\eta_{1}}^{2} +  \normto{\eta_2}^2  \leq & \frac{1 }{ c } \bilin{\mathcal{L}_{+}\eta_{1}}{\eta_{1}}
        +\widetilde{C}_2 \|\eta_2\|^2_{2} +\frac{1}{C_1} \bilin{\mathcal{L}_{-}
 \eta_2}{\eta_2}\\
        \leq &  \frac{1 }{ c } \bilin{\mathcal{L}_{+}\eta_{1}}{\eta_{1}}
        +\left(\frac{\widetilde{C}_2}{C_1} +\frac{1}{C_1} \right) \bilin{\mathcal{L}_{-} \eta_2}{\eta_2}
        \\
        \leq &
        \frac{1 }{ c } \left( \bilin{\mathcal{L}_{+}\eta_{1}}{\eta_{1}}
        + \bilin{\mathcal{L}_{-} \eta_2}{\eta_2} \right).
\end{align*}
This can complete the proof
of \eqref{coer:a} by the definition of $\widetilde{
\mathcal{H}}_{\omega, c}$.
\end{proof}

\subsection{Decomposition of the functions close to a traveling wave}\label{sect:decomp:asln}

Let $\omega^{0},$~$c^{0}$ satisfy $\sts{c^{0}}^{2} < 4\omega^{0}$. For $\alpha>0,$
we consider the following tube
of size $\alpha$ in $H^{1},$
\begin{align*}
  \mathcal{U}_1\sts{\alpha~,~\omega^{0}~,~c^{0}}\triangleq
  \set{
    u\in H^{1}(\R)~:~
    \inf_{(x, \gamma)\in\R^2}
    \normhone{
        u(\cdot)-\varphi_{\omega^{0}, c^{0}}\sts{\cdot-x}e^{\i\gamma}
    }
    <\alpha
  },
\end{align*}
which is close to some traveling wave with the subcritical parameter.
First, by the non-degenerate condition on $d''(\omega^0, c^0)$, we
have the following structure decomposition for the functions in the
above tube by the implicit function theorem.
\begin{lemma}
\label{lem:aslnst:decomp} There exists $\alpha_0>0, C_{I}>0$, such
that if $ u\in\mathcal{U}_1\sts{\delta,~\omega^{0},~c^{0}}$ with
$\delta<\alpha_0$, then there exist unique $\mathcal{C}^{1}$
functions $$\overrightarrow{q}_0\triangleq
\sts{~\omega_0~,~c_0~,~x_0~,~\gamma_0~} \in
\sts{0,~+\infty}\times\R^3$$ with $c_0^2 < 4\omega_0,$ such that
\begin{align}
&
        \Re\int
 R(x)\; \overline{\varepsilon(x)}
        \;\dx =0,
\quad
        \Re\int\sts{\i \partial_x R + \frac12
          \abs{R}^2R} (x)\; \overline{\varepsilon(x)}
        \;\dx =0, \label{orth:cons}
\\
 &   \Re\int \partial_x R(x)\; \overline{\varepsilon(x)}
        \;\dx =0,
 \quad
   \Re\int\i R(x)\;
           \overline{ \varepsilon(x)}
        \;\dx =0. \label{orth:syms}
\end{align}
where
\begin{align}
R(x)=& \; R\sts{\omega_0,c_0,x_0,\gamma_0 ; x}\triangleq
\varphi_{\omega_0,c_0}\sts{x-x_0}e^{\i\gamma_0}, \label{asln:stc}\\
     \varepsilon\sts{x}= &\; \varepsilon\sts{\omega_0,c_0,x_0,\gamma_0, u ; x}\triangleq  u\sts{x}- R\sts{\omega_0,c_0,x_0,\gamma_0 ;
     x}.
     \label{aslns:rem:def}
\end{align}
Moreover, we have
\begin{align*}
      \normhone{\varepsilon}+\abs{ \omega_0-\omega^{0} }+\abs{ c_0-c^{0} }
    &
        \leq C_{I}\;\delta.
\end{align*}
\end{lemma}

\begin{remark}
The orthogonal structures in \eqref{orth:cons} correspond to the
conservation laws of mass and momentum, while the orthogonal
structures in \eqref{orth:syms} correspond to the spatial translation invariance and the phase rotation
invariance .
\end{remark}
\begin{proof}[Proof of Lemma \ref{lem:aslnst:decomp}]

For any $u\in \mathcal{U}_1\sts{\delta, \omega^{0}, c^{0}},$ there
exists $(x^0, \gamma^0)\in\R^2$ such that
\begin{align*}
       \normhone{ u(\cdot)-\varphi_{\omega^{0},c^{0}}\sts{\cdot-x^0}e^{\i\gamma^0}}
    <\delta.
\end{align*}
Now let
\begin{equation*}\overrightarrow{q}^0=\sts{\omega^0,c^0,x^0,\gamma^0}, \; \text{and}\;
R^0(x)=\varphi_{\omega^0,c^0}\sts{x-x^0}e^{\i\gamma^0},
\end{equation*} and define
\begin{align*}
  \varrho_{1}\sts{\omega_0,c_0,x_0,\gamma_0, u}\triangleq &~
  \Re\int
    R\sts{\omega_0,c_0,x_0,\gamma_0 ; x}\;\overline{\varepsilon\sts{\omega_0,c_0,x_0,\gamma_0, u; x}}
  \;\dx,
  \\
  \varrho_{2}\sts{\omega_0,c_0,x_0,\gamma_0, u}\triangleq&~
  \Re\int\sts{\i \partial_x R + \frac12
          \abs{R}^2R} \sts{\omega_0,c_0,x_0,\gamma_0 ; x}\; \overline{\varepsilon\sts{\omega_0,c_0,x_0,\gamma_0, u; x}}
            \;\dx,
  \\
  \varrho_{3}\sts{\omega_0,c_0,x_0,\gamma_0, u}\triangleq&~
  \Re\int \partial_x R\sts{\omega_0,c_0,x_0,\gamma_0 ; x}\; \overline{\varepsilon\sts{\omega_0,c_0,x_0,\gamma_0, u; x}}
    \;\dx,
  \\
  \varrho_{4}\sts{\omega_0,c_0,x_0,\gamma_0, u}\triangleq&~
   \Re\int\i R\sts{\omega_0,c_0,x_0,\gamma_0 ; x}\;
           \overline{ \varepsilon\sts{\omega_0,c_0,x_0,\gamma_0, u; x}}
             \;\dx,
\end{align*}
where $R\sts{\omega_0,c_0,x_0,\gamma_0 ; x}$ and
$\varepsilon\sts{\omega_0,c_0,x_0,\gamma_0, u; x}$ are defined by
\eqref{asln:stc} and \eqref{aslns:rem:def}. By the direct
calculations, we have
$$\varepsilon\sts{\overrightarrow{q}^0, R^0; x}
 ~=~ 0;$$
 and
\begin{align*}
  \sts{\frac{\partial}{\partial\omega_0}\varepsilon}
   \sts{~\overrightarrow{q}^0, R^0;~x~}
    ~=~
        -\left.
            \frac{\partial
                }{
                    \partial\omega_0
                }
           R
         \right|_{\overrightarrow{q}_0=\overrightarrow{q}^0},
& \quad
 \sts{ \frac{\partial}{\partial c_0}\varepsilon}
    \sts{~\overrightarrow{q}^0, R^0;~x~}
    ~=~
    -\left.\frac{\partial}{\partial
    c_0}R\right|_{\overrightarrow{q}_0=\overrightarrow{q}^0},
  \\
  \sts{\frac{\partial}{\partial x_0}\varepsilon}
  \sts{~\overrightarrow{q}^0, R^0;~x~}
    ~=~
    \left.-\frac{\partial}{\partial
    x_0}R\right|_{\overrightarrow{q}_0=\overrightarrow{q}^0},
&\quad
 \sts{ \frac{\partial}{\partial\gamma_0}\varepsilon }
 \sts{~\overrightarrow{q}^0, R^0;~x~}
    ~=~
    \left.-\i R\right|_{\overrightarrow{q}_0=\overrightarrow{q}^0}.
\end{align*}
These imply that at the point $\sts{\overrightarrow{q}^0, R^0}$,
\begin{align*}
& \left.\frac{\partial\left(\varrho_{1}, \varrho_{2}, \varrho_{3},
\varrho_{4}\right)}{\partial\left(\omega_0, c_0, x_0,
\gamma_0\right)}\right|_{
(\overrightarrow{\mathbf{q}}_0,u)=(\overrightarrow{\mathbf{q}}^{0},R^0)
}
  \\
  = &
 \left. \begin{pmatrix}
        -\partial_{\omega_0}\mass{R}
    &
         -\partial_{c_0}\mass{R}
    &
         0
    &
         0
    \\
        -\partial_{\omega_0}\momentum{R}
    &
        -\partial_{c_0}\momentum{R}
    &
        0
    &
        0
    \\
        -\Re\int\partial_{x}R \; \overline{\partial_{\omega_0}R}
    &
        -\Re\int  \partial_{x}R\; \overline{\partial_{c_0}R}
    &
        - \Re\int  \partial_{x}R\; \overline{\partial_{x_0}R}
    &
        -\Re\int  \partial_{x}R\; \overline{\i R}
    \\
        -\Re\int  \i R \overline{\partial_{\omega_0}R }
    &
        -\Re\int  \i R \overline{\partial_{c_0}R }
    &
        -\Re\int  \i R \overline{\partial_{x_0}R }
    &
        -\Re\int  \i R \overline{\i R }
  \end{pmatrix} \right|_{ (\overrightarrow{\mathbf{q}}_0,u)=(\overrightarrow{\mathbf{q}}^{0},R^0) }
\end{align*}
By the simple calculations, the determinant of the above Jacobian is
$$-\normto{\partial_{x}\phi_{\omega^{0},c^{0}}}^{2}\cdot\normto{\phi_{\omega^0,c^0}}^{2}\cdot
\det d''\sts{~\omega^0,~c^0},$$ which is non-degenerate for
$\sts{c^0}^2 < 4 \omega^0$. Thus we can complete the proof by the
implicit function theorem.
\end{proof}


\section{Stability of the single traveling wave}\label{sect:stab:asln}
In this section, we shall give an alternative proof of the orbital
stability of the single traveling wave with $c^2 < 4\omega$ in the
energy space in \cite{ColinOhta-DNLS}, where the argument is based
on the concentration compactness principle. Now inspired by
Martel-Merle-Tsai's idea in \cite{MartelMT:Stab:NLS}, our argument
is the energy method together with the modulation analysis and
perturbation theory, which can be applied to the multi-traveling
wave case with the weak interactions.

Let $(\omega^{0}, c^{0})\in \R^2$ satisfy
$\sts{c^{0}}^{2}<4\omega^{0}$, $\alpha_0$ be determined by Lemma
\ref{lem:aslnst:decomp}, $A_0$, $\delta_0=\delta_0(A_0)$ be
determined later, and $\delta<\alpha_0$. Suppose that $u\sts{t}$ is
the solution of \eqref{DNLS} with the initial data
$u_{0}\in\mathcal{U}_1\sts{\delta, \omega^{0}, c^{0}}$, then by the
definition of the small tube $\mathcal{U}_1\sts{\delta,
 \omega^{0}, c^{0}}$,  there exist $x^{0}\in\R$ and $\gamma^{0}\in\R$
such that
\begin{align*}
    \normhone{ u_{0}(\cdot) - \varphi_{\omega^{0},c^{0}}\sts{\cdot-x^{0}}e^{\i\gamma^{0}}
    }<\delta.
\end{align*}
Let $A_0>2$ be determined later and define
\begin{align*}
    T^*=\sup\{\; t\geq 0, \; \sup_{\tau\in [0,t]}\inf_{x^0\in \R, \gamma^0\in\R}\normhone{ u(\tau,\cdot)-\varphi_{\omega^{0},c^{0}}\sts{\cdot-x^0}e^{\i\gamma^0}}
    \leq A_0\delta\; \}.
\end{align*}
By the continuity of $u(t)$ in $H^1$, we know that $T^*>0$. In order
to prove Theorem \ref{thm:stab:asln}, it suffices to show
$T^*=+\infty$ for some constants $\delta_0>0$ and $A_0>2$.

Assume that $T^*<+\infty$, we know that for any $t\in [0, T^*]$,
there exist $(x^0(t), \gamma^0(t))\in \R^2$ such that
\begin{align*}
      \normhone{ u(t,\cdot)-\varphi_{\omega^{0},c^{0}}\sts{\cdot-x^0(t)}e^{\i\gamma^0(t)}}  \leq A_0\delta.
\end{align*}
If necessary, we can choose $\delta_0$ sufficiently small to ensure
that the condition $ A_{0}\delta_0<\alpha_0$ holds, which enables us
to establish the structure decomposition to the solution $u(t)$,
$t\in [0, T^*]$.

\vskip0.1in \noindent \textbf{Step 1: The geometric decomposition of
solution $u(t)$ around the single traveling wave.} By Lemma
\ref{lem:aslnst:decomp}, we can modify the parameters $\omega^{0},
c^{0}, x^{0}(t)$ and $\gamma^{0}(t)$ such that $\left(c(t)\right)^2
< 4 \omega(t)$ for any $t\in [0, T^*]$,  and the remainder term
\begin{align}
\label{asln:rem:def}
    \varepsilon\sts{t, x} & \triangleq u\sts{t, x}-R\sts{t, x},\quad\text{where}\; R\sts{t, x}\triangleq  \varphi_{\omega\sts{t},c\sts{t}}\sts{x-x\sts{t}}e^{\i\gamma\sts{t}},
\end{align}
has the following orthogonal structures
\begin{align}
\label{asln:orth:nondeg}
        \Re\int R\sts{t}
            \overline{\varepsilon\sts{t}}
        \;\dx =0,
&\quad
        \Re\int
          \sts{\i \partial_x R\sts{t}+\frac12\abs{R(t)}^2 R \sts{t}}\;\overline{\varepsilon\sts{t}}
          \;\dx
        =0,
\\
\label{asln:orth:sym}
        \Re\int
        \partial_x R\sts{t}\overline{\varepsilon\sts{t}}
        \;\dx =0,
&\quad
        \Re\int
           \i\; R\sts{t} \overline{\varepsilon\sts{t}}
        \;\dx =0,
\end{align}
and for any $t\in [0, T^*]$, we have
\begin{align}
\label{alsn:data:small} &
        \normhone{\varepsilon\sts{0}}+\abs{ \omega\sts{0}-\omega^{0} }+\abs{ c\sts{0}-c^{0} }
        \leq C_I\delta,
\\
\label{asln:sl:rough} &
        \normhone{\varepsilon\sts{t}}+\abs{ \omega\sts{t}-\omega^{0} }+\abs{ c\sts{t}-c^{0} }
        \leq C_I A_{0}\delta.
\end{align}
If necessary, we can choose $\delta_0$ sufficiently small such that
$C_{I}A_{0}\delta_0<1$. Since \eqref{asln:sl:rough} is too rough, we
will combine the energy method with the coercivity property of the
linearized operator to show more refined estimates.

\vskip0.1in \noindent\textbf{Step 2: A conserved functional and
refined estimate of the remainder term
$\normhone{\varepsilon\sts{t}}$.} We now introduce the following
functional with parameters $\omega(0)$ and $c(0)$
\begin{align}
\label{acf:asln} \mathfrak{J}_{\omega(0), c(0)}\sts{u(t)} \triangleq
&\;
    E\sts{u(t)}+\omega\sts{0}\cdot\mass{u(t)}+c\sts{0}\cdot \momentum{u(t)},
\end{align}
which is conserved for the solution $u(t)$ of \eqref{DNLS} by the
conservation laws of mass, momentum and energy.

By the decomposition \eqref{asln:rem:def}, the orthogonal structures
\eqref{asln:orth:nondeg} and the estimate \eqref{asln:sl:rough}, we
have the following expansion formula.
\begin{lemma}
\label{lem:enlin:asln}
\begin{align}
\label{acf:lnen} \notag
        \mathfrak{J}_{\omega(0), c(0)}\sts{ u\sts{t} }
  =
      \mathfrak{J}_{\omega(0), c(0)}\sts{ R\sts{0} }
   & +   \mathcal{H}_{\omega\sts{t},c\sts{t}}\bilin{ \varepsilon\sts{t} }{ \varepsilon\sts{t} }
   +   \normhone{ \varepsilon\sts{t} }^{2}\beta\sts{ \normhone{ \varepsilon\sts{t} } }
  \\
  &
   +   \bigo{ | \omega\sts{t}-\omega\sts{0} |^{2} + | c\sts{t}-c\sts{0} |^{2} },
\end{align}
where $\beta\sts{r}\to 0$ as $r\to 0$, and
 \begin{align*}
\mathcal{H}_{\omega\sts{t},c\sts{t}}\bilin{ \varepsilon\sts{t} }{
\varepsilon\sts{t} }\triangleq &
        \int
        \frac{1}{2}\abs{\partial_{x}\varepsilon(t)}^{2}
       -   \frac{3}{32}\abs{R(t)}^{4}\abs{\varepsilon(t)}^{2}
         -      \frac38  \abs{R(t)}^{2}\sts{\Re\sts{ \overline{R(t)}\varepsilon(t)}}^{2}
            \;\dx \notag
\\
  &           +\int   \frac{1}{2}\omega\sts{t} \abs{\varepsilon(t)}^{2}
        -   \frac{1}{2}c\sts{t} \Im\sts{\overline{\varepsilon(t)}\partial_{x}\varepsilon(t)} \;\dx \notag
\\
&+
      \int  \frac{1}{4}c\sts{t} \abs{R(t)}^{2}\abs{\varepsilon(t)}^{2}
        \;\dx       +\frac12 c(t) \sts{\Re\sts{ \overline{R(t)}\varepsilon(t)}}^2
        \;\dx.
\end{align*}
\end{lemma}
\begin{proof}Firstly, by \eqref{asln:rem:def}, \eqref{asln:sl:rough} and the integration by parts, we have
\begin{align*}
  \frac12  \int \abs{\partial_{x}u\sts{t}}^{2}\;\dx
=&
  \frac12  \int \abs{\partial_{x}R\sts{t}}^2
    -   2\Re\sts{\overline{\partial^2_{x}R\sts{t}}\varepsilon\sts{t}}
    +   \abs{\partial_{x}\varepsilon\sts{t}}^2
    \;\dx,
\\
-\frac{1}{32}    \int \abs{u\sts{t}}^{6}\;\dx =&
  -\frac{1}{32}  \int \abs{R\sts{t}}^6
    +   6\abs{R\sts{t}}^{4}\Re\sts{\overline{R\sts{t}}\varepsilon\sts{t}}
  +
        3 \abs{R\sts{t}}^{4}\abs{\varepsilon\sts{t}}^{2}
    \;\dx
\\
& -\frac{1}{32} \int 12\abs{R\sts{t}}^{2}\sts{\Re\sts{
\overline{R\sts{t}}\varepsilon\sts{t}}}^2
    \;\dx
    +    \normhone{ \varepsilon\sts{t} }^{2}\beta\sts{ \normhone{ \varepsilon\sts{t} } }.
\end{align*}

Secondly, by the formula of mass $M(u(t))$, we have
\begin{align*}
 \frac12   \int \abs{u\sts{t}}^{2}\;\dx
&=
   \frac12 \int \abs{R\sts{t}}^2
    +   2\Re\sts{\overline{R\sts{t}}\varepsilon\sts{t}}
    +   \abs{\varepsilon\sts{t}}^2
    \;\dx.
\end{align*}

Last, by the formula of momentum $P(u(t))$, we get
\begin{align*}
 -\frac12   \Im\int\overline{u\sts{t}}\,\partial_x u\sts{t}\;\dx
=&
  -\frac12  \Im\int
        \overline{R\sts{t}}\partial_xR\sts{t}
    +   \overline{R\sts{t}}\partial_x \varepsilon\sts{t}
    +   \overline{\varepsilon\sts{t}}\partial_x R\sts{t}
    +   \overline{\varepsilon\sts{t}}\partial_x \varepsilon\sts{t}
    \;\dx
\\
=&
  -\frac12  \Im\int
        \overline{R\sts{t}}\partial_xR\sts{t}
       +   \overline{\varepsilon\sts{t}}\partial_x \varepsilon\sts{t}
    \;\dx  -\Re \int   \i \,\overline{\partial_xR\sts{t}} \varepsilon\sts{t}  \; \dx,
\\
\frac18    \int \abs{u\sts{t}}^{4}\;\dx =&\; \frac18    \int
\abs{R\sts{t}}^4
    +   4\abs{R(t)}^{2}\Re\sts{\overline{R(t)}\varepsilon(t)}
    +   2
            \abs{R\sts{t}}^{2}\abs{\varepsilon\sts{t}}^{2}
            \;\dx
\\
&
   + \frac{1}{8} \int 4\sts{\Re\sts{\overline{R\sts{t}}\varepsilon\sts{t}}}^2\;\dx+ \normhone{ \varepsilon\sts{t} }^{2}\beta\sts{ \normhone{ \varepsilon\sts{t} } }.
\end{align*}

Multiplying $M(u(t))$ with $\omega(0)$ and $P(u(t))$ with $c(0)$ and
summing up, we obtain
\begin{align}
 &     \mathfrak{J}_{\omega(0), c(0)}\sts{u(t)}
\notag
\\
= & \;
    \mathfrak{J}_{\omega(0), c(0)}\sts{R\sts{t}} \notag
    \\
 &  + \Re \int \sts{ -\overline{\partial^2_x R} -\frac{3}{16}\abs{R }^4\overline{R }
 +\omega(0)\overline{R } - \i c(0) \overline{\partial_x R }
 +\frac12 c(0)\abs{ R }^2 \overline{ R }  }\sts{t} \varepsilon (t)\; \dx \notag
 \\&
        + \int
        \frac{1}{2}\abs{\partial_{x}\varepsilon(t)}^{2}
       -   \frac{3}{32}\abs{R(t)}^{4}\abs{\varepsilon(t)}^{2}R\sts{t}
            \;\dx
       -      \frac38  \abs{R(t)}^{2}\sts{\Re\sts{ \overline{R(t)}\varepsilon(t)}}^{2}
            \;\dx\notag
\\
  &           +\int   \frac{1}{2}\omega\sts{0} \abs{\varepsilon(t)}^{2}
        -   \frac{1}{2}c\sts{0} \Im\sts{\overline{\varepsilon(t)}\partial_{x}\varepsilon(t)} \;\dx \notag
\\
&+
      \int  \frac{1}{4}c\sts{0} \abs{R(t)}^{2}\abs{\varepsilon(t)}^{2}
        \;\dx       +\frac12 c(0) \sts{\Re \overline{R(t)}\varepsilon(t)}^2
        \;\dx \notag
 \\
        &    +   \normhone{ \varepsilon\sts{t} }^{2}\beta\sts{ \normhone{ \varepsilon\sts{t} } }. \label{acf:asln:exp}
\end{align}

We first deal with the linear term in $\varepsilon(t)$. Since
$R\sts{t, x} =
\varphi_{\omega\sts{t},c\sts{t}}\sts{x-y\sts{t}}e^{\i\gamma\sts{t}}
$ satisfies that the following elliptic equation
\begin{align}\label{eq:tw:tvar}
 -\overline{\partial^2_x R\sts{t}} -\frac{3}{16}\abs{R\sts{t}}^4\overline{R\sts{t}} +\omega(t)\overline{R(t)}
 - \i c(t) \overline{\partial_x R\sts{t}}  +\frac12 c(t)\abs{ R\sts{t}}^2 \overline{ R\sts{t}}=0,
\end{align}
we have the following cancelation  by the orthogonal structure
\eqref{asln:orth:nondeg}
\begin{align}
&\Re \int \sts{ -\overline{\partial^2_x R\sts{t}}
-\frac{3}{16}\abs{R\sts{t}}^4\overline{R\sts{t}}
+\omega(0)\overline{R(t)} - \i c(0) \overline{\partial_x R\sts{t}}
+\frac12 c(0)\abs{ R\sts{t}}^2 \overline{ R\sts{t}}}
\varepsilon(t)\;\dx \notag
\\
=& \sts{\omega(0)-\omega(t)} \Re \int \overline{R(t)} \varepsilon(t)
\;\dx + \sts{c(0)-c(t)} \Re \int \sts{ - \i \overline{\partial_x
R\sts{t}} +\frac12  \abs{ R\sts{t}}^2 \overline{ R\sts{t}}
}\varepsilon(t) \; \dx\notag
\\
=& \sts{\omega(0)-\omega(t)} \Re \int  R(t)\overline{
\varepsilon(t)} \;\dx + \sts{c(0)-c(t)} \Re \int \sts{  \i
\partial_x R\sts{t} +\frac12  \abs{ R\sts{t}}^2  R\sts{t}
}\overline{\varepsilon(t)} \; \dx \notag
\\
=&\; 0. \label{acf:asln:1st}
\end{align}

Secondly, by the definition of the linearized energy
$\mathcal{H}_{\omega\sts{t},c\sts{t}}\bilin{ \varepsilon\sts{t} }{
\varepsilon\sts{t} }$,  we have
\begin{align}
 & \int
        \frac{1}{2}\abs{\partial_{x}\varepsilon(t)}^{2}
       -   \frac{3}{32}\abs{R(t)}^{4}\abs{\varepsilon(t)}^{2}R\sts{t}
            \;\dx
       -      \frac38  \abs{R(t)}^{2}\sts{\Re\sts{ \overline{R(t)}\varepsilon(t)}}^{2}
            \;\dx\notag
\\
  &           +\int   \frac{1}{2}\omega\sts{0} \abs{\varepsilon(t)}^{2}
        -   \frac{1}{2}c\sts{0} \Im\sts{\overline{\varepsilon(t)}\partial_{x}\varepsilon(t)} \;\dx \notag
\\
&+
      \int  \frac{1}{4}c\sts{0} \abs{R(t)}^{2}\abs{\varepsilon(t)}^{2}
        \;\dx       +\frac12 c(0) \sts{\Re \left(\overline{R(t)}\varepsilon(t)\right)}^2
        \;\dx  \notag
        \\
=& \; \mathcal{H}_{\omega\sts{t},c\sts{t}}\bilin{ \varepsilon\sts{t}
}{ \varepsilon\sts{t} } +
 C \normhone{\varepsilon(t)}^2 \big(    \abs{ \omega\sts{t}-\omega\sts{0} }
        +   \abs{ c\sts{t}-c\sts{0} } \big)  \notag
\\
= & \; \mathcal{H}_{\omega\sts{t},c\sts{t}}\bilin{
\varepsilon\sts{t} }{ \varepsilon\sts{t} } +
   O\left( \abs{ \omega\sts{t}-\omega\sts{0} }^2
        +   \abs{ c\sts{t}-c\sts{0} }^2\right) \notag \\& \; +  \normhone{\varepsilon(t)}^2 \beta\sts{ \normhone{ \varepsilon\sts{t} } }.
        \label{acf:asln:2nd}
\end{align}

Finally,  we estimate the main term by Taylor's expansion.
\begin{align}
&\mathfrak{J}_{\omega(0), c(0)}\sts{R\sts{t}}  \notag
\\
= & \mathfrak{J}_{\omega(t), c(t)}\sts{R\sts{t}} +
\sts{\omega(0)-\omega(t)} M(R(t)) + \sts{c(0)-c(t)} P(R(t))\notag
\\
= & J_{\omega(t), c(t)}\sts{\varphi_{\omega(t), c(t)}} +
\sts{\omega(0)-\omega(t)} M(\varphi_{\omega(t), c(t)}(t)) +
\sts{c(0)-c(t)} P(\varphi_{\omega(t), c(t)}(t))\notag
\\
=& J_{\omega(0), c(0)}\sts{\varphi_{\omega(0), c(0)}}
 + \sts{\omega(t)-\omega(0)} M(\varphi_{\omega(0), c(0)}) +  \sts{c(t)-c(0)} P(\varphi_{\omega(0), c(0)}) \notag
  \\
  & + O\sts{ |\omega\sts{t}-\omega\sts{0} |^2
        +  \abs{ c\sts{t}-c\sts{0} }^2}  + \sts{\omega(0)-\omega(t)} M(\varphi_{\omega(t), c(t)}) + \sts{c(0)-c(t)} P(\varphi_{\omega(t), c(t)}) \notag
        \\
=& J_{\omega(0), c(0)}\sts{\varphi_{\omega(0), c(0)}} + O\sts{
|\omega\sts{t}-\omega\sts{0} |^2
        +  \abs{ c\sts{t}-c\sts{0} }^2}\notag
        \\
=& \mathfrak{J}_{\omega(0), c(0)}\sts{R(0)} + O\sts{
|\omega\sts{t}-\omega\sts{0} |^2
        +  \abs{ c\sts{t}-c\sts{0} }^2}    \label{est:main}
\end{align}
where we used the fact that
\begin{align*}
\frac{\partial}{\partial \omega}J_{\omega, c} \sts{\varphi_{\omega,
c}} = M (\varphi_{\omega, c}), \quad \frac{\partial}{\partial
c}J_{\omega, c} \sts{\varphi_{\omega, c}} = P (\varphi_{\omega, c})
\end{align*}
in the third equality. Inserting \eqref{acf:asln:1st},
\eqref{acf:asln:2nd} and \eqref{est:main} into \eqref{acf:asln:exp},
we can complete the proof.
\end{proof}

By the conservation laws of mass, momentum and energy,  we have
\begin{align*}
 \mathfrak{J}_{\omega(0), c(0)}\sts{ u\sts{t} } = \mathfrak{J}_{\omega(0), c(0)}\sts{ u\sts{0} }.
\end{align*}
This together with Lemma \ref{lem:enlin:asln} and
\eqref{alsn:data:small} implies that
\begin{align} \label{acf:asln:quadest}
&\; \mathcal{H}_{\omega\sts{t},c\sts{t}}\bilin{ \varepsilon\sts{t} }{ \varepsilon\sts{t} }   \\
=&\; \mathcal{H}_{\omega\sts{0},c\sts{0}}\bilin{ \varepsilon\sts{0}
}{ \varepsilon\sts{0} }
    +   \normhone{ \varepsilon\sts{0} }^{2}\beta\sts{ \normhone{ \varepsilon\sts{0} } }+  \normhone{ \varepsilon\sts{t} }^{2}\beta\sts{ \normhone{ \varepsilon\sts{t} } } \notag
    \\
  & +   \bigo{ | \omega\sts{t}-\omega\sts{0} |^{2} + | c\sts{t}-c\sts{0} |^{2} } \notag
\\
    \leq & \;    C    \normhone{ \varepsilon\sts{0} }^2  +
     \normhone{ \varepsilon\sts{t} }^{2}\beta\sts{ \normhone{ \varepsilon\sts{t} } }
     +   \bigo{ | \omega\sts{t}-\omega\sts{0} |^{2} + | c\sts{t}-c\sts{0} |^{2} }.\notag
\end{align}

As for the linearized energy
$\mathcal{H}_{\omega\sts{t},c\sts{t}}\bilin{ \varepsilon\sts{t} }{
\varepsilon\sts{t} } $ with $c(t)^2< 4 \omega(t)$, we have the
following coercivity property under the orthogonal conditions
\eqref{asln:orth:nondeg} and \eqref{asln:orth:sym}.

\begin{lemma}\label{lem:asln:quadcoer} Suppose that $c(t)^2< 4 \omega(t)$ for any $t\in [0, T^*]$, and $\varepsilon(t)\in H^{1}(\R)$ satisfies
the orthogonal conditions  \eqref{asln:orth:nondeg} and
\eqref{asln:orth:sym}, then there exists a constant $C_0>0$ such
that
  \begin{align*}
  C_0  \normhone{\varepsilon(t)}^2 \leq \mathcal{H}_{\omega\sts{t},c\sts{t}}\bilin{ \varepsilon\sts{t} }{ \varepsilon\sts{t} }.
  \end{align*}
\end{lemma}
\begin{proof} Since the traveling wave has the following structure
\begin{align*}
R\sts{t, x} = &
\varphi_{\omega\sts{t},c\sts{t}}\sts{x-x\sts{t}}e^{\i\gamma\sts{t}}
=   \phi_{\omega\sts{t},c\sts{t}}\sts{x-x\sts{t}}e^{\i\gamma\sts{t}}
e^{i\frac12 c(t)\sts{x-x(t)}},
\end{align*}
we introduce the similar structure for the remainder term
$\varepsilon(t)$ and define $\eta(t)$ as following
\begin{align*}
\eta\sts{t, x}= \varepsilon(t,x+x(t))e^{-\i\gamma\sts{t}}
e^{-i\frac12 c(t)x}\;\Longleftrightarrow\;
\varepsilon(t,x)=\eta\sts{t, x-x\sts{t}}e^{\i\gamma\sts{t}}
e^{i\frac12 c(t)\sts{x-x(t)}}.
\end{align*}
By the simple computations, the linearized energy
$\mathcal{H}_{\omega\sts{t},c\sts{t}}\bilin{ \varepsilon\sts{t} }{
\varepsilon\sts{t} } $ is
\begin{align*}
& \mathcal{H}_{\omega\sts{t},c\sts{t}}\bilin{ \varepsilon\sts{t} }{
\varepsilon\sts{t} }
\\
=&  \int\left(
    \frac12\abs{\partial_x \eta}^2+\frac12\sts{\omega(t)-\frac14 c(t)^2}\abs{\eta}^2
    -\frac{3}{32}\phi^4_{\omega(t),c(t)}\abs{\eta}^2 -\frac38 \phi^4_{\omega(t),c(t)}\abs{\eta_1(t)}^2 \right)\, \dx\\
    &
    \qquad \qquad + \int \left(\frac14 c(t)  \phi^2_{\omega(t),c(t)}\abs{\eta(t)}^2 +\frac12 c(t)
    \phi^2_{\omega(t),c(t)}\abs{\eta_1(t)}^2 \right) \,\dx
    \\
 =&   \bilin{\mathcal{L}_{+} \eta_1(t)}{\eta_1(t)}+ \bilin{\mathcal{L}_{-} \eta_2(t)}{\eta_2(t)}
    =\widetilde{   \mathcal{H}}_{\omega(t),c(t)} \bilin{\eta(t)}{\eta(t)},
    \end{align*}
the orthogonal conditions  \eqref{asln:orth:nondeg} and
\eqref{asln:orth:sym} on $\varepsilon(t)$ are equivalent to the
following conditions on $\eta(t)=\eta_1(t)+\i \,\eta_2(t)$:
 $$\bilin{\eta_2(t)}{\phi_{\omega(t),c(t)}}             = 0,$$
 and $\bilin{ \eta_1(t) }{\phi_{\omega(t),c(t)}}= 0$,  $\bilin{\eta_1(t)}{\partial_x\phi_{\omega(t),c(t)}}= 0,$
 \begin{align*}
     \frac{1}{2}\bilin{ \eta_1(t)}{\phi_{\omega(t),c(t)}^3}
        +\bilin{ \eta_2(t)}{\partial_x\phi_{\omega(t),c(t)}}        = 0.
  \end{align*}
By Proposition \ref{prop:coer:asln}, we have
\begin{align*} \mathcal{H}_{\omega\sts{t},c\sts{t}}\bilin{ \varepsilon\sts{t} }{ \varepsilon\sts{t} } =
\widetilde{   \mathcal{H}}_{\omega(t),c(t)} \bilin{\eta(t)}{\eta(t)}
\geq C \normhone{\eta(t)}^2,
\end{align*}
which together \eqref{asln:sl:rough} and  the fact that
\begin{align*}
 \normhone{\varepsilon(t)}^{2}
    \leq C   \normto{\eta_{x}(t)}^{2}+C\sts{1+c(t)^2}\normto{\eta(t)}^{2}
\end{align*}
implies the result.
\end{proof}

By Lemma \ref{lem:asln:quadcoer} and \eqref{acf:asln:quadest}, there
exists some constant $C>0$ such that for any $t\in [0, T^*]$, we
have
\begin{align}\label{asln:rem:refine}
\normhone{ \varepsilon\sts{t} }^{2}
  \leq
       C \normhone{ \varepsilon\sts{0} }^{2}
    +   C\sts{\abs{ \omega\sts{t}-\omega\sts{0} }^{2}
    +   \abs{ c\sts{t}-c\sts{0} }^{2}}.
\end{align}
This completes the refined estimate of the remainder term
$\normhone{\varepsilon\sts{t}}$.

\vskip 0.1in \noindent \textbf{Step 3: Refined estimate of
$|\omega(t)-\omega(0)|+|c(t)-c(0)|$.}  By \eqref{asln:rem:def},
\eqref{asln:orth:nondeg}, and \eqref{asln:sl:rough}, we have for any
$t\in [0, T^*]$
\begin{align*}
\begin{pmatrix}
 M(u(t)) \\
P(u(t)) \\
\end{pmatrix}
=& \begin{pmatrix}
 M(\varphi_{\omega\sts{t},c\sts{t}}) \\
P(\varphi_{\omega\sts{t},c\sts{t}}) \\
\end{pmatrix} + O \sts{\normhone{ \varepsilon\sts{t} }^{2}}
\\
=& \begin{pmatrix}
 M(\varphi_{\omega\sts{0},c\sts{0}}) \\
P(\varphi_{\omega\sts{0},c\sts{0}}) \\
\end{pmatrix} +d''\sts{\omega\sts{0}, c\sts{0}}\begin{pmatrix}
                             \omega\sts{t}-\omega\sts{0} \\
                             c\sts{t}-c\sts{0} \\
                           \end{pmatrix}+ O \sts{\normhone{ \varepsilon\sts{t} }^{2}}
\\
& \quad + \sts{\abs{ \omega\sts{t}-\omega\sts{0} }
    +   \abs{ c\sts{t}-c\sts{0} }} \beta\sts{\abs{ \omega\sts{t}-\omega\sts{0} }
    +   \abs{ c\sts{t}-c\sts{0} }},
\end{align*}
which, together with \eqref{alsn:data:small}, the mass and momentum
conservation laws and the non-degenerate conditions $\det
d''\sts{\omega^{0}, c^{0}}<0$, implies that for sufficient small
$\delta$
\begin{align}
\label{asln:para:refine}
  \abs{\omega\sts{t}-\omega\sts{0}} + \abs{c\sts{t}-c\sts{0}}\leq C
    \sts{
        \normhone{\varepsilon\sts{t}}^{2}+\normhone{\varepsilon\sts{0}}^{2}
    }.
\end{align}
By \eqref{asln:sl:rough}, \eqref{asln:rem:refine} and
\eqref{asln:para:refine}, we have
\begin{align}\label{asln:dyn}
\normhone{ \varepsilon\sts{t} }  + \abs{\omega\sts{t}-\omega\sts{0}}
+ \abs{c\sts{t}-c\sts{0}}
  \leq
       C \normhone{ \varepsilon\sts{0} }.
\end{align}

\vskip0.1in \noindent\textbf{Step 4: Conclusion.} Now by
\eqref{alsn:data:small}, and \eqref{asln:dyn}, we have for any $t\in
[0, T^*]$
\begin{align*}
 & \inf_{(x^0, \gamma^0)\in\R^2} \normhone{ u\sts{t,\cdot}-\varphi_{\omega^{0},c^{0}}\sts{ \cdot-x^0}e^{\i\gamma^0} }
  \\
  \leq &
  \normhone{ u\sts{t,\cdot}-\varphi_{\omega^{0},c^{0}}\sts{ \cdot-x\sts{t} }e^{\i\gamma\sts{t}} }
  \\
  \leq
  &\;
    \normhone{ u\sts{t, \cdot}-\varphi_{\omega\sts{t},c\sts{t}}\sts{ \cdot-x\sts{t} }e^{\i\gamma\sts{t}} }
    +
    C \sts{\abs{\omega(t)-\omega^0}+\abs{c(t)-c^0}}
  \\
  \leq
  &\;
    \normhone{\varepsilon\sts{t}}
    +
    C\sts{\abs{\omega\sts{t}-\omega\sts{0}} + \abs{c\sts{t}-c\sts{0}}+\abs{\omega\sts{0}-\omega^{0}}+ \abs{c\sts{0}-c^{0}}}
    \\
 \leq   & \;
  C\sts{\normhone{\varepsilon\sts{0}}
    +\abs{\omega\sts{0}-\omega^{0}}+ \abs{c\sts{0}-c^{0}}}
   \\
 \leq   & \;   C\; C_I \; \delta.
\end{align*}

If choosing $A_0\geq 2C C_I$, then for any $t\in [0, T^*]$, we have
\begin{align*}
 \inf_{(x^0, \gamma^0)\in\R^2} \normhone{ u\sts{t,\cdot}-\varphi_{\omega^{0},c^{0}}\sts{ \cdot-x^0}e^{\i\gamma^0} }
   \leq
  \frac12 A_0 \delta,
  \end{align*}
which contradicts with the assumption $T^*<+\infty$ by the
continuity of $u(t)$ in $H^1(\R)$. This implies $T^*=+\infty$ and
completes the proof of Theorem \ref{thm:stab:asln}.


\section{Decomposition of the solution around the sum of two traveling
waves}\label{sect:decomp:tsln}

From now on, we will consider the stability of the solution around
the sum of two traveling waves for \eqref{DNLS} in the energy space
as that of gKdV and NLS equations in \cite{MartelMT:Stab:gKdV,
MartelMT:Stab:NLS}, we shall extend the structure decomposition of
the functions (or solutions) around the single traveling wave as in
Lemma \ref{lem:aslnst:decomp} to the corresponding decomposition. It
is noticed that the interaction between two traveling waves is weak
since they locate far away from each other.

First, we introduce some notations. Let $(\omega_{k}^{0},
c_{k}^{0})\in \R^2$ be such that $0<c^0_1<c^0_2$ and $\sts{c_{k}^{0}
}^{2}<4\omega_{k}^{0}$, $k=1, 2$. Let $\alpha < \alpha_0$ be small
enough, and $L>L_0$ be large enough, where $\alpha_0$, $L_0$ will be
determined later. Now we define the following $H^{1}$-tube which is
close to the sum of two traveling waves with weak interaction,
 \begin{multline*}
   \mathcal{U}\sts{\alpha~,~L~,~\omega_{1}^{0}~,~c_{1}^{0}~,~\omega_{2}^{0}~,~c_{2}^{0}~}\\
  \triangleq
  \set{
    u\in H^{1}(\R)~:~
    \inf_{\substack{x_{2}-x_{1}> L\\\gamma_{1},\gamma_{2}\in\R}}
    \normhone{
        u   -
                \sum^2_{k=1}\varphi_{\omega_{k}^{0},c_{k}^{0}}\sts{\cdot-x_{k}}\,e^{\i\gamma_{k}}
    }
    <\alpha
  }.
 \end{multline*}

Next, a similar argument, but more delicate, as we adopted in the
proof of Lemma \ref{lem:aslnst:decomp} gives us the useful
geometrical decomposition of the functions in the above $H^1$-tube.
More precisely, we have
\begin{lemma}
\label{lem:tslnst:decomp}
  There exist constants $\alpha_0>0$ sufficient small, $\widetilde{L}_{0}$ large enough and $C_{II}>0$,
   such that if $
u\in\mathcal{U}\sts{\alpha,~L,~\omega_{1}^{0},~c_{1}^{0},~\omega_{2}^{0},~c_{2}^{0}}$
with $\alpha<\alpha_0,$  and $L>\widetilde{L}_{0}$, then there exist
unique $\mathcal{C}^{1}$ functions $$
    \overrightarrow{\mathbf{q}}\triangleq\,
    \sts{ ~\omega_{1}~,~c_{1}~,~x_{1}~,~\gamma_{1}~,~\omega_{2}~,~c_{2}~,~x_{2}~,~\gamma_{2}~} \in
\sts{0,~+\infty}\times\R^3\times\sts{0, +\infty}\times\R^3 $$ with
$\left(c_{k}\right)^2 < 4\omega_{k}$, $k=1,2$, such that
\begin{align}
        \Re\int
 R_{k}(x)\; \overline{\varepsilon(x)}
        \;\dx =0, &
\quad
        \Re\int\sts{\i \partial_x R_{k} + \frac12
          \abs{R_{k}}^2R_{k}} (x)\; \overline{\varepsilon(x)}
        \;\dx =0, \label{tslns:orth:nondeg}
\\
       \Re\int \partial_x R_{k}(x)\; \overline{\varepsilon(x)}
        \;\dx =0, &
 \quad
   \Re\int\i R_{k}(x)\;
           \overline{ \varepsilon(x)}
        \;\dx =0, \label{tslns:orth:sym}
\end{align}
where $k=1,2$ and
\begin{align}
    R_{k}\sts{x} = & \;R\sts{ ~\omega_{k},~c_{k},~x_{k},~\gamma_{k};~x} \triangleq
    \varphi_{\omega_{k},c_{k}}\sts{x-x_{k}}e^{\i\gamma_{k}}, \label{tsln:stc}
    \\
    \varepsilon
    \sts{x}
    = &\;
    \varepsilon
    \sts{\overrightarrow{\mathbf{q}},~u;~x}\triangleq u\sts{x} -
    \sum^2_{k=1}R_{k}\sts{x}. \label{tslnst:rem:def}
\end{align}
 Moreover, we have
\begin{align}\label{est:prior}
      \normhone{\varepsilon}+\sum^2_{k=1}\left( |\omega_k-\omega_{k}^{0}|+| c_k-c_{k}^{0} |\right)
    &
        \leq C_{II}\;\alpha.
\end{align}
\end{lemma}
\begin{remark}
\begin{enumerate}
  \item For $k=1,2$, the traveling wave $\varphi_{\omega_{k},c_{k}}(x-x_k)e^{i\gamma_k}$
   of \eqref{DNLS} makes sense only for $\left(c_{k}\right)^{2}< 4\omega_{k} $ or $\left(c_{k}\right)^{2}= 4\omega_{k}$ with $c_k>0$ in
   \cite{MiaoTX-2015}, thus the parameter $\alpha_0$ is dependent of the parameters $(\omega_{1}^{0},~c_{1}^{0},~\omega_{2}^{0},~c_{2}^{0})$.
  \item The condition ``$\widetilde{L}_0$ large enough'' is essential in the stability theory of the sum of multi traveling waves
  in the energy space. It makes sure that the interaction between all traveling waves is weak, hence the implicit function theorem and perturbation
  theory can be applied.
\end{enumerate}
\end{remark}
\begin{proof}[Proof of Lemma \ref{lem:tslnst:decomp}]
First, by the definition of the $H^1$-tube, we know that for any
$u\in
\mathcal{U}\sts{\alpha~,~L~,~\omega_{1}^{0}~,~c_{1}^{0}~,~\omega_{2}^{0}~,~c_{2}^{0}~},$
there exist $x_{1}^{0},$ $\gamma_{1}^{0},$ $x_{2}^{0}$ and
$\gamma_{2}^{0}$ with $x^0_2-x^0_1>L$ such that
\begin{align}
\label{eq:2h1neighborhood}
    \normhone{
        u
        -
        R^{0}
    }<\alpha,
\end{align}
where
$
    \displaystyle R^{0}\sts{x}\triangleq \sum^2_{k=1} R_{k}^{0}\sts{x}$, $
     R_{k}^{0}\sts{x} \triangleq R\sts{
     ~\omega_{k}^{0},~c_{k}^{0},~x_{k}^{0},~\gamma_{k}^{0};~x}.
$
Denote  the open $H^1$-ball of center $R^{0}\sts{x}$ and of radius $\alpha$
by $B\left( R^{0} , \alpha\right)$, then $u\in B\left( R^{0} , \alpha\right)$.

Now let
\begin{align*}
    \overrightarrow{\mathbf{q}}^{0}
    \triangleq
    \sts{
        ~\omega_{1}^{0},~c_{1}^{0},~x_{1}^{0},~\gamma_{1}^{0},~\omega_{2}^{0},
        ~c_{2}^{0},~x_{2}^{0},~\gamma_{2}^{0}},
\end{align*}
and define for $k=1,~2$
\begin{align*}
&
    \varrho_{1}^{k}\sts{~\overrightarrow{\mathbf{q}}~,~u~}
    ~=~
    \Re\int
        R_{k}(x)\; \overline{\varepsilon\sts{~\overrightarrow{\mathbf{q}}~,~u~; ~x~}}
    \;\dx~,
\\
&
    \varrho_{2}^{k}\sts{~\overrightarrow{\mathbf{q}}~,~u~}
    ~=~
    \Re\int\sts{\i \partial_x R_{k} + \frac12
        \abs{R_{k}}^2R_{k}} (x)\; \overline{\varepsilon\sts{~\overrightarrow{\mathbf{q}}~,~u~; ~x~}}
    \;\dx~,
\\
&
    \varrho_{3}^{k}\sts{~\overrightarrow{\mathbf{q}}~,~u~}
    ~=~
    \Re\int
        \partial_x R_{k}(x)\; \overline{\varepsilon\sts{~\overrightarrow{\mathbf{q}}~,~u~; ~x~}}
    \;\dx~,
\\
&
    \varrho_{4}^{k}\sts{~\overrightarrow{\mathbf{q}}~,~u~}
    ~=~
    \Re\int
        \i R_{k}(x)\;\overline{ \varepsilon\sts{~\overrightarrow{\mathbf{q}}~,~u~; ~x~}}
    \;\dx~,
\end{align*}
where $R_{k}(x)$ and
$\varepsilon\sts{~\overrightarrow{\mathbf{q}}~,~u~; ~x~}$ are
defined by \eqref{tsln:stc} and \eqref{tslnst:rem:def}. It is easy
to see that
\begin{align}
\label{eq:2epsilon0}
    \varepsilon
    \sts{\overrightarrow{\mathbf{q}}^{0},~R^{0};
            ~x
    }\equiv 0.
\end{align}

By \eqref{eq:2epsilon0} and the fact that
\begin{align*}
    \sts{\frac{\partial}{\partial\omega_{k}}\varepsilon}
    \sts{
        ~\overrightarrow{\mathbf{q}}^{0}~,~R^{0}~;
            ~x~
    }
    ~=~&
        -\left.
            \frac{\partial
                }{
                    \partial\omega_{k}
                }
           R_{k}
         \right|_{ \overrightarrow{\mathbf{q}}=\overrightarrow{\mathbf{q}}^{0}
         }~,
\\
    \sts{\frac{\partial}{\partial c_{k}}\varepsilon}
    \sts{
        ~\overrightarrow{\mathbf{q}}^{0}~,~R^{0}~;
            ~x~
    }
    ~=~&
        -\left.
            \frac{\partial
                }{
                    \partial c_{k}
                }
           R_{k}
         \right|_{ \overrightarrow{\mathbf{q}}=\overrightarrow{\mathbf{q}}^{0}
         }~,
\\
  \sts{\frac{\partial}{\partial x_k}\varepsilon}
    \sts{
        ~\overrightarrow{\mathbf{q}}^{0}~,~R^{0}~;
            ~x~
    }
    ~=~&
    -\left.
        \frac{\partial}{\partial x_{k}}R_{k}
    \right|_{ \overrightarrow{\mathbf{q}}=\overrightarrow{\mathbf{q}}^{0}
    }~,
\\
 \sts{ \frac{\partial}{\partial\gamma_{k}}\varepsilon }
    \sts{
        ~\overrightarrow{\mathbf{q}}^{0}~,~R^{0}~;
            ~x~
    }
    ~=~&
    -
    \left.
    \i~R_{k}
    \right|_{ \overrightarrow{\mathbf{q}}=\overrightarrow{\mathbf{q}}^{0} }~,
\end{align*}
we have for $k, k'=1,~2$
\begin{align*}
    \frac{\partial\varrho_{1}^{k}
        }{
                \partial \omega_{k'}
    }
    \sts{~\overrightarrow{\mathbf{q}}^{0}~,~R^0~}
    &
    ~=~
    -
    \Re\int
        R^0_{k}(x)\;
       \overline{ \left.
            \frac{\partial}{\partial \omega_{k'}}R_{k'}
        \right|_{ \overrightarrow{\mathbf{q}}=\overrightarrow{\mathbf{q}}^{0} }(x)}
    \;\dx~,
\\
    \frac{\partial\varrho_{1}^{k}
        }{
                \partial c_{k'}
    }
    \sts{~\overrightarrow{\mathbf{q}}^{0}~,~R^0~}
    &
    ~=~
    -
    \Re\int
        R^0_{k}(x)\;
       \overline{ \left.
            \frac{\partial
            }{
                \partial c_{k'}
            }
            R_{k'}
        \right|_{ \overrightarrow{\mathbf{q}}=\overrightarrow{\mathbf{q}}^{0} }(x)  }              \;\dx~,
\\
    \frac{\partial\varrho_{1}^{k}
        }{
                \partial x_{k'}
    }
    \sts{~\overrightarrow{\mathbf{q}}^{0}~,~R^0~}
    &
    ~=~
    -
    \Re\int
        R^0_{k}(x)\;
        \overline{\left.
            \frac{\partial
            }{
                \partial x_{k'}
            }
            R_{k'}
        \right|_{ \overrightarrow{\mathbf{q}}=\overrightarrow{\mathbf{q}}^{0} }(x)  }              \;\dx~,
\\
    \frac{\partial\varrho_{1}^{k}
        }{
                \partial \gamma_{k'}
    }
    \sts{~\overrightarrow{\mathbf{q}}^{0}~,~R^0~}
    &
    ~=~
    \Re\int
        R^0_{k}(x)\;
   \i~ \overline{\left.
R_{k'}
    \right|_{ \overrightarrow{\mathbf{q}}=\overrightarrow{\mathbf{q}}^{0} }(x)   }             \;\dx~,
\end{align*} and
\begin{align*}
    \frac{\partial\varrho_{2}^{k}
        }{
                \partial \omega_{k'}
    }\sts{~\overrightarrow{\mathbf{q}}^{0}~,~R^0~}
    &
    ~=~
    -
    \Re\int\sts{\i~\partial_x R^0_{k} + \frac12
        \abs{R^0_{k}}^2R^0_{k}} (x)\;
        \overline{\left.
            \frac{\partial}{\partial \omega_{k'}}R_{k'}
        \right|_{ \overrightarrow{\mathbf{q}}=\overrightarrow{\mathbf{q}}^{0} }(x)}
    \;\dx~,
\\
    \frac{\partial\varrho_{2}^{k}
        }{
                \partial c_{k'}
    }\sts{~\overrightarrow{\mathbf{q}}^{0}~,~R^0~}
    &
    ~=~
    -
    \Re\int\sts{\i~\partial_x R^0_{k} + \frac12
        \abs{R^0_{k}}^2R^0_{k}} (x)\;
        \overline{\left.
            \frac{\partial
            }{
                \partial c_{k'}
            }
        R_{k'}
        \right|_{ \overrightarrow{\mathbf{q}}=\overrightarrow{\mathbf{q}}^{0} }(x)}
    \;\dx~,
\\
    \frac{\partial\varrho_{2}^{k}
        }{
                \partial x_{k'}
    }\sts{~\overrightarrow{\mathbf{q}}^{0}~,~R^0~}
    &
    ~=~
    -
    \Re\int\sts{\i~\partial_x R^0_{k} + \frac12
        \abs{R^0_{k}}^2R^0_{k}} (x)\;
        \overline{\left.
            \frac{\partial
            }{
                \partial x_{k'}
            }
        R_{k'}
        \right|_{ \overrightarrow{\mathbf{q}}=\overrightarrow{\mathbf{q}}^{0} }(x)}
    \;\dx~,
\\
    \frac{\partial\varrho_{2}^{k}
        }{
                \partial \gamma_{k'}
    }\sts{~\overrightarrow{\mathbf{q}}^{0}~,~R^0~}
    &
    ~=~
    \Re\int\sts{\i~\partial_x R^0_{k} + \frac12
        \abs{R^0_{k}}^2R^0_{k}} (x)\;
        \i~\overline{\left.
           R_{k'}
        \right|_{ \overrightarrow{\mathbf{q}}=\overrightarrow{\mathbf{q}}^{0} }(x)}
    \;\dx~,
\end{align*} and
\begin{align*}
    \frac{\partial\varrho_{3}^{k}
        }{
                \partial \omega_{k'}
    }\sts{~\overrightarrow{\mathbf{q}}^{0}~,~R^0~}
    &
    ~=~
    -
    \Re\int
        \partial_x R^0_{k}(x)\;
        \overline{\left.
            \frac{\partial}{\partial \omega_{k'}}R_{k'}
        \right|_{ \overrightarrow{\mathbf{q}}=\overrightarrow{\mathbf{q}}^{0} }(x)}
    \;\dx~,
\\
    \frac{\partial\varrho_{3}^{k}
        }{
                \partial c_{k'}
    }\sts{~\overrightarrow{\mathbf{q}}^{0}~,~R^0~}
    &
    ~=~
    -
    \Re\int
        \partial_x R^0_{k}(x)\;
        \overline{\left.
            \frac{\partial}{\partial c_{k'}}R_{k'}
        \right|_{ \overrightarrow{\mathbf{q}}=\overrightarrow{\mathbf{q}}^{0} }(x)}
    \;\dx~,
\\
    \frac{\partial\varrho_{3}^{k}
        }{
                \partial x_{k'}
    }\sts{~\overrightarrow{\mathbf{q}}^{0}~,~R^0~}
    &
    ~=~
    -
    \Re\int
        \partial_x R^0_{k}(x)\;
        \overline{\left.
            \frac{\partial}{\partial x_{k'}}R_{k'}
        \right|_{ \overrightarrow{\mathbf{q}}=\overrightarrow{\mathbf{q}}^{0} }(x)}
    \;\dx~,
\\
    \frac{\partial\varrho_{3}^{k}
        }{
                \partial \gamma_{k'}
    }\sts{~\overrightarrow{\mathbf{q}}^{0}~,~R^0~}
    &
    ~=~
    \Re\int
        \partial_x R^0_{k}(x)\;
        \i~\overline{\left.
            R_{k'}
        \right|_{ \overrightarrow{\mathbf{q}}=\overrightarrow{\mathbf{q}}^{0} }(x)}
    \;\dx~,
\end{align*} and
\begin{align*}
    \frac{\partial\varrho_{4}^{k}
        }{
                \partial \omega_{k'}
    }\sts{~\overrightarrow{\mathbf{q}}^{0}~,~R^0~}
    &
    ~=~
    -
    \Re\int
        \i~ R^0_{k}(x)\;
        \overline{\left.
            \frac{\partial}{\partial \omega_{k'}}R_{k'}
        \right|_{ \overrightarrow{\mathbf{q}}=\overrightarrow{\mathbf{q}}^{0} }(x)}
    \;\dx~,
\\
    \frac{\partial\varrho_{4}^{k}
        }{
                \partial c_{k'}
    }\sts{~\overrightarrow{\mathbf{q}}^{0}~,~R^0~}
    &
    ~=~
    -
    \Re\int
        \i~ R^0_{k}(x)\;
        \overline{\left.
            \frac{\partial}{\partial c_{k'}}R_{k'}
        \right|_{ \overrightarrow{\mathbf{q}}=\overrightarrow{\mathbf{q}}^{0} }(x)}
    \;\dx~,
\\
    \frac{\partial\varrho_{4}^{k}
        }{
                \partial x_{k'}
    }\sts{~\overrightarrow{\mathbf{q}}^{0}~,~R^0~}
    &
    ~=~
    -
    \Re\int
        \i~ R^0_{k}(x)\;
        \overline{\left.
            \frac{\partial}{\partial x_{k'}}R_{k'}
        \right|_{ \overrightarrow{\mathbf{q}}=\overrightarrow{\mathbf{q}}^{0} }(x)}
    \;\dx~,
\\
    \frac{\partial\varrho_{4}^{k}
        }{
                \partial \gamma_{k'}
    }\sts{~\overrightarrow{\mathbf{q}}^{0}~,~R^0~}
    &
    ~=~
    \Re\int
        \i~ R^0_{k}(x)\;
       \i~\overline{ \left.
         R_{k'}
        \right|_{ \overrightarrow{\mathbf{q}}=\overrightarrow{\mathbf{q}}^{0} }(x)}
    \;\dx~.
\end{align*}
In order to use the implicit function theorem for
$\varrho_{j}^{k}(\overrightarrow{\mathbf{q}},u)$ around
$(\overrightarrow{\mathbf{q}}^{0},R^0)$ with $j=1,2,3,4$ and
$k=1,2$, we only need to verify that the determinant of the
corresponding Jacobian $\Xi$ is nonzero, where
\begin{align}
\label{eq:jacob}
    \Xi =
    \begin{pmatrix}
      \Xi_{1,1} & \Xi_{1,2} \\
      \Xi_{2,1} & \Xi_{2,2} \\
    \end{pmatrix},
\quad
  \Xi_{k,k'}
  =
 \left. \begin{pmatrix}
        \frac{\partial\varrho_{1}^{k}
        }{
            \partial \omega_{k'}
        }
       &
         \frac{\partial\varrho_{1}^{k}
        }{
            \partial c_{k'}
        }
    &
         \frac{\partial\varrho_{1}^{k}
        }{
            \partial x_{k'}
        }
    &
        \frac{\partial\varrho_{1}^{k}
        }{
            \partial \gamma_{k'}
        }
    \\
        \frac{\partial\varrho_{2}^{k}
        }{
            \partial \omega_{k'}
        }
    &
        \frac{\partial\varrho_{2}^{k}
        }{
            \partial c_{k'}
        }
    &
        \frac{\partial\varrho_{2}^{k}
        }{
            \partial x_{k'}
        }
    &
        \frac{\partial\varrho_{2}^{k}
        }{
            \partial \gamma_{k'}
        }
    \\
        \frac{\partial\varrho_{3}^{k}
        }{
            \partial \omega_{k'}
        }
    &
        \frac{\partial\varrho_{3}^{k}
        }{
            \partial c_{k'}
        }
    &
        \frac{\partial\varrho_{3}^{k}
        }{
            \partial x_{k'}
        }
    &
        \frac{\partial\varrho_{3}^{k}
        }{
            \partial \gamma_{k'}
        }
    \\
        \frac{\partial\varrho_{4}^{k}
        }{
            \partial \omega_{k'}
        }
    &
        \frac{\partial\varrho_{4}^{k}
        }{
            \partial c_{k'}
        }
    &
        \frac{\partial\varrho_{4}^{k}
        }{
            \partial x_{k'}
        }
    &
        \frac{\partial\varrho_{4}^{k}
        }{
            \partial \gamma_{k'}
        }
  \end{pmatrix} \right|_{ (\overrightarrow{\mathbf{q}},u)=(\overrightarrow{\mathbf{q}}^{0},R^0)
  }.
\end{align}

First, by a straight calculation, we have for $k=1,2$
\begin{align*}
     \left.  \Xi_{k,k} =  \begin{pmatrix}
        -
     \partial_{\omega_{k}}\mass{ R_{k} }
      & -
      \partial_{c_{k}}\mass{ R_{k} }
      & 0
      & 0
      \\
        - \partial_{\omega_{k}}\momentum{ R_{k} }
      & -\partial_{c_{k}}\momentum{ R_{k} }
      & 0
      & 0
      \\
        -
        \Re\int    \partial_{x}R_{k}\; \overline{\partial_{\omega_k}R_{k}}
       & -
        \Re\int    \partial_{x}R_{k}\; \overline{\partial_{c_k}R_{k}}
       & -
         \Re\int    \partial_{x}R_{k}\; \overline{\partial_{x_{k}}R_{k}}
       & -
         \Re\int    \partial_{x}R_{k}\; \overline{\i R_{k}}
      \\
        -
        \Re\int  \i R_{k} \overline{\partial_{\omega_{k}}R_{k} }
      & -
        \Re\int  \i R_{k} \overline{\partial_{c_{k}}R_{k} }
      & -
        \Re\int  \i R_{k} \overline{\partial_{x_{k}}R_{k} }
      & -
        \Re\int  \i R_{k} \overline{\i R_{k} }
      \end{pmatrix}\right|_{ (\overrightarrow{\mathbf{q}},u)=(\overrightarrow{\mathbf{q}}^{0},R^0)  },
\end{align*}
then as in the proof of Lemma \ref{lem:aslnst:decomp}, we have
\begin{align}
\label{eq:djacob}
    \det\Xi_{k,k}
    =
    -
    \normto{\partial_{x}\phi_{\omega_{k}^{0},c_{k}^{0}}}^{2}
    \normto{\phi_{\omega_{k}^0,c_{k}^0}}^{2}
    \cdot \det d''\sts{~\omega_{k}^0~,~c_{k}^0~},\quad\text{ for } k=1,~2\,.
\end{align}

Next, we consider $\Xi_{k,k'}$ with $k\neq k'.$
Since  the center distance between $\mathcal{R}_{1}^{0}$ and
$\mathcal{R}_{2}^{0}$ is at least $L$, we have
\begin{align*}
    \int
        \abs{ \mathcal{R}_{1}^{0}\sts{x}~\mathcal{R}_{2}^{0}\sts{x} }
    \;\dx
    \leq &\;
    C
    \int
        e^{ -\frac{ \sqrt{4\omega_{1}^{0}-\sts{c_{1}^{0}}^{2}} }{2}\abs{ x-x^0_{1} } }
        ~
        e^{ -\frac{ \sqrt{4\omega_{2}^{0}-\sts{c_{2}^{0}}^{2}} }{2}\abs{ x-x^0_{2} } }
    \;\dx
\\
    \leq & \;
    C
    e^{ -\theta \abs{ x^0_{1}-x^0_{2} } }
    \leq
    C
    e^{ - \theta L },
\end{align*}
where
$C=C\sts{\omega_{1}^{0},~c_{1}^{0},~\omega_{2}^{0},~c_{2}^{0}}$,
$\theta
    =\frac14
    \min
    \ltl{
         \sqrt{4\omega_{1}^{0}-\sts{c_{1}^{0}}^{2}},
        ~
         \sqrt{4\omega_{2}^{0}-\sts{c_{2}^{0}}^{2}}},
$ and $\mathcal{R}_{k}^{0}$ takes one of the expressions
$\left.R_{k}\right|_{
\overrightarrow{\mathbf{q}}=\overrightarrow{\mathbf{q}}^{0} },$
$\left.\partial_{x}R_{k}\right|_{
\overrightarrow{\mathbf{q}}=\overrightarrow{\mathbf{q}}^{0} },$
$\left.\frac{\partial}{\partial \omega_{k}}R_{k}^{0}\right|_{
\overrightarrow{\mathbf{q}}=\overrightarrow{\mathbf{q}}^{0} },$
$\left.\frac{\partial}{\partial c_{k}}R_{k}\right|_{
\overrightarrow{\mathbf{q}}=\overrightarrow{\mathbf{q}}^{0} },$
$\left.\frac{\partial}{\partial x_{k}}R_{k}\right|_{
\overrightarrow{\mathbf{q}}=\overrightarrow{\mathbf{q}}^{0} }$ and
$\left.\frac{\partial}{\partial \gamma_{k}}R_{k}\right|_{
\overrightarrow{\mathbf{q}}=\overrightarrow{\mathbf{q}}^{0} }.$ It
follows that for  $k\neq k'$ and $j=1,~2,~3,~4\,$
\begin{align*}
    \abs{
        \frac{\partial\varrho_{j}^{k}
        }{
                \partial \omega_{k'}
        }
        \sts{~\overrightarrow{\mathbf{q}}^{0},R^0}
    }
    +
    \abs{
        \frac{\partial\varrho_{j}^{k}
        }{
                \partial c_{k'}
        }
        \sts{~\overrightarrow{\mathbf{q}}^{0},R^0}
    }
    +
    \abs{
        \frac{\partial\varrho_{j}^{k}
        }{
                \partial x_{k'}
        }
        \sts{~\overrightarrow{\mathbf{q}}^{0},R^0}
    }
    +
    \abs{
        \frac{\partial\varrho_{j}^{k}
        }{
                \partial \gamma_{k'}
        }
        \sts{~\overrightarrow{\mathbf{q}}^{0},R^0}
    }
    \leq
    C
    e^{ - \theta L },
\end{align*} which means
\begin{align}
\label{eq:sjacob}
  \det \Xi_{k,k'}= \bigo{ e^{ -\theta L } }, \quad\text{ for } k\neq k'.
\end{align}

Now, inserting \eqref{eq:djacob} and \eqref{eq:sjacob} into
\eqref{eq:jacob}, we have
\begin{align*}
    \det\Xi
    =
    \prod_{k=1}^{2}\left(
        \normto{\partial_{x}\phi_{\omega_{k}^{0},c_{k}^{0}}}^{2}
        \normto{\phi_{\omega_{k}^0,c_{k}^0}}^{2}
        \cdot \det d''\sts{~\omega_{k}^0~,~c_{k}^0~}\right)
        +
        \bigo{
            e^{ -\theta L }
            },
\end{align*}
which implies that there exists $L_0\triangleq
L_0(\omega^0_1,~c^0_1, ~\omega^0_2,~c^0_2)$ large enough such that
\begin{align*}
    \det\Xi\neq 0 \quad \text{for}\;L>L_0.
\end{align*}

At last, the implicit function theorem implies the results for any $u\in B\left( R^{0} , \alpha\right)$ with small $\alpha$, and the estimate \eqref{est:prior} with constant $C_{II}$ is indepentdent of parameters $x_k, \gamma_k$, $k=1, 2$ of the ball $B\left( R^{0} , \alpha\right)$, provided that $x_2-x_1>L_0$.
\end{proof}

Now we apply the above decomposition of the function to the solution
$u(t)$ of \eqref{DNLS} in $[0, T_0]$, and obtain the corresponding
dynamical version.  More precisely, we have
\begin{lemma}
\label{lem:tsln:decomp} Suppose
$u\in\mathcal{C}\sts{~\left[0~,~T_{0}\right],~H^{1}\sts{\R}~}$ is a
solution to \eqref{DNLS} with initial data
$u_{0}\in\mathcal{U}\sts{\alpha~,~L,~\omega^{0}_1~,~c^{0}_1,~\omega^{0}_2~,~c^{0}_2}$,
and
\begin{align*}
  u\sts{t}\in
  \mathcal{U}\sts{\alpha~,~\frac{L}{2}~,~\omega_{1}^{0}~,~c_{1}^{0}~,~\omega_{2}^{0}~,~c_{2}^{0}~},
  \quad\text{ for any }t\in\left(0~,~T_{0}\right],
\end{align*}
where $\alpha<\alpha_{0}$ and $\frac{L}{2}>\widetilde{L}_{0}$ with
$\alpha_{0}$ and $\widetilde{L}_{0}$ given by Lemma
\ref{lem:tslnst:decomp}. Then there exist unique $\mathcal{C}^{1}$
functions $$
    \overrightarrow{\mathbf{q}}(t)\triangleq\,
    \sts{ ~\omega_{1}(t),~c_{1}(t),~x_{1}(t),~\gamma_{1}(t),~\omega_{2}(t),~c_{2}(t),~x_{2}(t),~\gamma_{2}(t)}
$$ on $\left[0,~T_{0}\right]$ with values
$\sts{0,~+\infty}\times\R^3\times\sts{0,~+\infty}\times\R^3 $ and
$\left(c_{k}\sts{t}\right)^2 < 4\omega_{k}\sts{t}$, $k=1,2$ for all
$t\in\left[0~,~T_{0}\right]$ such that
\begin{align}
        \Re\int
 R_{k}\sts{t}\; \overline{\varepsilon\sts{t}}
        \;\dx =0,
        &\quad
        \Re\int\sts{\i \partial_x R_{k} + \frac12
          \abs{R_{k}}^2R_{k}} \sts{t}\; \overline{\varepsilon\sts{t}}
        \;\dx =0, \label{tsl:orth:nondeg}
\\
   \Re\int \partial_x R_{k}\sts{t}\; \overline{\varepsilon\sts{t}}
        \;\dx =0,
& \quad
   \Re\int\i R_{k}\sts{t}\;
           \overline{ \varepsilon\sts{t}}
        \;\dx =0, \label{tsl:orth:sym}
\end{align}
where $k=1,2$ and
\begin{align}
\label{tsln:tw}
    R_{k}\sts{t\,,\,x}
    = & \;
    R
    \sts{
        \,\omega_{k}\sts{t},\,c_{k}\sts{t},\,x_{k}\sts{t},\,\gamma_{k}\sts{t};\,x\,
    }
    \triangleq
    \varphi_{\omega_{k}\sts{t},c_{k}\sts{t}}\sts{x-x_{k}\sts{t}}e^{\i\gamma_{k}\sts{t}},
\\
\label{tsln:rem:def}
    \varepsilon
    \sts{t\,,\,x}
    = & \;
    \varepsilon
    \sts{\,\overrightarrow{\mathbf{q}}\sts{t}\,,\,u\sts{t\,,\,x}\,;
            \,x\,
    }\triangleq u\sts{t\,,\,x} - \sum^2_{k=1} R_{k}\sts{t\,,\,x}.
\end{align}
Moreover, for any $t\in \left[0~,~T_{0}\right],$ we have
\begin{align}
\label{tsln:sl:rough}
      \normhone{\varepsilon\sts{t}}
      +
      \sum_{k=1}^{2}\big(\abs{ \omega_k\sts{t}-\omega_{k}^{0} }
      +
    \abs{ c_k\sts{t}-c_{k}^{0} } \big)
      \leq
      C_{II}\;\alpha,
\end{align}
\begin{align}\label{tsl:cent:dist} x_{2}\sts{t}-x_{1}\sts{t} \geq \frac{L}{2},
\end{align}
and
\begin{align}
 \abs{\dot{\omega}_{k}\sts{t}}~
            +~\abs{\dot{c}_{k}\sts{t}}~
              +~\abs{ \dot{x}_{k}\sts{t}-c_{k}\sts{t}}~
            +  &~\abs{ \dot{\gamma_{k}}\sts{t}-\omega_{k}\sts{t} }
            \notag
     \\
        \leq & \;
        C_{II}\sts{ \|\varepsilon(t)\|_{H^1}
                +
                e^{ -\theta_{1}\sts{ \frac{L}{2} + \theta_1 t } }
            }, \label{tsln:para:dyn}
\end{align}
where $\displaystyle \theta_{1}
    = \frac14
    \min
    \ltl{
        \sqrt{4\omega_{1}^{0}-\sts{c_{1}^{0}}^{2}},
        ~
        \sqrt{4\omega_{2}^{0}-\sts{c_{2}^{0}}^{2}}, ~ c_{2}^{0}-c_{1}^{0}}.$
\end{lemma}
\begin{proof}Applying Lemma \ref{lem:tslnst:decomp} to $u\sts{t}$ for all
$t\in [0~,~T_{0}]$,  we can obtain \eqref{tsl:orth:nondeg},
\eqref{tsl:orth:sym} and \eqref{tsln:sl:rough}. It remains to show
\eqref{tsl:cent:dist} and \eqref{tsln:para:dyn}.

Since $u\sts{t}\in
\mathcal{U}\sts{\alpha,~\frac{L}{2},~\omega_{1}^{0},~c_{1}^{0},~\omega_{2}^{0},~c_{2}^{0}}$
for $t\in (0, T_0]$, there exist $x_{1}^{0}\sts{t},$
$x_{2}^{0}\sts{t},$ $\gamma_{1}^{0}\sts{t}$ and
$\gamma_{2}^{0}\sts{t}$ such that $ x_{2}^{0}\sts{ t } -
x_{1}^{0}\sts{ t }\geq \frac{L}{2}$ and
\begin{align}
\label{dynam:eq:2h1neighborhood}
    \normhone{
        u\sts{ t, \cdot }
        -
        \sum^2_{k=1}
    R_{k}^{0}\sts{t\,,\,\cdot}
    }<\alpha,
\end{align}
where
\begin{align*}
  R_{k}^{0}\sts{t\,,\,x}
  &
  =
  R
  \sts{
    ~\omega_{k}^{0}~,~c_{k}^{0}~,
    ~x_{k}^{0}\sts{ t }~,~\gamma_{k}^{0}\sts{ t }~;
    ~x~
  }
  =
  \varphi_{ \omega_{k}^{0}, c_{k}^{0} }
    \sts{x-x_{k}^{0}\sts{ t }}e^{\i\gamma_{k}^{0}\sts{ t }}.
\end{align*}
As in the proof of Lemma \ref{lem:tslnst:decomp}, we know that for
$k=1, 2,$
\begin{align*}
  \abs{ x_{k}\sts{ t } - x_{k}^{0}\sts{ t } }<\; C_{II}~\alpha,
\end{align*}
which implies
\begin{align*} x_{2}\sts{t}-x_{1}\sts{t} \geq \frac{L}{4},
\end{align*}
for sufficient small $\alpha$ and sufficient large $L$. We will
improve this estimate by the dynamics of $\dot x_k(t)$ and
$x_2(0)-x_1(0)\geq \frac{L}{2}$.

The $\mathcal{C}^{1}$ regularity of $\omega_{k}\sts{t},$
$c_{k}\sts{t},$ $x_{k}\sts{t}$ and $\gamma_{k}\sts{t}$ in $t$ can be
shown by a standard regularization argument, we can refer to
\cite{MartelM:Instab:gKdV} for more details. Now we formally verify
\eqref{tsln:para:dyn} by the equation of $\varepsilon\sts{t}$, and
the orthogonal structure \eqref{tsl:orth:nondeg} and
\eqref{tsl:orth:sym}. A simple calculation gives that
\begin{align}
i\partial_t \varepsilon  + \mathcal{L} \varepsilon =\quad &
\sum^{2}_{k=1}\sts{-\i \partial_t R_k - \partial^2_x R_k -  \frac12
\i |R_k|^2 \partial_x R_k + \frac12 \i R_k^2\overline{\partial_xR_k}
- \frac{3}{16}|R_k|^4R_k}
\notag\\
&- \frac12 \i |R_1+R_2|^2 \partial_x \sts{R_1+R_2} + \frac12 \i
|R_1|^2 \partial_x R_1 +   \frac12 \i |R_2|^2 \partial_x R_2
\notag\\
&+ \frac12 \i \sts{R_1+R_2}^2\overline{\partial_x \sts{R_1+R_2}} -
\frac12 \i R_1^2\overline{\partial_xR_1} - \frac12 \i
R_2^2\overline{\partial_xR_2}
\notag\\
&- \frac{3}{16}|R_1+R_2|^4\sts{R_1+R_2} + \frac{3}{16} |R_1|^4R_1 +
\frac{3}{16}|R_2|^4R_2
\notag\\
&+ H.O.T, \label{eq:rem}
\end{align}
where $\mathcal{L} \varepsilon$ and $H.O.T$ are defined by
\begin{align*}
 \mathcal{L} \varepsilon =
\partial^2_x \varepsilon &+  \frac12 \i \left[|R_1+R_2|^2 \partial_x
\varepsilon + 2 \partial_x \sts{R_1+R_2} \Re
\sts{\overline{(R_1+R_2)} \varepsilon}\right]
\\
&\;- \frac12 \i \left[\sts{R_1+R_2}^2 \overline{\partial_x
\varepsilon} + 2 \overline{ \partial_x\sts{R_1+R_2}}
\left(R_1+R_2\right) \varepsilon \right]
\\
&\; +\frac{3}{16}\left[|R_1+R_2|^4 \varepsilon + 4 |R_1+R_2|^2 \Re
\sts{\overline{(R_1+R_2)} \varepsilon}\right],
\end{align*}
and
\begin{align*}
H.O.T= & -\frac12 \i \left[2\partial_x \varepsilon \Re
\sts{\overline{(R_1+R_2)}\varepsilon} + \partial_x(R_1 +
R_2)|\varepsilon|^2 + |\varepsilon|^2\partial_x \varepsilon\right]
\\
& + \frac12 \i \left[2\overline{\partial_x \varepsilon}
\sts{R_1+R_2}\varepsilon  + \overline{\partial_x\sts{R_1 + R_2}}
\varepsilon ^2 + \varepsilon^2\overline{\partial_x
\varepsilon}\right]
\\
&-\frac{3}{16}\left[ 4(R_1+R_2)\sts{\Re
\sts{\overline{(R_1+R_2)}\varepsilon}}^2 + 2 |R_1 + R_2|^2 (R_1+R_2)
|\varepsilon|^2\right.
\\
&\qquad + 4 |R_1+R_2|^2\varepsilon \Re
\sts{\overline{(R_1+R_2)}\varepsilon}+4(R_1+R_2)|\varepsilon|^2\Re\sts{\overline{(R_1+R_2)}
\;\varepsilon}
\\
& \qquad + 4\sts{\Re{\overline{(R_1+R_2)}\; \varepsilon}}^2
\varepsilon +2
|R_1+R_2|^2|\varepsilon|^2\varepsilon+(R_1+R_2)|\varepsilon|^4 \\
&\qquad \left. + 4
\Re\sts{\overline{(R_1+R_2)}\varepsilon}|\varepsilon|^2\varepsilon
+|\varepsilon|^4\varepsilon\right].
\end{align*}

By \eqref{eq:tw}, we know that $R_k(t,x)$, $k =1, 2$, satisfy
\begin{align*}
& -i  \partial_t R_k - \partial^2_x R_k -  \frac12 \i |R_k|^2
\partial_x R_k + \frac12 \i R_k^2\overline{\partial_xR_k} -
\frac{3}{16}|R_k|^4R_k \notag
\\
= & - i \dot{\omega}_k(t) ~
                \frac{ \partial }{ \partial{\omega_{k}} }R_{k}(t)
            -i \dot{c}_{k}\sts{t}~
                \frac{ \partial }{ \partial{c_{k}} }R_{k}(t)
\\
&
 -i \sts{ \dot{x}_{k}\sts{t}-c_{k}(t)}~
                \frac{ \partial }{ \partial{x_{k}} }R_{k}(t)
-i \sts{ \dot{\gamma_{k}}(t)-\omega_{k}(t) }
                \frac{ \partial }{ \partial{\gamma_{k}} }R_{k}(t).
\end{align*}
Inserting the above identity into \eqref{eq:rem}, we have
\begin{align}
i\partial_t \varepsilon  + \mathcal{L} \varepsilon +&
\sum^{2}_{k=1}\sts{ i \dot{\omega}_k(t) ~
                \frac{ \partial }{ \partial{\omega_{k}} }R_{k}(t)
            +i \dot{c}_{k}\sts{t}~
                \frac{ \partial }{ \partial{c_{k}} }R_{k}(t)}
\label{eq:paras}\\
 +&\sum^{2}_{k=1}\sts{i \sts{ \dot{x}_{k}\sts{t}-c_{k}(t)}~
                \frac{ \partial }{ \partial{x_{k}} }R_{k}(t)
+i \sts{ \dot{\gamma_{k}}(t)-\omega_{k}(t) }
                \frac{ \partial }{ \partial{\gamma_{k}} }R_{k}(t)}
\notag\\
=\; &- \frac12 i |R_1+R_2|^2 \partial_x \sts{R_1+R_2} + \frac12 i
|R_1|^2 \partial_x R_1 +   \frac12 i |R_2|^2 \partial_x R_2
\label{eq:winter:a}\\
&+ \frac12 i \sts{R_1+R_2}^2\overline{\partial_x \sts{R_1+R_2}} -
\frac12 i R_1^2\overline{\partial_xR_1} - \frac12 \i
R_2^2\overline{\partial_xR_2}
\label{eq:winter:b}\\
&- \frac{3}{16}|R_1+R_2|^4\sts{R_1+R_2} + \frac{3}{16} |R_1|^4R_1 +
\frac{3}{16}|R_2|^4R_2
\label{eq:winter:c}\\
&+ H.O.T. \notag
\end{align}
Because $R_1(t,x)$ and $R_2(t,x)$ have exponential decay and their
centers locate far away from each other (distance is at least $L/4$
from \eqref{tsl:cent:dist}), we know that
\eqref{eq:winter:a}-\eqref{eq:winter:c} are weak interaction terms
between $R_1(t,x)$ and $R_2(t,x)$.

Now we will combine the above equation about $\varepsilon(t,x)$ with
the orthogonal structures \eqref{tsl:orth:nondeg},
\eqref{tsl:orth:sym} for $k=1,2$ to show \eqref{tsln:para:dyn}. On
one hand, multiplying the above equation by
$\overline{R_{k}\sts{t}}$, $\overline{\i
\partial_x R_{k}\sts{t} + \frac{1}{2} \abs{R_{k}}^2R_{k}\sts{t}}$, $\overline{\partial_{x}R_{k}\sts{t}},$
$\overline{\i~R_{k}\sts{t}},$ for $k=1,2$,  and  taking the
imaginary part, we have from \eqref{tsln:sl:rough} and $ \det\Xi_{k,k}
> 0$, $k=1,2$ that
\begin{multline}
\label{tsln:para:dynr}
    \abs{\dot{\omega}_{k}\sts{t}}~
            +~\abs{\dot{c}_{k}\sts{t}}~
            +~\abs{ \dot{x}_{k}\sts{t}-c_{k}\sts{t}}~
            +~\abs{ \dot{\gamma_{k}}\sts{t}-\omega_{k}\sts{t} }~
        \leq
       C \sts{ \|\varepsilon(t)\|_{H^1}
                +
                e^{ -\theta_{1}  \frac{L}{4}   }
            },
\end{multline}
where $\|\varepsilon(t)\|_{H^1}$ comes from the contribution of the
linear terms $i\partial_t \varepsilon  + \mathcal{L} \varepsilon $
and $H.O.T$ term, and $e^{ -\theta_{1}  \frac{L}{4}   }$ comes from
the weak interaction terms in
\eqref{eq:winter:a}-\eqref{eq:winter:c}, especially the fact that
$x_2(t)-x_1(t)\geq \frac{L}{4}$. On the other hand, in order to get
the precise estimate, we estimate $x_2(t)-x_1(t)$ as following
\begin{align*}
       \dot{x}_{2}\sts{t}-\dot{x}_{1}\sts{t}
&
    =    \sts{ \dot{x}_{2}\sts{t}-c_{2}\sts{t}} -\sts{ \dot{x}_{1}\sts{t}-c_{1}\sts{t}} + \sts{ c_{2}\sts{t}-c_{1}\sts{t} }
\\
&
    \geq \sts{ c_{2}\sts{t}-c_{1}\sts{t} } - 2 C \left(C_{II}
    \alpha + e^{ -2\theta_{1}   \frac{L}{4}} \right)
\\
&
    \geq \sts{ c_{2}^{0}-c_{1}^{0} } -
     2~C_{II}\alpha - 2 C \left(C_{II}
    \alpha + e^{ -2\theta_{1}   \frac{L}{4}} \right)
\\
&
    \geq \frac{ c_{2}^{0}-c_{1}^{0} }{4} \geq \theta_1
\end{align*}
for sufficient small $\alpha$ and large $L$. This yields
\begin{align*}
    x_{2}\sts{t} - x_{1}\sts{t}
&
    =
    x_{2}\sts{0} - x_{1}\sts{0}
    +
    \int_{0}^{t}\sts{ \dot{x}_{2}\sts{s}-\dot{x}_{1}\sts{s} }\;\mathrm{d}s
    \geq \frac{L}{2} + \theta_1 t,
\end{align*}
from which we can obtain \eqref{tsl:cent:dist} and
\eqref{tsln:para:dyn}.
\end{proof}


\section{Monotonicity formulas for the derivative NLS}\label{sect:mono forms}
As shown in Section \ref{sect:stab:asln}, under the non-degenerate
condition
\begin{align*}
  \det d''\sts{~\omega~,~c~}  &<0,
\end{align*}
there are two key points to show the stability of the solution of
\eqref{DNLS} around the single traveling wave in the energy space
besides of the modulation analysis. One is the action functional $
\mathfrak{J}_{\omega(0), c(0)}\sts{ u\sts{t} }$, which is a
conserved quantity, and is used to obtain the refined estimate of
$\normhone{ \varepsilon\sts{t} } $ by the perturbation argument,
\begin{align*}
\normhone{ \varepsilon\sts{t} }^{2}
  \leq
       C \normhone{ \varepsilon\sts{0} }^{2}
    +   C\sts{\abs{ \omega\sts{t}-\omega\sts{0} }^{2}
    +   \abs{ c\sts{t}-c\sts{0} }^{2}}.
\end{align*}
The other is the conservation laws of mass and momentum\footnote{Of
course, $\omega(t)$ and $c(t)$ are related to the mass and momentum
of the corresponding traveling wave at time $t$.}, which is enough
to show the refined estimate that
\begin{align*}
  \abs{\omega\sts{t}-\omega\sts{0}} + \abs{c\sts{t}-c\sts{0}}\leq C
    \sts{
        \normhone{\varepsilon\sts{t}}^{2}+\normhone{\varepsilon\sts{0}}^{2}
    }.
\end{align*}

As for the multi-traveling waves case, inspired by the ideas in
\cite{MartelMT:Stab:NLS}, we will introduce the localized action
functional $ \mathfrak{E}\sts{u\sts{t}}=
    E\sts{u\sts{t}}+\mathfrak{F}\sts{t}$, which is almost conserved and
    can be used to obtain the refined estimate of $\normhone{ \varepsilon\sts{t} } $ by
the perturbation argument. Since it is not enough to refine the
estimates of $\abs{\omega_k\sts{t}-\omega_k\sts{0}} +
\abs{c_k\sts{t}-c_k\sts{0}}$, $k=1, 2$, only by the conservation
laws of mass and momentum, we need to introduce some kinds of the
localized functionals $\mathfrak{Q}_{\pm,0}\sts{t}$ and
$\mathfrak{Q}_{0,\pm}\sts{t}$ and characterize its dynamical
estimate, which is related to $\abs{\omega_k\sts{t}-\omega_k\sts{0}}
+ \abs{c_k\sts{t}-c_k\sts{0}}$. They are the main goals in this
section.

\subsection{Monotone result for the line $x = \overline{x}^{0}+\sigma t $}  First of all, we introduce a suitable
cutoff function $h,$ which is a nondecreasing $\mathcal{C}^{3}$
function with
\begin{align*}
    h\sts{ x }
    =
    \left\{
        \begin{array}{ll}
           0        ,           & \text{ if } x<-1,                \\
           1        ,           & \text{ if } x>1,
        \end{array}
    \right.
\end{align*}
\begin{align} \label{est:ctff}
    \sts{ h'\sts{x} }^{2}\leq C h\sts{x},
    \quad
    \sts{ h''\sts{x} }^{2}\leq C h'\sts{x},\quad \text{for }x\in\R,
\end{align}
and
\begin{align*}
    h'\sts{x}>0,\quad\text{for }x\in\sts{-1~,~1}.
\end{align*}
Next, by setting\footnote{By the assumption (c) in Theorem
\ref{thm:stab:tsln}, we know that $c^0_1 < \sigma < c^0_2$.}
\begin{align*}
    \overline{x}^{0}=\frac{ x_{1}^{0}+x_{2}^{0} }{2},
    \quad
    \sigma=2\frac{ \omega_{2}\sts{0}-\omega_{1}\sts{0} }{ c_{2}\sts{0}-c_{1}\sts{0} },
    \quad
    a = \frac{L^{2}}{64},
\end{align*}
and
\begin{align}\label{ctf}
    \mathfrak{h} \sts{t,x} &=h\sts{ \frac{x-\overline{x}^{0}-\sigma t }{\sqrt{t+a}} },
    \quad
    \mathfrak{g}\sts{t,x}=1-\mathfrak{h}\sts{t,x},
\end{align}
we define the functional $\mathfrak{F}\sts{t}$, including the
localized mass and momentum about each traveling wave,
\begin{align}
\label{acf:lMM}
  \mathfrak{F}\sts{t}
\notag
    =
    &
        \omega_{1}\sts{0}\int \frac{1}{2}\abs{u\sts{t}}^{2}\mathfrak{g}\sts{t}\;\dx
    + c_{1}\sts{0} \int\sts{ -\frac12 \Im \left(\overline{u}\partial_{x}u\right) +\frac18
    \abs{u}^{4}}\sts{t}\mathfrak{g}\sts{t}\;\dx
\\
&   + \omega_{2}\sts{0} \int\frac{1}{2}
\abs{u\sts{t}}^{2}\mathfrak{h}\sts{t}\;\dx + c_{2}\sts{0}
        \int\sts{   -   \frac{1}{2}\; \Im \left(\overline{u}\partial_{x}u\right)+\frac18\abs{u}^{4}
        }\sts{t}\mathfrak{h}\sts{t}\;\dx
\end{align}
A simple calculation gives us that
\begin{align}
\label{acf:monf} \notag
    \mathfrak{F}\sts{t}
= &\;  \omega_{1}\sts{0}
        \int
            \frac{1}{2}\abs{u\sts{t}}^{2}
        \;\dx +c_{1}\sts{0}
        \int  \sts{-   \frac{1}{2}\Im \left(
            \overline{u}\partial_{x}u\right)\sts{t} + \frac18 \abs{u\sts{t}}^{4} }
        \;\dx
\\
\notag
&
    + \sts{\omega_{2}\sts{0}-\omega_{1}\sts{0} }
        \int
            \frac{1 }{2}\abs{u\sts{t}}^{2}\mathfrak{h}\sts{t}
        \;\dx
    \\
    &+ \sts{c_{2}\sts{0}-c_{1}\sts{0}}
       \int\sts{   -   \frac{1}{2}\; \Im\left(\overline{u}\partial_{x}u\right)+\frac18\abs{u}^{4} }\sts{t} \mathfrak{h}\sts{t}
        \;\dx \notag
\\
=
&
   \; \omega_{1}\sts{0}\mass{u\sts{t}} + c_{1}\sts{0}\momentum{u\sts{t}}+\mathfrak{Q}\sts{t},
\end{align}
where
\begin{align*}
\mathfrak{Q}\sts{t} =&
      \sts{\omega_{2}\sts{0}-\omega_{1}\sts{0} }
        \int\frac{ 1}{2}
            \abs{u\sts{t}}^{2}\mathfrak{h}\sts{t}
        \;\dx
\\
& + \sts{c_{2}\sts{0}-c_{1}\sts{0}}
       \int\sts{   -   \frac{1}{2}\; \Im\left(\overline{u}\partial_{x}u\right)+\frac18\abs{u}^{4} }\sts{t} \mathfrak{h}\sts{t}
        \;\dx.
\end{align*}

Next, we have
\begin{lemma}
\label{lem:monf:dert}
    Let $u\sts{t}$ be a solution of \eqref{DNLS} satisfying the assumption of Lemma \ref{lem:tsln:decomp} on $[0, T_0]$.
    Then there exist $C>0$ such that
\begin{align*}
     \frac{\mathrm{d}}{\mathrm{d}t}\mathfrak{Q}\sts{t}
\leq  \;
    \frac{C}{\sts{t+a}^{3/2} }
          \sts{  \int_{ \abs{x-\overline{x}^{0}-\sigma t}<\sqrt{t+a} }
                \abs{u(t,x)}^{2}
            \;\dx + \left(\int_{ \abs{x-\overline{x}^{0}-\sigma t}<\sqrt{t+a} }
                \abs{u(t,x)}^{2}
            \;\dx \right)^3}.
\end{align*}
\end{lemma}
\begin{proof}By
 \eqref{DNLS}, we have
\begin{align*}
 \frac{1}{c_2(0)-c_1(0)}  &  \frac{\mathrm{d}}{\mathrm{d}t}\mathfrak{Q}\sts{t}
\\
= &\;   \frac{
            \sigma}{
            2\sqrt{t+a}
        }
       \int \Im
            \sts{\overline{u(t)}\partial_{x} u(t)}
            \;
            h'\sts{ \frac{x-\overline{x}^{0}-\sigma t }{\sqrt{t+a}} }
        \;\dx
        \\
&   +\frac14 \sigma  \int
            \abs{u\sts{t}}^{2}
            \;
            h'\sts{ \frac{x-\overline{x}^{0}-\sigma t }{\sqrt{t+a}}
            }\left[-\frac{\sigma}{\sqrt{t+a}} - \frac{x-\overline{x}^{0}-\sigma t }{2\sts{t+a}^{3/2}}\right]
        \;\dx
\\
& -   \frac{1 }{ \sqrt{t+a} }
        \int
            \abs{\partial_{x} u\sts{t}}^{2}
            \;
            h'\sts{ \frac{x-\overline{x}^{0}-\sigma t }{\sqrt{t+a}} }
        \;\dx
\\
&
    +   \frac{1}{4\sts{t+a}^{3/2}}
        \int
            \abs{u\sts{t}}^{2}
            \;
            h'''\sts{ \frac{x-\overline{x}^{0}-\sigma t }{\sqrt{t+a}} }
        \;\dx
\\
&
    +   \frac{1}{16 \sqrt{t+a} }
        \int
            \abs{u\sts{t}}^{6}
            \;
            h'\sts{ \frac{x-\overline{x}^{0}-\sigma t }{\sqrt{t+a}} }
        \;\dx
\\
&
   -\frac12
        \int
           \Im \sts{\overline{u(t)}\partial_{x} u(t)}
            \;
            h'\sts{ \frac{x-\overline{x}^{0}-\sigma t }{\sqrt{t+a}} }
            \left[-\frac{\sigma}{\sqrt{t+a}} - \frac{x-\overline{x}^{0}-\sigma t }{2\sts{t+a}^{3/2}}\right]
        \;\dx
\\
&
   +  \frac18
        \int
            \abs{u\sts{t}}^{4}
            \;
            h'\sts{ \frac{x-\overline{x}^{0}-\sigma t }{\sqrt{t+a}} }
            \left[-\frac{\sigma}{\sqrt{t+a}} - \frac{x-\overline{x}^{0}-\sigma t }{2\sts{t+a}^{3/2}}\right]
        \;\dx.
\end{align*}
Now by introducing
\begin{align*}
v\sts{t,x}\triangleq ~e^{-\i\frac{1}{2}\; \sigma x }u\sts{t,x},
\end{align*}
we have
\begin{align}
  \frac{1}{c_2(0)-c_1(0)}  &
  \frac{\mathrm{d}}{\mathrm{d}t}\mathfrak{Q}\sts{t} \label{est:monf}
\\
=&  -   \frac{1 }{ \sqrt{t+a} }
        \int
            \abs{\partial_{x} v\sts{t}}^{2}
            \;
            h'\sts{ \frac{x-\overline{x}^{0}-\sigma t }{\sqrt{t+a}} }
        \;\dx \notag
\\
& +\frac{ 1}{
            4(t+a)
        }
       \int \Im
            \sts{\overline{v(t)}\partial_{x} v(t)}
            \;
            h'\sts{ \frac{x-\overline{x}^{0}-\sigma t }{\sqrt{t+a}}
            } \frac{x-\overline{x}^{0}-\sigma t }{\sqrt{t+a}}
        \;\dx \label{est:im}
\\
&
    +   \frac{1}{4\sts{t+a}^{3/2}}
        \int
            \abs{v\sts{t}}^{2}
            \;
            h'''\sts{ \frac{x-\overline{x}^{0}-\sigma t }{\sqrt{t+a}} }
        \;\dx \notag
\\
&
    +   \frac{1}{16 \sqrt{t+a} }
        \int
            \abs{v\sts{t}}^{6}
            \;
            h'\sts{ \frac{x-\overline{x}^{0}-\sigma t }{\sqrt{t+a}} }
        \;\dx \label{est:sth}
\\
&
   +  \frac18
        \int
            \abs{v\sts{t}}^{4}
            \;
            h'\sts{ \frac{x-\overline{x}^{0}-\sigma t }{\sqrt{t+a}} }
            \left[-\frac{\sigma}{\sqrt{t+a}} - \frac{x-\overline{x}^{0}-\sigma t }{2\sts{t+a}^{3/2}}\right]
        \;\dx. \label{est:fth}
\end{align}

\step{Estimate of \eqref{est:im}} By the Cauchy-Schwarz inequality
and the support property of $h'$, we have
\begin{align}
 &  \left|  \frac{ 1}{
            4(t+a)
        }
       \int \Im
            \sts{\overline{v(t)}\partial_{x} v(t)}
            \;
            h'\sts{ \frac{x-\overline{x}^{0}-\sigma t }{\sqrt{t+a}}
            } \frac{x-\overline{x}^{0}-\sigma t }{\sqrt{t+a}}
        \;\dx \right| \label{est:im:ok}
\\
\leq &\;
    \frac{ 1 }{ 8 \sqrt{t+a} }
        \int \abs{\partial_{x}v\sts{t}}^{2}
            h'\sts{ \frac{x-\overline{x}^{0}-\sigma t }{\sqrt{t+a}} }
        \;\dx \notag
            +
    \frac{C}{\sts{t+a}^{3/2} }
           \int_{ \abs{x-\overline{x}^{0}-\sigma t}<\sqrt{t+a} }
                \abs{v\sts{t}}^{2}
            \;\dx. \notag
\end{align}

\step{Estimate of \eqref{est:sth}}By \eqref{est:ctff}, Lemma $3.3$
in \cite{MartelMT:Stab:NLS}, we have
\begin{align} \label{est:sth:ok}
 & \frac{1}{16 \sqrt{t+a} }
        \int
            \abs{v\sts{t}}^{6}
            \;
            h'\sts{ \frac{x-\overline{x}^{0}-\sigma t }{\sqrt{t+a}} }
        \;\dx
\\
        \leq  &\;
    \frac{ 1 }{ 8 \sqrt{t+a} }
        \int \abs{\partial_{x}v(t)}^{2}
            h'\sts{ \frac{x-\overline{x}^{0}-\sigma t }{\sqrt{t+a}} }
        \;\dx
+
    \frac{C}{ \sts{t+a}^{3/2} }
    \sts{   \int_{ \abs{x-\overline{x}^{0}-\sigma t}<\sqrt{t+a} }
                \abs{v(t)}^{2}
            \;\dx
    }^{3}. \notag
\end{align}

Before we estimate the fourth order term, we first give the
following useful lemma, which is similar to Lemma $3.3$ in
\cite{MartelMT:Stab:NLS} for the sixth order term.
\begin{lemma}
\label{lem:est:fth} Let $h(x)\geq 0$ be a $\mathcal{C}^{2}$ bounded
function. Then for any $w\in H^{1},$ we have
    \begin{align*}
        \int\abs{w(x)}^{4}h(x)\;\dx
    \leq
       C   \sts{ \int \abs{w_{x}}^{2}h \;\dx }^{ 1/2}
            \sts{ \int_{\mathrm{supp}\; h}\abs{w}^{2}\;\dx }^{3/2}
        +   \int \abs{w}^{2}h_{x}\;\dx \int_{\mathrm{supp}\;
        h}\abs{w}^{2}\;\dx,
  \end{align*}
  where $\mathrm{supp}\; h$ denotes the support of $h$.
\end{lemma}
\begin{proof}First, the Leibnitz rule gives us that
\begin{align*}
    \frac{ \mathrm{d} }{ \mathrm{d}x }\sts{ \abs{w}^{2}h }
&
    =
    2h\Re\sts{ \overline{w}w_{x} } + \abs{w}^{2}h_{x},
\end{align*}
therefore, it follows from $w\in H^{1}(\R)$ and the boundedness of
$h$ that
\begin{align*}
  \sts{ \abs{w}^{2}h }\sts{x} &=
        2\int_{-\infty}^{x}h\Re\sts{ \overline{w}w_{x} }\;\dx
    +   \int_{-\infty}^{x}\abs{w}^{2}h_{x}\;\dx,
\end{align*}
which implies that
\begin{align*}
    \left\|~\abs{w}^{2}h~\right\|_{ L^{\infty} }
&   \leq
    2\int\abs{w}\abs{w_{x}}h\;\dx + \int\abs{w}^{2}\abs{h_{x}}\;\dx
\\
&   \leq
    2\sts{ \int\abs{w_{x}}^{2}h }^{1/2}\sts{ \int\abs{w}^{2}h }^{1/2}
    +
    \int\abs{w}^{2}\abs{h_{x}}\;\dx.
\end{align*}
Therefore,  we have
\begin{align*}
    \int\abs{w}^{4}h\;\dx
\leq &
    \left\|~\abs{w}^{2}h~\right\|_{ L^{\infty} }\int_{\mathrm{supp}\;h}\abs{w}^{2}
\\
    \leq &
    \sts{
            2   \sts{ \int \abs{w_{x}}^{2}h \;\dx }^{1/2 }
                \sts{ \int_{\mathrm{supp}\;h}\abs{w}^{2}h \;\dx }^{ 1/2}
            +   \int \abs{w}^{2}h_{x}\;\dx
        }
   \cdot
            \int_{\mathrm{supp}\;h}\abs{w}^{2}
\\
    \leq &\;
       C  \sts{ \int \abs{w_{x}}^{2}h \;\dx }^{ 1/2 }
            \sts{ \int_{\mathrm{supp}\;h}\abs{w}^{2} \;\dx }^{3/2}
        +   \int \abs{w}^{2}h_{x}\;\dx
        \int_{\mathrm{supp}\;h}\abs{w}^{2}\;\dx.
\end{align*}
This completes the proof.
\end{proof}

\step{Estimate of \eqref{est:fth} } First, by the fact that $h'\geq
0$ and $\sigma>0$, we have\footnote{Here we throw away the first
term for $\sigma>0$ because of $c^0_2>c^0_1>0$ and
$\omega^0_2>\omega^0_1$. In general case, we cannot use the estimate
in Lemma \ref{lem:est:fth} since the resulted decay in $t$ is just
$(t+a)^{-1}$, which is non-integral(critical).}
\begin{align}\label{est:fth:a}
- \frac18 \int
            \abs{v\sts{t}}^{4}
            \;
            h'\sts{ \frac{x-\overline{x}^{0}-\sigma t }{\sqrt{t+a}}
            } \frac{\sigma}{\sqrt{t+a}}
                    \;\dx \leq 0.
\end{align}
Second, by the boundedness of $h'$ and $h'',$ and the fact that
$\mathrm{supp}\;h'\subset\sts{-1~,~1},$
 $h'\geq 0$, Lemma \ref{lem:est:fth} and the Cauchy-Schwarz inequality, we have
\begin{align}\label{est:fth:b}
& \left|- \frac18
        \int
            \abs{v\sts{t}}^{4}
            \;
            h'\sts{ \frac{x-\overline{x}^{0}-\sigma t }{\sqrt{t+a}} }
            \frac{x-\overline{x}^{0}-\sigma t }{2\sts{t+a}^{3/2}}
        \;\dx \right|
\\
\leq &\; \frac1{16(t+a)}
        \int
            \abs{v\sts{t}}^{4}
            \;
            h'\sts{ \frac{x-\overline{x}^{0}-\sigma t }{\sqrt{t+a}} }
                  \;\dx \notag
                  \\
\leq &\; \frac{C}{t+a}
           \sts{
            \int \abs{\partial_{x}v(t)}^{2}
                h'\sts{ \frac{x-\overline{x}^{0}-\sigma t}{\sqrt{t+a}} }
            \;\dx
            }^{1/2}
            \sts{
            \int_{ \abs{x-\overline{x}^{0}-\sigma t}<\sqrt{t+a} } \abs{v(t)}^{2}
            \;\dx
            }^{3/2} \notag
\\
\notag &\quad  +
    \frac{ 1 }{ 16\sts{t+a}^{3/2} }
    \sts{
        \int_{ \abs{x-\overline{x}^{0}-\sigma t}<\sqrt{t+a} }
            \abs{v(t)}^{2}
        \;\dx
    }^{2}
\\
\leq &\;
    \frac{ 1}{8\sqrt{t+a} }
        \int \abs{\partial_{x}v(t)}^{2}
            h'\sts{ \frac{x-\overline{x}^{0}-\sigma t }{\sqrt{t+a}} }
        \;\dx \notag
        \\
    & \quad +
    \frac{C}{ \sts{t+a}^{3/2 }}
    \sts{   \int_{ \abs{x-\overline{x}^{0}-\sigma t}<\sqrt{t+a} }
                \abs{v(t)}^{2}
            \;\dx
    }^{2} +
    \frac{C}{ \sts{t+a}^{3/2 }}
    \sts{   \int_{ \abs{x-\overline{x}^{0}-\sigma t}<\sqrt{t+a} }
                \abs{v(t)}^{2}
            \;\dx
    }^{3} . \notag
\end{align}

Inserting \eqref{est:im:ok}, \eqref{est:sth:ok}, \eqref{est:fth:a}
and \eqref{est:fth:b} into \eqref{est:monf}, we can obtain the
result by the fact that $\abs{v(t)}=\abs{u(t)}$.
\end{proof}

Next, we turn to compute the local mass of the solution $u(t,x)$ to
\eqref{DNLS} in the region $\abs{x-\overline{x}^{0}-\sigma
t}<\sqrt{t+a} $.

\begin{lemma}
\label{lem:lm:sm}
 Let $u\sts{t}$ be a solution of \eqref{DNLS} satisfying the assumption of Lemma \ref{lem:tsln:decomp} on $[0, T_0]$.
 Then there exist $C$, $L_2$, $\alpha_2>0$   such that if $L\geq L_2$ and $0<\alpha<\alpha_2$ we have
\begin{align}\label{est:mass}
    \int_{ \abs{x-\overline{x}^{0}-\sigma t}<\sqrt{t+a} }
        \abs{u(t,x)}^{2}
    \;\dx \leq 2\int \abs{\varepsilon\sts{t,x}}^{2}\;\dx+ C e^{ -\theta_2\sts{L + \theta_2 t} }  < 1
\end{align}
 for
$t\in\left[0, T_{0}\right]$, where
$$0<\theta_2\triangleq \frac{1}{16}\min\left(\sigma-c^0_1,\; c^0_2-\sigma,\;
\frac{\sqrt{4\omega^0_1-(c^0_1)^2}}{4},\;
\frac{\sqrt{4\omega^0_2-(c^0_2)^2}}{4}\right)<\frac{\theta_1}{4}.$$
\end{lemma}
\begin{proof}By \eqref{tsln:sl:rough}, the second inequality in \eqref{est:mass} is obvious
for sufficient small $\alpha$ and large $L$. We now show the first
inequality in \eqref{est:mass}. It follows from Lemma
\ref{lem:tsln:decomp} that
\begin{align*}
  u\sts{t,x}=\sum_{k=1}^{2}R_{k}\sts{t,x} + \varepsilon\sts{t,x},
\end{align*}
where $ R_{k}\sts{t,x}
  = \; \varphi_{\omega_{k}\sts{t},c_{k}\sts{t}}\sts{x-x_{k}\sts{t}}e^{\i\gamma_{k}\sts{t}}.
$

First, by the Sobolev inequality and \eqref{tsln:sl:rough}, we have
\begin{align*}
  \| \varepsilon\sts{t} \|_{L^{\infty}(\R)}\leq C\normhone{ \varepsilon\sts{t} }\leq C~C_{II}~\alpha.
\end{align*}
This yields $\| \varepsilon\sts{t} \|_{L^{\infty}}<\frac{1}{2}$ for
$\alpha$ is small enough.

Next, we estimate the contribution of the traveling waves. When $x$
is in the region $\abs{x-\overline{x}^{0}-\sigma
t}<\sqrt{t+a}\leq \sqrt{t} + \frac{L}{8},$ we have
\begin{align}
\label{eq:tw1:dist} \notag
    \abs{ x-x_{1}\sts{t} }
&
=
    \abs{ \sts{ x-\overline{x}^{0}-\sigma t } -\sts{ x_{1}\sts{t} -\overline{x}^{0}-\sigma t } }
\\
\notag
&
\geq
    \abs{ x_{1}\sts{t} -\overline{x}^{0}-\sigma t } - \abs{ x-\overline{x}^{0}-\sigma t }
\\
&
\geq
    \abs{ x_{1}\sts{t} -\overline{x}^{0}-\sigma t } - \sqrt{t} - \frac{L}{8}.
\end{align}
By \eqref{tsln:sl:rough} and \eqref{tsln:para:dyn}, we have for
sufficient small $\alpha$ and sufficient large $L$ that
\begin{align*}
    \frac{d}{dt}\sts{\overline{x}^{0}+\sigma t - x_{1}\sts{t}}
& =
    \sigma - \sts{ \dot{x}_{1}\sts{t} - c_{1}\sts{t} } - c_{1}\sts{t}
\\
&
\geq
    \sigma - C_{II}\left(C_{II}\alpha +
                e^{ -\theta_{1}\sts{ \frac{L}{2} + \theta_1 t } } \right) - c_{1}^0 - C_{II}~\alpha
\\
&
\geq
    \frac{\sigma - c_{1}^0}{2},
\end{align*}
and so,
\begin{align*}
    \overline{x}^{0}+\sigma t - x_{1}\sts{t}
 \geq
    \overline{x}^{0} - x_{1}\sts{0} + \frac{ \sigma - c_{1}^0}{2}\;t
\geq
    \frac{L}{4} + \frac{\sigma - c_{1}^0}{2}\;t.
\end{align*}
Now inserting the above estimate into \eqref{eq:tw1:dist},
 we obtain for $\alpha$ small enough and $L$ large enough that
\begin{align}
    \abs{ x-x_{1}\sts{t} }
 & \geq
     \overline{x}^{0}  +\sigma t - x_{1}\sts{t}   - \sqrt{t} -
     \frac{L}{8} \notag
\\
&
\geq
       \frac{L}{16} + \frac{\sigma - c_{1}^0}{4}\;t
    +
    \sts{ \frac{\sigma - c_{1}^0}{4}\;t - \sqrt{t} + \frac{L}{16}} \notag
\\
&
\geq
    \frac{L}{16} + \frac{\sigma - c_{1}^0}{4}\;t \geq \frac{L}{16} + 4 \theta_2 t
    , \label{efreg:h}
\end{align}
which implies that
\begin{align*}
  \abs{
        R_{1}\sts{t,x}
    }
    \leq
    C e^{ -\frac{ \sqrt{ 4\omega_{1}\sts{t}- c_{1}^{2}\sts{t}  }}{2}\abs{ x-x_{1}\sts{t} } }
    \leq
    C~
    e^{
        -\frac{ \sqrt{ 4\omega_{1}^0- \sts{c_{1}^0}^{2} } }{4}\sts{\frac{L}{16}
        + 4\theta_2 t}}
    \leq
    C~    e^{ - 16 \theta_2 \sts{\frac{L}{16}+ 4\theta_2 t} }.
\end{align*}
In a similar way, we have
\begin{align*}
  \abs{
        R_{2}\sts{t,x}
    }
    \leq
    C~    e^{ - 16 \theta_2 \sts{\frac{L}{16}+ 4\theta_2 t} }.
\end{align*}
Thus, we have
\begin{align*}
    \int_{ \abs{x-\overline{x}^{0}-\sigma t}<\sqrt{t+a} }
        \abs{u(t,x)}^{2}
    \;\dx
\leq & \; 2
        \int_{ \abs{x-\overline{x}^{0}-\sigma t}<\sqrt{t+a} }
            \abs{\varepsilon\sts{t,\,x}}^{2}
        \;\dx
       +  C~e^{
        -  \theta_2 \sts{L
        +  \theta_{2} t}
        }.
\end{align*}
This completes the proof.
\end{proof}
As a consequence of \eqref{tsln:sl:rough}, Lemma \ref{lem:monf:dert}
and Lemma \ref{lem:lm:sm}, we have
\begin{proposition}\label{prop:monf}  Let $u\sts{t}$ be a solution of \eqref{DNLS} satisfying the assumptions
 of Lemma \ref{lem:tsln:decomp} on $[0, T_0]$,
  then there exist $C$, $L_2$, $\alpha_2>0$   such that if $L\geq L_2$ and $0<\alpha<\alpha_2$ we have
    \begin{align*}
        \mathfrak{Q}\sts{t}-\mathfrak{Q}\sts{0}
        \leq
        \frac{C}{L}\sup_{0<s<t}\int\abs{\varepsilon\sts{s, x}}^{2} \; \dx + C~e^{-\theta_2 L} \quad \text{for}\; t\in\left[0, T_{0}\right].
    \end{align*}
    \end{proposition}

\subsection{Monotone results for the lines $x=\overline{x}^{0}+\sigma_{0,\pm} t $ and $x=\overline{x}^{0}+\sigma_{\pm, 0} t $}
In order to get the refined estimates of $|\omega_k(t)-\omega_k(0)|$
and $|c_k(t)-c_k(0)|$ for $k=1,2$ in next section, we introduce the
following monotone functionals for the different lines and
characterize its property here. First we denote\footnote{It is easy
to verify that $c_{1}^{0}<\sigma_{\pm,0},\; \sigma_{0,\pm}
<c_{2}^{0}$ by assumption $(c)$ in Theorem \ref{thm:stab:tsln}.}
\begin{align}\label{tw:relspds}
&
    \sigma_{+,0}
    =   \sigma_{0,-}=
    \frac{\sigma+c_2^0}{2},
\quad
    \sigma_{-,0}=  \sigma_{0,+}
    =
    \frac{\sigma+\max\{c_1^0, 0\}}{2},
\end{align}
and define
\begin{align*}
    \mathfrak{h}_{\pm,0} \sts{t,x}
     =h\sts{ \frac{x-\overline{x}^{0}-\sigma_{\pm,0} t
    }{\sqrt{t+a}}},\quad
    \mathfrak{h}_{0,\pm} \sts{t,x}
    =h\sts{ \frac{x-\overline{x}^{0}-\sigma_{0,\pm} t }{\sqrt{t+a}}},
\end{align*}
and
\begin{align*}
  \frac{1}{c_2(0)-c_1(0)}  \mathfrak{Q}_{+,0}\sts{t}
=
&
    \frac{\sigma_{+,0}}{2}
        \int
            \frac12 \abs{u\sts{t,x}}^{2}\mathfrak{h}_{+,0}\sts{t,x}
        \;\dx
\\
&
    +
       \int \left(-   \frac{1}{2} \Im \left(
            \overline{u\sts{t,x}}\partial_{x}u\sts{t,x}\right) + \frac18 \abs{u\sts{t,x}}^{4} \right)\mathfrak{h}_{+,0}\sts{t,x}
        \;\dx,
\end{align*}
\begin{align*}
   \frac{1}{c_2(0)-c_1(0)}   \mathfrak{Q}_{-,0}\sts{t}
= &
    \frac{\sigma_{-,0}}{2}
        \int
            \frac12 \abs{u\sts{t,x}}^{2}\mathfrak{h}_{-,0}\sts{t,x}
        \;\dx
\\
&
    +
       \int \left(-   \frac{1}{2} \Im \left(
            \overline{u\sts{t,x}}\partial_{x}u\sts{t,x}\right) + \frac18 \abs{u\sts{t,x}}^{4} \right)\mathfrak{h}_{-,0}\sts{t,x}
        \;\dx,
\end{align*}
\begin{align*}
    \frac{1}{\omega_2(0)-\omega_1(0)}\mathfrak{Q}_{0,+}\sts{t}
= &
            \int
            \frac12 \abs{u\sts{t,x}}^{2}\mathfrak{h}_{0,+}\sts{t,x}
        \;\dx
\\
&
    + \frac{2}{\sigma_{0,+}}
       \int \left(-   \frac{1}{2} \Im \left(
            \overline{u\sts{t,x}}\partial_{x}u\sts{t,x}\right) + \frac18 \abs{u\sts{t,x}}^{4} \right)\mathfrak{h}_{0,+}\sts{t,x}
        \;\dx,
\end{align*}
and
\begin{align*}
    \frac{1}{\omega_2(0)-\omega_1(0)} \mathfrak{Q}_{0,-}\sts{t}
= &
            \int
            \frac12 \abs{u\sts{t,x}}^{2}\mathfrak{h}_{0,-}\sts{t,x}
        \;\dx
\\
&
    + \frac{2}{\sigma_{0,-}}
       \int \left(-   \frac{1}{2} \Im \left(
            \overline{u\sts{t,x}}\partial_{x}u\sts{t,x}\right) + \frac18 \abs{u\sts{t,x}}^{4} \right)\mathfrak{h}_{0,-}\sts{t,x}
        \;\dx.
\end{align*}

By the similar arguments as that in the proof of Proposition
\ref{prop:monf}, we have
\begin{corollary}\label{coro:monf} Let $u\sts{t}$ be a solution of \eqref{DNLS}
satisfying the assumption of Lemma \ref{lem:tsln:decomp} on $[0,
T_0]$. Then there exist $C$, $L_3$, $\alpha_3>0$   such that if
$L\geq L_3$ and $0<\alpha<\alpha_3$ we have for $t\in\left[0,
T_{0}\right],$
    \begin{align*}
        &
        \mathfrak{Q}_{\pm,0}\sts{t}-\mathfrak{Q}_{\pm,0}\sts{0}
        \leq
        \frac{C}{L}\sup_{0<s<t}\int\abs{\varepsilon\sts{s, x}}^{2}\dx + C~e^{-\theta_3 L},
        \\
        &
        \mathfrak{Q}_{0,\pm}\sts{t}-\mathfrak{Q}_{0,\pm}\sts{0}
        \leq
        \frac{C}{L}\sup_{0<s<t}\int\abs{\varepsilon\sts{s, x}}^{2}\dx + C~e^{-\theta_3
        L},
    \end{align*}
    where $$0<\theta_3=\frac{1}{16}\min\left(\sigma_{\pm, 0}-c^0_1,\;
c^0_2-\sigma_{\pm, 0},\; \frac{\sqrt{4\omega^0_1-(c^0_1)^2}}{4},\;
\frac{\sqrt{4\omega^0_2-(c^0_2)^2}}{4}\right)\leq \theta_2.$$
\end{corollary}

\begin{remark}
   The introduction of the constant ``a'' aims to ensure that the variations of $\mathfrak{Q}\sts{t}$, $\mathfrak{Q}_{\pm,0}\sts{t}$ and $\mathfrak{Q}_{0,\pm}\sts{t}$ in Proposition \ref{prop:monf} and Corollary \ref{coro:monf} can be small by comparing with the $L^2$ estimate of the remainder term $\varepsilon\sts{t}$ in the interval $[0,T_0]$, up to small errors $e^{-\theta_2 L}$ and  $e^{-\theta_3 L}$.
\end{remark}

\section{Stability of the sum of two traveling waves}\label{sect:stab:tsln}
In this section, we show the result in Theorem \ref{thm:stab:tsln}.
Let $\omega^{0}_k$ and $c^{0}_k$ satisfy the assumptions in Theorem
\ref{thm:stab:tsln}. Fix
\begin{align*}
\theta_0 = \min(\theta_1, \theta_2, \theta_3).
\end{align*}
Let $\alpha_0$ be defined by Lemma \ref{lem:tslnst:decomp} and
$A_0>2$, $\delta_0=\delta_0(A_0), L_0=L_0(A_0)$ be chosen later.
Suppose that $u\sts{t}$ is the solution of \eqref{DNLS} with initial
data
$u_{0}\in\mathcal{U}\sts{\delta~,~L,~\omega^{0}_1~,~c^{0}_1,~\omega^{0}_2~,~c^{0}_2}.$
Then there exist $x^{0}_k\in\R$ and $\gamma^{0}_k\in\R$, $k=1,2$
such that
\begin{align*}
    \normhone{ u_{0}(\cdot) - \sum^2_{k=1}\varphi_{\omega^{0}_k,c^{0}_k}\sts{\cdot-x^{0}_k}e^{\i\gamma^{0}_k} }<\delta,
\end{align*}
where $\delta<\alpha_0$ and $x^0_2-x^0_1 > L$. Now we define
\begin{align*}
       T^*=\sup\left\{t\geq 0, \;
       \sup_{\tau\in[0,t]}\inf_{\substack{x_{2}^0-x_{1}^0 > \frac{L}{2} \\ \gamma^0_k\in\R,~ k=1,2}}
       \normhone{ u(\tau,\cdot)-\sum^{2}_{k=1}\varphi_{\omega^{0}_k,c^{0}_k}\sts{\cdot -x^0_k}e^{\i\gamma^0_k}}
    \leq A_0\left(\delta+e^{-\theta_0 \frac{L}{2}}\right)\right\}.
\end{align*}
By the continuity of $u(t)$ in $H^1$, we know that $T^*>0$. In order
to prove Theorem \ref{thm:stab:tsln}, it suffices to show
$T^*=+\infty$ for some $A_0>2$, $\delta_0>0$, and $L_0$.

We argue with contradiction. Suppose that $T^*<+\infty$, we know
that for any $t\in [0, T^*]$, there exist $(x^0_k(t),
\gamma^0_k(t))\in \R^2$, $k=1, 2$ such that $x^0_2(t)\geq
x^0_1(t)+\frac{L}{2}$ and
\begin{align*}
      \normhone{ u(t,\cdot)-\sum^2_{k=1}\varphi_{\omega^{0}_k,c^{0}_k}\sts{\cdot-x^{0}_k(t)}e^{\i\gamma^{0}_k(t)}}   \leq A_0\left(\delta +e^{-\theta_0
      \frac{L}{2}} \right).
\end{align*}

\noindent \textbf{Step 1: Decomposition of the solution $u(t)$
around the sum of two traveling waves.} Let $\widetilde{L}_{0}>0$ be
determined by Lemma \ref{lem:tslnst:decomp}, and $L_2, L_3$ be
determined by Proposition \ref{prop:monf} and Corollary
\ref{coro:monf}, and choose $\delta_0>0$ small enough and $L_{0}$
large enough, such that for $\delta<\delta_0$ and
$L>L_0(A_0)>\max\ltl{ 2\widetilde{L}_{0}, L_2, L_3}$,
\begin{align*}
    A_{0}~\sts{ \delta + e^{ - \theta_{0} \frac{L}{2} }
    }<\alpha_{0}.
\end{align*}
Now by Lemma \ref{lem:tsln:decomp}, there exists a unique
$\mathcal{C}^{1}$ functions $$
    \overrightarrow{\mathbf{q}}\sts{t}~
    =\,
    \sts{
        \omega_{1}\sts{t},c_{1}\sts{t},x_{1}\sts{t},\gamma_{1}\sts{t},
        \omega_{2}\sts{t},c_{2}\sts{t},x_{2}\sts{t},\gamma_{2}\sts{t}
    }
$$
on $[0~,~T^{*}]$ with
 $\left(c_{k}\sts{t}\right)^2 < 4\omega_{k}\sts{t}$, $k=1, 2$ such that
the remainder term
\begin{align}
\label{tsln:rem}
    \varepsilon
    \sts{t\,,\,x}
    = u\sts{t\,,\,x} - \sum^2_{k=1}R_{k}\sts{t\,,\,x}
\end{align}
 has the orthogonality conditions \eqref{tsl:orth:nondeg}-\eqref{tsl:orth:sym} on
$\left[0~,~T^{*}\right],$
for $k=1,~2$, where $ R_{k}\sts{t\,,\,x}
    \triangleq
    \varphi_{\omega_{k}\sts{t},c_{k}\sts{t}}\sts{x-x_{k}\sts{t}}e^{\i\gamma_{k}\sts{t}}$.
      Moreover, we have
\begin{align}
\label{tsln:para:rest}
      \normhone{\varepsilon\sts{t}}
      + &
      \sum_{k=1}^{2}\left(\abs{ \omega_k\sts{t}-\omega_{k}^{0} }
      +
      \abs{ c_k\sts{t}-c_{k}^{0} } \right)
      \leq
     C_{II}A_{0}\sts{ \delta + e^{ -  \theta_{0}  \frac{L}{2} } }, \\
          \label{tsln:cent:dist}
      & \left|x_2(t)-x_1(t)\right| >\;  \frac{L}{2} + \theta_0 t.
\end{align}
for any $t\in \left[0~,~T^{*}\right],$ and
\begin{align}
\label{tsln:para:iest}
      \normhone{\varepsilon\sts{0}}
      + & \;
      \sum_{k=1}^{2}\left(\abs{ \omega_k\sts{0}-\omega_{k}^{0} }
      +
    \abs{ c_k\sts{0}-c_{k}^{0} } \right)
      \leq
     C_{II}\delta. \end{align}
If necessary, we can take $\delta_0$ sufficiently small and $L_0$
sufficiently large to ensure that $C_{II}A_{0}\sts{ \delta + e^{ -
\theta_{0}  \frac{L}{2} } }<1$.

\noindent \textbf{Step 2: Action functional and refined estimate of
$\varepsilon(t)$ and $\left| \mathfrak{Q}\sts{t} -
\mathfrak{Q}\sts{0} \right|$.} For this purpose, we introduce the
following localized action functional,
\begin{align}
\label{acf:tsln}
    \mathfrak{E}\sts{u\sts{t}}
=
&
    E\sts{u\sts{t}}+\mathfrak{F}\sts{t},
\end{align}
where $\mathfrak{F}(t)$ is defined by \eqref{acf:monf}, which is not
conserved along the flow \eqref{DNLS}. Nevertheless from Proposition
\ref{prop:monf}, we know that it is almost conserved or
controllable. First, by the linearized argument and the orthogonal
structures \eqref{tsl:orth:nondeg} and \eqref{tsl:orth:sym}, we
have
\begin{lemma}
\label{lem:tsln:acfln}For $t\in [0, T^*]$, we have the following expansion formula.
\begin{align*}
  \mathfrak{E}\sts{u\sts{t}}
= &\; \sum^2_{k=1}
    \mathfrak{J}_{\omega_{k}\sts{0},c_{k}\sts{0}}\sts{R_{k}\sts{0}}
    +
    \mathfrak{H}\bilin{\varepsilon\sts{t}}{\varepsilon\sts{t}}
\\
&    +
    \bigo{ \sum^{2}_{k=1}\left(\abs{ \omega_{k}\sts{t} - \omega_{k}\sts{0}
    }^{2}+ \abs{ c_{k}\sts{t} - c_{k}\sts{0} }^{2}\right)
    }
\\
&
    +
    \normhone{\varepsilon\sts{t}}^{2}\beta\sts{ \normhone{\varepsilon\sts{t}} }
    +
    \bigo{ e^{ -8\theta_0 \left(\frac{L}{2}+ 8 \theta_0 t\right) }}
\end{align*}
where
\begin{align*}
   & \;  \mathfrak{H}\bilin{\varepsilon\sts{t}}{\varepsilon\sts{t}}
\\ = &\;   \frac{1}{2}\int
\abs{\partial_{x}\varepsilon\sts{t}}^{2}\;\dx
    -   \frac{3}{32} \sum^2_{k=1}
            \int \abs{ R_{k}\sts{t} }^{4} \abs{ \varepsilon\sts{t}
            }^{2}\; dx
           -\frac38 \sum^2_{k=1}
            \int \abs{ R_{k}\sts{t} }^{2}
                \left[ \Re{\sts{ \overline{R}_{k}\varepsilon}\sts{t}}
                \right]^{2}\; dx
\\
&
    +
    \frac{\omega_{1}\sts{t}}{2}
    \int
        \abs{
            \varepsilon\sts{t}
        }^{2}\mathfrak{g}\sts{t}
    \;\dx
    +
    \frac{\omega_{2}\sts{t}}{2}
    \int
        \abs{
            \varepsilon\sts{t}
        }^{2}\mathfrak{h}\sts{t}
    \;\dx
\\
&
    -
    \frac{c_{1}\sts{t}}{2}
    \Im\int
        \sts{\overline{\varepsilon}\partial_{x}\varepsilon}\sts{t}
        \mathfrak{g}\sts{t}\;\dx
    -
    \frac{c_{2}\sts{t}}{2}
    \Im\int
        \sts{ \overline{\varepsilon} \partial_{x}\varepsilon }\sts{t}
        \mathfrak{h}\sts{t}\;\dx
\\
&
    + \sum^2_{k=1}
    \frac{c_{k}\sts{t}}{4}
    \int \abs{R_{k}\sts{t}}^{2}\abs{\varepsilon\sts{t}}^{2} \;\dx
    + \sum^2_{k=1}
    \frac{c_{k}\sts{t}}{2}
    \int\left[\Re\sts{\overline{R}_{k}\varepsilon}\sts{t}\right]^{2}\;\dx.
\end{align*}
\end{lemma}
\begin{proof}
See the proof in Appendix \ref{app:tsln:acfln}.
\end{proof}

On one hand, by Lemma \ref{lem:tsln:acfln}, we have
\begin{align}
  \mathfrak{E}\sts{u\sts{0}}
= & \sum^2_{k=1}
    \mathfrak{J}_{\omega_{k}\sts{0},c_{k}\sts{0}}\sts{R_{k}\sts{0}}
        +
    \mathfrak{H}\bilin{\varepsilon\sts{0}}{\varepsilon\sts{0}}
    \notag
\\
&
    +
    \normhone{\varepsilon\sts{0}}^{2}\beta\sts{ \normhone{\varepsilon\sts{0}} }
    +
    \bigo{ e^{ - 8\theta_{0} \frac{L}{2} } }, \label{acf:tsln:id}
\end{align}
and
\begin{align}
  \mathfrak{E}\sts{u\sts{t}}
= &
\sum^2_{k=1}\mathfrak{J}_{\omega_{k}\sts{0},c_{k}\sts{0}}\sts{R_{k}\sts{0}}
    +
    \mathfrak{H}\bilin{\varepsilon\sts{t}}{\varepsilon\sts{t}}
    \notag
\\
&
    +
    \bigo{ \sum^2_{k=1}\left( \abs{ \omega_{k}\sts{t} - \omega_{k}\sts{0} }^{2}
     + \abs{ c_{k}\sts{t} - c_{k}\sts{0} }^{2}
    \right)} \notag
\\
&
    +
    \normhone{\varepsilon\sts{t}}^{2}\beta\sts{ \normhone{\varepsilon\sts{t}} }
    +
    \bigo{ e^{ -8 \theta_0\sts{ \frac{L}{2} +8\theta_0 t } } }.
    \label{acf:tsln}
\end{align}
On the other hand, by \eqref{acf:monf}, the conservation laws of
mass, momentum and energy, we have
\begin{align}\label{acf:tsln:alc}
  \mathfrak{E}\sts{u\sts{t}} -   \mathfrak{E}\sts{u\sts{0}}  = \mathfrak{Q}\sts{t} - \mathfrak{Q}\sts{0} .
\end{align}
Combining \eqref{acf:tsln:id}, \eqref{acf:tsln} and
\eqref{acf:tsln:alc}, we have
\begin{align}
    \mathfrak{H}\bilin{\varepsilon\sts{t}}{\varepsilon\sts{t}} + &
    \normhone{\varepsilon\sts{t}}^{2}\beta\sts{ \normhone{\varepsilon\sts{t}}
    } \notag
 \\
\leq &\; \mathfrak{Q}\sts{t} - \mathfrak{Q}\sts{0}
    +
    \mathfrak{H}\bilin{\varepsilon\sts{0}}{\varepsilon\sts{0}}
    +
    \normhone{\varepsilon\sts{0}}^{2}\beta\sts{ \normhone{\varepsilon\sts{0}}
    } \notag
\\
&
    +
    \bigo{ \sum^2_{k=1}\left(\abs{ \omega_{k}\sts{t} - \omega_{k}\sts{0} }^{2}+ \abs{ c_{k}\sts{t} - c_{k}\sts{0} }^{2}
    \right)}
    +
    \bigo{ e^{ - 4 \theta_0 L } } \notag
\\
\leq &\; \mathfrak{Q}\sts{t} - \mathfrak{Q}\sts{0}
        + C
    \normhone{\varepsilon\sts{0}}^{2}
     +
    \bigo{ e^{ -4 \theta_0 L } } \notag
\\
&
    +
    \bigo{ \sum^2_{k=1}\left(\abs{ \omega_{k}\sts{t} - \omega_{k}\sts{0} }^{2}+ \abs{ c_{k}\sts{t} - c_{k}\sts{0} }^{2}
    \right)},  \label{est:quad term}
\end{align}
where we use \eqref{tsln:para:iest} and the fact that $
\mathfrak{H}\bilin{\varepsilon\sts{0}}{\varepsilon\sts{0}}
    \leq C \normhone{ \varepsilon\sts{0} }^{2}
$ in the second inequality.

Now by the orthogonal structures \eqref{tsl:orth:nondeg},
\eqref{tsl:orth:sym} and the standard localized argument, we have
\begin{lemma}
\label{lem:tsln:coer}
  There exists $C_{1}>0$ such that
  \begin{align*}
    \mathfrak{H}\bilin{\varepsilon\sts{t}}{\varepsilon\sts{t}}
    \geq
   C_{1}\normhone{ \varepsilon\sts{t} }^{2}\,.
  \end{align*}
\end{lemma}
\begin{proof} See the proof in Appendix \ref{app:tsln:coer}.
\end{proof}

Now by \eqref{tsln:para:rest}, \eqref{est:quad term} and the above
lemma, we have
\begin{align}
    \frac{ C_{1} }{ 2 }\normhone{ \varepsilon\sts{t} }^{2}
\leq
&\;
    \mathfrak{Q}\sts{t} - \mathfrak{Q}\sts{0}
    +
    C\; \normhone{ \varepsilon\sts{0} }^{2}
    +
    C\; e^{ - 4\theta_0 L } \notag
\\
 &  +
    \bigo{ \sum^2_{k=1}\left(\abs{ \omega_{k}\sts{t} - \omega_{k}\sts{0} }^{2}+ \abs{ c_{k}\sts{t} - c_{k}\sts{0} }^{2}
    \right)}. \label{tsln:rem_mon}
\end{align}
By Proposition \ref{prop:monf}, we have
\begin{align}
    \frac{ C_{1} }{ 2 }\normhone{ \varepsilon\sts{t} }^{2}
\leq &\;
\frac{C}{L}\sup_{0<s<t}\int\abs{\varepsilon\sts{s}}^{2}
    +
    C\; \normhone{ \varepsilon\sts{0} }^{2}
    +
    C\; e^{ - \theta_0 L } \notag
\\
 &  +
    \bigo{ \sum^2_{k=1}\left(\abs{ \omega_{k}\sts{t} - \omega_{k}\sts{0} }^{2}+ \abs{ c_{k}\sts{t} - c_{k}\sts{0} }^{2}
    \right)}, \label{tsln:rem:ref}
\end{align}
and
\begin{align}
    \left| \mathfrak{Q}\sts{t} - \mathfrak{Q}\sts{0} \right|
\leq &\;
\frac{C}{L}\sup_{0<s<t}\int\abs{\varepsilon\sts{s}}^{2}
    +
    C\; \normhone{ \varepsilon\sts{0} }^{2}
    +
    C\; e^{ - \theta_0 L } \notag
\\
 &  +
    \bigo{ \sum^2_{k=1}\left(\abs{ \omega_{k}\sts{t} - \omega_{k}\sts{0} }^{2}+ \abs{ c_{k}\sts{t} - c_{k}\sts{0} }^{2}
    \right)}. \label{tsln:mon:ref}
\end{align}

\noindent \textbf{Step 3: Refined estimates of $\abs{ c_{k}\sts{t} -
c_{k}\sts{0} }$ and $\abs{ \omega_{k}\sts{t} - \omega_{k}\sts{0}
}$.}

By the definitions of $\mathfrak{Q}_{+,0}\sts{t}$ and
$\mathfrak{Q}\sts{t}$, we have
\begin{align}
   \frac{ \mathfrak{Q}_{+,0}\sts{t} - \mathfrak{Q}\sts{t}  }{c_2(0)-c_1(0)}
= &\;
    \frac{\sigma_{+,0}-\sigma}{2}
        \int
           \frac{1}{2} \abs{u\sts{t,x}}^{2}\mathfrak{h}_{+,0}\sts{t,x}
        \;\dx \notag
\\
\notag
&
    +
   \frac{\sigma}{2}
        \int
           \frac{1}{2} \abs{u\sts{t,x}}^{2}
            \left[ \mathfrak{h}_{+,0}\sts{t,x} - \mathfrak{h}\sts{t,x} \right]
        \;\dx
\\
&
   +
        \int \sts{
            -\frac{1}{2}
            \Im \left(\overline{u}\partial_{x}u \right)+
            \frac18 \abs{u}^{4}
            }\sts{t,x}
            \big[ \mathfrak{h}_{+,0}\sts{t,x} - \mathfrak{h}\sts{t,x} \big]
        \;\dx \label{est:k_mass:form}
\end{align}
On one hand, by \eqref{efreg:h}, we know that $ \mathfrak{h}_{+,0} -
\mathfrak{h} \equiv 0 $ for $x<x_{1}\sts{t} +   L/16$ or
$x>x_{2}\sts{t} -  L/16,$ which implies that
\begin{align*}
    \int
        \sts{\; \abs{ R\sts{t,x} } + \abs{ \partial_{x}R\sts{t,x} } \;}
        \abs{ \mathfrak{h}_{+,0}\sts{t,x} - \mathfrak{h}\sts{t,x} }
    \;\dx
    \leq C e^{-\theta_0 L }.
\end{align*}
This together with the Cauchy-Schwarz inequality yields
\begin{align}
\label{est:k_mass:sm}
    \abs{
        \int
            \left(\abs{u\sts{t}}^{2}+ \abs{ \Im \left(
            \overline{u\sts{t}}\partial_{x}u\sts{t}\right)}+  \abs{u\sts{t}}^{4} \right)
            \left[ \mathfrak{h}_{+,0}\sts{t} - \mathfrak{h}\sts{t} \right]
        \;\dx
    }
\leq C \normhone{\varepsilon\sts{t} }^{2} + Ce^{ - \theta_0 L}.
\end{align}
On the other hand, it follows from \eqref{tsl:orth:nondeg} that
\begin{align}
\label{est:k_mass:main} & \abs{
    \int
        \abs{u\sts{t,x}}^{2}
        \mathfrak{h}_{+,0}\sts{t,x}
    \;\dx
    -
    \int |R_{2}\sts{t,x}|^2\;\dx
} \leq C \normhone{\varepsilon\sts{t}}^{2} + C~e^{ - \theta_0
L}.
\end{align}
Thus, Inserting \eqref{est:k_mass:sm} and \eqref{est:k_mass:main}
into \eqref{est:k_mass:form}, we have
\begin{align*}
    \abs{
        \frac{ \mathfrak{Q}_{+,0}\sts{t} - \mathfrak{Q}\sts{t}  }{c_2(0)-c_1(0)}
-
    \frac{\sigma_{+,0}-\sigma}{2} M\sts{ R_{2}\sts{t}}
    }
\leq C \normhone{\varepsilon\sts{t}}^{2} + C~e^{ - \theta_0 L}.
\end{align*}
Particularly, we have for $t>0$ and $t=0$ that
\begin{align*}
      -\left(   \frac{ \mathfrak{Q}_{+,0}\sts{t} - \mathfrak{Q}\sts{t}  }{c_2(0)-c_1(0)}
-
    \frac{\sigma_{+,0}-\sigma}{2} M\sts{ R_{2}\sts{t}} \right)
\leq C \normhone{\varepsilon\sts{t}}^{2} + C~e^{ - \theta_0 L},
\end{align*}
and
\begin{align*}
         \frac{ \mathfrak{Q}_{+,0}\sts{0} - \mathfrak{Q}\sts{0}  }{c_2(0)-c_1(0)}
-
    \frac{\sigma_{+,0}-\sigma}{2} M\sts{ R_{2}\sts{0}}
\leq C \normhone{\varepsilon\sts{0}}^{2} + C~e^{ - \theta_0 L}.
\end{align*}
Summing up the above two inequalities, we obtain from
\eqref{tsln:mon:ref} and Corollary \ref{coro:monf} that
\begin{align}
       \frac{\sigma_{+,0}-\sigma}{2}
               & \; \sts{ M\sts{R_{2}(t)} - M \sts{R_2(0)} } \notag
\\
\leq & \;\frac{ \mathfrak{Q}_{+,0}\sts{t} -  \mathfrak{Q}_{+,0}\sts{0}}{c_2(0)-c_1(0)} -
  \frac{ \mathfrak{Q}\sts{t}  - \mathfrak{Q}\sts{0}  }{c_2(0)-c_1(0)}
   +
    C\sup_{s\in\left[0\;,\;t\right]}\normhone{\varepsilon\sts{t}}^{2}
    +
    Ce^{ - \theta_0 L} \notag
\\
    \leq & \;
  \frac{ \mathfrak{Q}_{+,0}\sts{t} -  \mathfrak{Q}_{+,0}\sts{0}}{c_2(0)-c_1(0)} +\left|
  \frac{ \mathfrak{Q}\sts{t}  - \mathfrak{Q}\sts{0}  }{c_2(0)-c_1(0)} \right|
   +
    C\sup_{s\in\left[0\;,\;t\right]}\normhone{\varepsilon\sts{t}}^{2}
    +
    Ce^{ - \theta_0 L} \notag
\\
\leq &\;
    C\sup_{s\in\left[0\;,\;t\right]}\normhone{\varepsilon\sts{s}}^{2}
    +
    Ce^{ - \theta_0 L}. \label{est:k_mass:up}
\end{align}

By the similar arguments as above, we can show that
\begin{align*}
    \abs{
         \frac{ \mathfrak{Q}\sts{t} - \mathfrak{Q}_{-,0}\sts{t}  }{c_2(0)-c_1(0)}
-
    \frac{\sigma -\sigma_{-,0}}{2} M\sts{ R_{2}\sts{t}}
    }
\leq C \normhone{\varepsilon\sts{t}}^{2} + C~e^{ - \theta_0 L},
\end{align*}
and
\begin{align}
\label{est:k_mass:low}
       - \frac{\sigma -\sigma_{-,0}}{2}
        \sts{
             M\sts{ R_{2}(t)}
            -
            M\sts{R_{2}(0)}
        }
\leq\;
    C\sup_{s\in\left[0\;,\;t\right]}\normhone{\varepsilon\sts{s}}^{2}
    +
    Ce^{ - \theta_0 L}.
\end{align}
Combining \eqref{est:k_mass:up} with \eqref{est:k_mass:low}, we
obtain
\begin{align}
\label{est:k_mass:bdd}
    \abs{
            M\sts{ R_{2}(t)}
            -
            M\sts{ R_{2}(0)}
        }
    \leq \;
    C\sup_{s\in\left[0\;,\;t\right]}\normhone{\varepsilon\sts{s}}^{2}
    +
    Ce^{ - \theta_0 L}.
\end{align}
By the similar arguments as above and the definitions of
$\mathfrak{Q}_{0, \pm}\sts{t}$ and $\mathfrak{Q}\sts{t}$, we can
show that
\begin{align}
\label{est:k_mom:bdd}
    \abs{
            \momentum{R_{2}\sts{t}}
            -
            \momentum{R_{2}\sts{0}}
        }
    \leq
    C\sup_{s\in\left[0\;,\;t\right]}\normhone{\varepsilon\sts{t}}^{2}
    +
    Ce^{ - \theta_0 L}.
\end{align}

Now by the mass conservation and the orthogonality condition
\eqref{tsl:orth:nondeg}, we have
\begin{align}
\label{est:tsln:massdyn}
    \abs{ \sum_{k=1}^{2}\big( M\sts{ R_{k}(t)}
            -
            M\sts{ R_{k}(0)} \big)
    }
    \leq\;
    C\sup_{s\in\left[0\;,\;t\right]}\normhone{\varepsilon\sts{t}}^{2}
    +
    Ce^{ - \theta_0 L},
\end{align}
where we used the fact that $|x_2(t)-x_1(t)|> \frac{L}{2}+ \theta_0
t$. Thus, by \eqref{est:k_mass:bdd} and \eqref{est:tsln:massdyn}, we
obtain
\begin{align}
\label{est:f_mass:bdd}
    \abs{
           M( R_{1}\sts{t} )
            -
            M( R_{1}\sts{0})
        }
    \leq\;
    C\sup_{s\in\left[0\;,\;t\right]}\normhone{\varepsilon\sts{t}}^{2}
    +
    Ce^{ - \theta_0 L}.
\end{align}
In a similar way, we have
\begin{align}
\label{est:f_mom:bdd}
    \abs{
            \momentum{R_{1}\sts{t}}
            -
            \momentum{R_{1}\sts{0}}
        }
    \leq
    C\sup_{s\in\left[0\;,\;t\right]}\normhone{\varepsilon\sts{t}}^{2}
    +
    Ce^{ - \theta_0 L}.
\end{align}

By \eqref{tsln:sl:rough} and the non-degenerate condition
$$\det  d''\sts{~\omega^0_k,~c^0_k~}  <0$$ for $k=1, 2$. We can now
refine the estimates of $\abs{ \omega_k\sts{t}-\omega_k\sts{0}
}+\abs{ c_k\sts{t}-c_k\sts{0} }.$
\begin{lemma}
\label{lem:k_para:ref}
  For any $t\in\left[0~,~T^*\right],$ we have for $k=1, 2$
\begin{align*}
  \abs{\omega_{k}\sts{t}- \omega_{k}\sts{0}} + \abs{c_{k}\sts{t}- c_{k}\sts{0}}
    \leq \;
    C\sup_{ s\in\left[0\;,\;t\right] }\normhone{\varepsilon\sts{s}}^{2}
    +
    Ce^{ - \theta_0 L}.
\end{align*}
\end{lemma}

\begin{proof}
On one hand, by \eqref{tsln:para:rest} and
\begin{align*}
    \det d''\sts{\omega_{k}^{0}~,~c_{k}^{0}}<0,
\end{align*}
we have for sufficiently small $\alpha_0$ that
\begin{align}
    2 \det d''\sts{\omega_{k}^{0},~c_{k}^{0}}
    <
    \det d''\sts{\omega_{k}\sts{t},~c_{k}\sts{t}}
    <
    \frac{1}{2}\det d''\sts{\omega_{k}^{0},~c_{k}^{0}}<0.
    \label{est:k_para:jac}
\end{align}
On the other hand, by the Taylor formula, we have
\begin{align}
 & \begin{pmatrix}
        \mass{R_{k}\sts{t}}-\mass{R_{k}\sts{0}}\\
        \momentum{R_{k}\sts{t}}-\momentum{R_{k}\sts{0}}\\
    \end{pmatrix}
    =   \begin{pmatrix}
        \mass{\varphi_{\omega_k(t), c_k(t)}}-\mass{\varphi_{\omega_k(0), c_k(0)}}\\
        \momentum{\varphi_{\omega_k(t), c_k(t)}}-\momentum{\varphi_{\omega_k(0), c_k(0)}}\\
    \end{pmatrix} \notag
 \\
    = &\;   d''\sts{\omega_{k}\sts{0}~,~c_{k}\sts{0}}
    \begin{pmatrix}
        \omega_{k}\sts{t}-\omega_{k}\sts{0} \\
        c_{k}\sts{t}-c_{k}\sts{0} \\
    \end{pmatrix}
  +
    \bigo{
        \abs{ \omega_{k}\sts{t}-\omega_{k}\sts{0} }^{2}
        +
        \abs{ c_{k}\sts{t}-c_{k}\sts{0} }^{2}
    }.
    \label{est:k_para:taylor}
\end{align}
By \eqref{est:k_mass:bdd}, \eqref{est:k_mom:bdd},
\eqref{est:f_mass:bdd}, \eqref{est:f_mom:bdd},
\eqref{est:k_para:jac} and \eqref{est:k_para:taylor}, we can obtain
the result.
\end{proof}

\noindent \textbf{Step 4: Conclusion.} Combining
\eqref{tsln:rem:ref} and Lemma \ref{lem:k_para:ref}, we have for any
$t\in [0, T^*]$
\begin{align*}
    \frac{ C_{1} }{ 2 }\normhone{ \varepsilon\sts{t} }^{2}
\leq &\; \frac{C}{L}\sup_{s\in [0, t] }\normhone{
\varepsilon\sts{s} }^{2}
   + C \sup_{s\in [0,t]} \left( \beta \sts{\normhone{ \varepsilon\sts{t} }}  \normhone{ \varepsilon\sts{t} }^{2}\right)
\\  & +
    C\; \normhone{ \varepsilon\sts{0} }^{2}
    +
    C\; e^{ - \theta_0 L }.
     \end{align*}
By \eqref{tsln:sl:rough}, and taking $\alpha_0$ sufficiently small
and $L_0$ sufficiently large, we have for any $t\in [0, T^*]$
\begin{align*}
\normhone{ \varepsilon\sts{t} }^{2} \leq  C\; \normhone{
\varepsilon\sts{0} }^{2}
    +
    C e^{ - \theta_0 L }.
   \end{align*}
This together Lemma \ref{lem:k_para:ref} implies that
\begin{align*}
\normhone{ \varepsilon\sts{t} }^{2} + \sum^{2}_{k=1}\big(
\abs{\omega_{k}\sts{t}- \omega_{k}\sts{0}} + \abs{c_{k}\sts{t}-
c_{k}\sts{0}} \big)\leq C  \normhone{ \varepsilon\sts{0} }^{2}
    +
    C e^{ - \theta_0 L }.
   \end{align*}
Last, we have
\begin{align*}
 & \inf_{x_2^0-x_1^0>\frac{L}{2}, \gamma_k^0\in\R} \normhone{ u\sts{t,\cdot}-\sum^2_{k=1}\varphi_{\omega_k^{0},c_k^{0}}\sts{ \cdot-x_k^0}e^{\i\gamma_k^0} }
  \\
  \leq &
  \normhone{ u\sts{t,\cdot}-\sum^2_{k=1}\varphi_{\omega_k^{0},c_k^{0}}\sts{ \cdot-x_k\sts{t} }e^{\i\gamma_k\sts{t}} }
  \\
  \leq
  &
    \normhone{ u\sts{t, \cdot}-\sum^2_{k=1}\varphi_{\omega_k\sts{t},c_k\sts{t}}\sts{ \cdot-x_k\sts{t} }e^{\i\gamma_k\sts{t}} }
    +
    C \sts{\sum^2_{k=1}\left(\abs{\omega_k(t)-\omega_k^0}+\abs{c_k(t)-c_k^0}\right)}
  \\
  \leq
  &
    \normhone{\varepsilon\sts{t}}
    +
    C\sts{\sum^2_{k=1}\left(\abs{\omega_k\sts{t}-\omega_k\sts{0}} + \abs{c_k\sts{t}-c_k\sts{0}}+\abs{\omega_k\sts{0}-\omega_k^{0}}+ \abs{c_k\sts{0}-c_k^{0}}\right)}
    \\
 \leq   &\;
  C\sts{\normhone{\varepsilon\sts{0}}
    +\sum^2_{k=1}\left(\abs{\omega_k\sts{0}-\omega_k^{0}}+ \abs{c_k\sts{0}-c_k^{0}}\right)} + C e^{-\theta_0 \frac{L}{2}}
   \\
 \leq   & \;   C C_{II}\left(\delta + e^{-\theta_0 \frac{L}{2}}\right).
\end{align*}
If choosing $A_0 \geq 2CC_{II}$, then for any $t\in [0, T^*]$, we
have
\begin{align*}
 \inf_{x_2^0-x_1^0>\frac{L}{2}, \gamma_k^0\in\R} \normhone{ u\sts{t,\cdot}-\sum^2_{k=1}\varphi_{\omega_k^{0},c_k^{0}}\sts{ \cdot-x_k^0}e^{\i\gamma_k^0} }
   \leq
  \frac12 A_0 \delta,
  \end{align*}
which contradicts with the assumption $T^*<+\infty$ by the
continuity of $u(t)$ in $H^1(\R)$. This implies $T^*=+\infty$ and
completes the proof of Theorem \ref{thm:stab:tsln}.


\begin{appendices}
\section{The coercivity of the quadratic term}\label{app:coer:arf}
In this appendix, we prove Proposition \ref{prop:coer:arf}. The
proof of Part (1) is the same as that in Proposition 2.8 (a)
\cite{Wein:stab:SJMA}. As for Part (2), the proof is divided into
several steps.

{\bf Step 1: Spectral decomposition.} First of all, it follows from
the exponential decay of $\phi_{\omega,c}$ that $\mathcal{L}_{+}$ is
a relatively compact perturbation of the operator $-\frac12
\partial_{x}^{2} + \frac12\left(\omega-\frac{c^{2}}{4}\right).$ By Weyl's theorem in
\cite{ReedSimon:book:IV}, we obtain that the essential spectrum of
$\mathcal{L}_{+}$ on $L^{2}\sts{\R}$ is
\begin{align*}
  \sigma_{\ess}\sts{ \mathcal{L}_{+} } &=\sigma_{\ess}\sts{ -\frac12\partial_{x}^{2} + \frac12\left(\omega-\frac{c^{2}}{4}\right) } = \left[ \frac12\left(\omega-\frac{c^{2}}{4}\right)~,~+\infty \right).
\end{align*}
Moreover, all spectrum below the lower bound of the essential
spectrum are either an isolated point of $\sigma\sts{
\mathcal{L}_{+} }$ or an eigenvalue of finite multiplicity of
$\mathcal{L}_{+}.$

Next, since $\phi_{\omega,c}$ satisfies
\begin{align}\label{eq:gr}
        \left(\omega-\frac{c^2}{4}\right)\phi_{\omega,c}
    -   \partial^2_x \phi_{\omega,c}-\frac{3}{16}\phi_{\omega,c}^{5}
    =
    -   \frac{c}{2}\phi_{\omega,c}^{3}\;,
\end{align}
then by differentiating equation \eqref{eq:gr} with respect to $x,$
we obtain
\begin{align}\label{eq:zeroeign}
    \mathcal{L}_{+}\partial_x \phi_{\omega,c} =0.
\end{align}
Therefore, by $\partial_x \phi_{\omega,c}\in L^{2}(\R),$ we obtain
from \eqref{eq:zeroeign} that $0$ is an eigenvalue of
$\mathcal{L}_{+}.$ By a classical ODE argument as in
\cite{Wein:stab:SJMA}, we obtain
\begin{align}\label{eq:kernel}
    \ker \mathcal{L}_{+} = \spann\left\{~\partial_x \phi_{\omega,c}~\right\}.
\end{align}
Thus, it follows from Strum-Liouville theory that $0$ is the second
eigenvalue of $\mathcal{L}_{+},$ and moreover $\mathcal{L}_{+}$
enjoys only one negative eigenvalue $-\lambda^2_{1}$ with a
$L^{2}(\R)$ normalized eigenfunction $\chi.$ More precisely, we have
\begin{align}\label{eq:negv}
    \mathcal{L}_{+}\chi = -\lambda^2_{1}\chi\quad\text{with }\quad \normto{\chi}=1.
\end{align}
Now, define
\begin{align}\label{eq:poseign}
    \mu\triangleq &
        \inf
            \ltl{
             \frac{\bilin{\mathcal{L}_{+}\psi}{\psi}}{\bilin{ \psi }{\psi}}
             ~:~ \psi \in L^2(\R), ~~
                \bilin{\psi}{\chi}
                =
                \bilin{\psi}{\partial_x\phi_{\omega,c}}   = 0~
            },
\end{align}
then by a classical variational argument, it is easy to see that
$\mu>0.$ Therefore, the space $L^{2}(\R)$ can be decomposed as a
direct sum as follows
\begin{align}\label{eq:decomp}
    L^{2}=N\bigoplus \ker \mathcal{L}_{+} \bigoplus P,
\end{align}
where $N=\spann\left\{~\chi~\right\},$ $\ker \mathcal{L}_{+}$ is
defined by \eqref{eq:kernel}, and $P$ is a closed subspace of
$L^{2}$ such that
\begin{align}
\label{eq:pos:a_sys}
    \bilin{\mathcal{L}_{+}\psi}{\psi}\geq \mu \bilin{\psi}{\psi},\quad \text{ for any } \psi\in P.
\end{align}

{\bf Step 2: Nonnegative property.}  We show
\begin{align*}
        \inf
            \ltl{
             \frac{\bilin{\mathcal{L}_{+}\psi}{\psi}}{\bilin{ \psi }{\psi}}
             ~:~ \psi\in L^2, ~~
             \bilin{\psi}{\phi_{\omega,c}}
           =     \bilin{\psi}{\phi_{\omega,c}^3}
           =   \bilin{\psi}{\partial_x\phi_{\omega,c}}   = 0~
            } \geq 0.
\end{align*}
In fact, by differentiating equation \eqref{eq:gr} with respect to
$c$ and $\omega,$ we have
\begin{align}
    \label{eq:stru}
    \mathcal{L}_{+}\partial_{c}\phi_{\omega,c}
    =
    \frac{c}{2}\phi_{\omega,c}-\frac{1}{2}\phi_{\omega,c}^{3},\quad
    \mathcal{L}_{+}\partial_{\omega}\phi_{\omega,c} =-\phi_{\omega,c}.
\end{align}
On one hand, \eqref{eq:decomp} allows us to decompose
$\partial_{c}\phi_{\omega,c}$ and $\partial_{\omega}\phi_{\omega,c}$
as follows,
  \begin{align}
    \label{var:decomp}
    \partial_{c}\phi_{\omega,c}
        = a_{1}\chi + b_{1}\partial_{x}\phi_{\omega,c}+p_{1},\quad
    \partial_{\omega}\phi_{\omega,c}
        = a_{2}\chi+b_{2}\partial_{x}\phi_{\omega,c}+p_{2}.
  \end{align}
where $a_{1},$ $a_{2},$ $b_{1}$ and $b_{2}$ are constants; $\chi$ is
defined by \eqref{eq:negv}; $p_{1}$ and $p_{2}$ belong to the
subspace $P$ defined by \eqref{eq:pos:a_sys}. On the other hand, for
any $\psi\in L^{2}(\R)$ with
\begin{align*}
    \bilin{\psi}{\phi_{\omega,c}}
    =
    \bilin{\psi}{\phi_{\omega,c}^3}
    =
    \bilin{\psi}{\partial_x\phi_{\omega,c}}   = 0,
\end{align*}
we decompose $\psi$ as follows
  \begin{align}
    \label{psi:decomp}
    \psi=a\chi+p,\quad\text{ with } a\in\R,\text{~and~} p\in P.
  \end{align}
By some straight calculations, we have
  \begin{align}
    \label{psi:p}
    \bilin{\mathcal{L}_{+} \psi}{\psi} &= -\lambda^{2}_1a^{2}+\bilin{\mathcal{L}_{+} p}{p}.
  \end{align}
Then, it follows from $\bilin{\psi}{\phi_{\omega,c}}= 0$ and
\eqref{eq:stru} that
  \begin{align*}
    \bilin{\psi}{\mathcal{L}_{+}\partial_{\omega}\phi_{\omega,c}}=0,
  \end{align*}
which together with \eqref{var:decomp} and \eqref{psi:decomp}
implies that
 \begin{align}
    \label{p:b}
    -aa_{2}\lambda^{2}_1+\bilin{\mathcal{L}_{+} p}{p_{2}}=0.
  \end{align}
By a similar argument as above, we have
 \begin{align}
    \label{p:a}
    -aa_{1}\lambda^{2}_1+\bilin{\mathcal{L}_{+} p}{p_{1}}=0.
  \end{align}
Next, since
  \begin{align*}
    \det d''\action{\omega}{c}<0,
  \end{align*}
there exists $\sts{\xi_{1}~,~\xi_{2}}\in\R^{2}$ such that
  \begin{align}
    \label{direct}
    \begin{pmatrix}
      \xi_{1} & \xi_{2} \\
    \end{pmatrix}
    d''\action{\omega}{c}
    \begin{pmatrix}
      \xi_{1} \\
      \xi_{2} \\
    \end{pmatrix}>0.
  \end{align}
Now, let $\sts{\xi_{1}~,~\xi_{2}}\in\R^{2}$ satisfy \eqref{direct},
and
  \begin{align*}
    p_{0}\triangleq~& \xi_{1}p_{1}+\xi_{2}p_{2},
  \end{align*}
then by a straight calculation, we have
  \begin{align}
    \bilin{\mathcal{L}_{+}p_{0}}{p_{0}}=&\;
            \xi_{1}^{2}\bilin{\mathcal{L}_{+}p_{1}}{p_{1}}
        +   2\xi_{1}\xi_{2}\bilin{\mathcal{L}_{+}p_{1}}{p_{2}}
        +   \xi_{2}^{2}\bilin{\mathcal{L}_{+}p_{2}}{p_{2}} \notag
        \\
        =&\;
            \xi_{1}^{2}\bilin{\mathcal{L}_{+}\partial_{c}\phi_{\omega,c}}{\partial_{c}\phi_{\omega,c}}
        +   2\xi_{1}\xi_{2}
            \bilin{\mathcal{L}_{+}\partial_{c}\phi_{\omega,c}}{\partial_{\omega}\phi_{\omega,c}}
            \notag
            \\
        &\;
        +   \xi_{2}^{2}
            \bilin{\mathcal{L}_{+}\partial_{\omega}\phi_{\omega,c}}{\partial_{\omega}\phi_{\omega,c}}
        +       \xi_{1}^{2}a_{1}^{2}\lambda^{2}_1
                +   2\xi_{1}\xi_{2}a_{1}a_{2}\lambda^{2}_1
                +   \xi_{2}^{2}a_{2}^{2}\lambda^{2}_1 \notag
        \\
        =&\;  -   \begin{pmatrix}
                    \xi_{1} & \xi_{2} \\
                \end{pmatrix}
                d''\action{\omega}{c}
                \begin{pmatrix}
                    \xi_{1} \\
                    \xi_{2} \\
                \end{pmatrix}
            +   \sts{a_1\xi_{1}+a_2\xi_{2}}^{2}\lambda^{2}_1 \notag
        \\
        <& \;  \sts{a_{1}\xi_{1}+a_{2}\xi_{2}}^{2}\lambda^{2}_1.     \label{est:p_0}
  \end{align}
Next, by \eqref{p:a}, \eqref{p:b}, \eqref{est:p_0} and the
Cauchy-Schwarz inequality, it is easy to see that,
  \begin{align*}
    \bilin{ \mathcal{L}_{+}p }{p}
    \geq
        \frac{ {\bilin{ \mathcal{L}_{+}p }{p_{0}}}^{2} }{ \bilin{ \mathcal{L}_{+}p_{0} }{p_{0}} }
    &\geq \frac{a^{2}\lambda^{2}_1\sts{ a_{1}\xi_{1}+a_{2}\xi_{2} }^{2} }{ \sts{ a_{1}\xi_{1}+a_{2}\xi_{2} }^{2} }
    =a^{2}\lambda^{2}_1,
  \end{align*}
which, together \eqref{psi:p}, implies that,
  \begin{align*}
    \bilin{ \mathcal{L}_{+}\psi}{ \psi } \geq \;0.
  \end{align*}

{\bf Step 3: Positive property.}  Last we show
\begin{align*}
        \inf
            \ltl{
             \frac{\bilin{\mathcal{L}_{+}\psi}{\psi}}{\bilin{ \psi }{\psi}}
             ~:~ \psi\in L^2(\R), ~
             \bilin{\psi}{\phi_{\omega,c}}
           =     \bilin{\psi}{\phi_{\omega,c}^3}
           =   \bilin{\psi}{\partial_x\phi_{\omega,c}}   = 0~
            } > 0.
\end{align*}

We argue by contradiction. Suppose that there exists a sequence
$\psi_{n}\in L^2(\R)$ such that
  \begin{align*}
    \bilin{ \mathcal{L}_{+}\psi_{n}}{ \psi_{n} }\to   0,
  \end{align*}
with
\begin{align*}
    \bilin{\psi_{n}}{\phi_{\omega,c}}
    =
    \bilin{\psi_{n}}{\phi_{\omega,c}^3}
    =
    \bilin{\psi_{n}}{\partial_x\phi_{\omega,c}}
    = 0
    \text{~~and~~}
    \bilin{\psi_{n}}{\psi_{n}}=1.
\end{align*}
By a decomposition similar as \eqref{psi:decomp}, we have for any
$n$
  \begin{align*}
    \psi_{n}=a_{n}\chi+p_{n},
    \quad\text{ with ~} a_{n}\in\R, \text{~and~} p_{n}\in P,
  \end{align*}
moreover, $\bilin{ \mathcal{L}_{+}p_{n} }{ p_{0}
}=\left(a_1\xi_1+a_2\xi_2\right)a_{n}\lambda^{2}_1.$ Therefore, by
the similar arguments as in Step $2$, we have
  \begin{align*}
     0\leftarrow\bilin{ \mathcal{L}_{+}\psi_{n}}{ \psi_{n} }
     &\geq -a^2_{n}\lambda^{2}_1
        +
        \frac{ a_{n}^{2}\sts{ \xi_{1}a_{1}+\xi_{2}a_{2}
        }^{2}\lambda^{4}_1
            }{
                \bilin{ \mathcal{L}_{+}p_{0} }{ p_{0} }
            }
     \\
     &=a_{n}^{2}\lambda^{2}_1
        \sts{
            \frac{
                \lambda^{2}_1\sts{ \xi_{1}a_{1}+\xi_{2}a_{2} }^{2}
                }{
                    \bilin{ \mathcal{L}_{+}p_{0} }{ p_{0} }
                }
            -1
            }.
  \end{align*}
Thus, it follows from \eqref{est:p_0} that,
  \begin{align*}
    a_{n}\to 0,\qquad \text{as } n\to\infty,
  \end{align*}
which implies that
\begin{align*}
  \bilin{ \mathcal{L}_{+}p_{n} }{ p_{n} }\to 0.
\end{align*}
Thus $p_{n}\to 0$ in $L^2(\R)$, which is in contradiction with
$\bilin{\psi_{n}}{\psi_{n}}=1.$ This ends the proof.


\section{The linearization of the action functional}\label{app:tsln:acfln}
In this part, we show Lemma \ref{lem:tsln:acfln}. First of all, we
show the following claim,
\begin{claim}\label{decay} Let $\mathcal{R}_{k}$ be one of the expression $R_{k},$~$\partial_{x}R_{k}$ and
$\partial_{x}^{2}R_{k}$, and $\mathfrak{g}$~and~$\mathfrak{h}$ be
defined by \eqref{ctf}, then
\begin{align}
\label{w-act}
    \int
        \abs{ \mathcal{R}_{1}\sts{t\,,\,x}~\mathcal{R}_{2}\sts{t\,,\,x} }
    \;\dx
    \leq
    C~
    e^{ -8\theta_{2}\sts{ \frac{L}{2} + 8\theta_{2}\;t } }.
\end{align}
\begin{align}
\label{cutoff-decay}
    \int
        \abs{ \mathcal{R}_{1}\sts{t\,,\,x}~\mathfrak{h}\sts{t\,,\,x} }
    \;\dx
    +
    \int
        \abs{ \mathcal{R}_{2}\sts{t\,,\,x}~\mathfrak{g}\sts{t\,,\,x} }
    \;\dx
    \leq
    C~
    e^{ -8\theta_{2}
        \sts{ \frac{L}{2} + 8\theta_{2}t }
    },
\end{align}
\end{claim}
\begin{proof}

Firstly, by Lemma \ref{lem:tsln:decomp}, we have
\begin{align*}
  \dot{x}_{2}\sts{t}-\dot{x}_{1}\sts{t}
  =
&
    \sts{\; c_{2}^{0}-c_{1}^{0}\;}
    +
    \sts{\; \dot{x}_{2}\sts{t}- c_{2}\sts{t} \;}
    -
    \sts{\; \dot{x}_{1}\sts{t}- c_{1}\sts{t} \;}
\\
&
    +
    \sts{\; c_{2}\sts{t} - c_{2}\sts{0} \;}
    -
    \sts{\; c_{1}\sts{t} - c_{1}\sts{0}\; }
    +
    \sts{\; c_{2}\sts{0} - c_{2}^{0}\; }
    -
    \sts{\; c_{1}\sts{0} - c_{1}^{0} \; }
\\
\geq &
    \frac{ c_{2}^{0}-c_{1}^{0} }{ 4 },
\end{align*}
therefore, integrating in $t>0$ gives us that for $t>0$
\begin{align*}
    x_{2}\sts{t}-x_{1}\sts{t}
   \geq
    \frac{L}{2} + \frac{ c_{2}^{0}-c_{1}^{0} }{ 4 }\; t
    \geq
    \frac{L}{2} + 8\theta_{2}\;t.
\end{align*}
Thus,
\begin{align*}
    \int
        \abs{ \mathcal{R}_{1}\sts{t\,,\,x}~\mathcal{R}_{2}\sts{t\,,\,x} }
    \;\dx
    &
    \leq
    C
    \int
        e^{ -\frac{ \sqrt{ 4\omega_{1}\sts{t}-c_{1}^{2}\sts{t} } }{2}\abs{ x-x_{1}\sts{t} } }
        ~
        e^{ -\frac{ \sqrt{4\omega_{2}\sts{t} -c_{2}^{2}\sts{t} } }{2}\abs{ x-x_{2}\sts{t} } }
    \;\dx
\\
    &
    \leq
    C
    \int
        e^{ -\frac{ \sqrt{ 4\omega_{1}^{0}-\sts{c_{1}^{0}}^{2} } }{4}\abs{ x-x_{1}\sts{t} } }
        ~
        e^{ -\frac{ \sqrt{ 4\omega_{2}^{0}-\sts{c_{2}^{0}}^{2} } }{4}\abs{ x-x_{2}\sts{t} } }
    \;\dx
\\
    &
    \leq
    C
    e^{ -\frac{ \sqrt{ 4\omega_{1}^{0}-\sts{c_{1}^{0}}^{2} } }{8}\abs{ x_{2}\sts{t}-x_{1}\sts{t} }
    }
    \\
    & \leq
    C
    e^{ -8\theta_{2}\sts{ \frac{L}{2} + 8\theta_{2}\;t } }.
\end{align*}

Secondly, as for \eqref{cutoff-decay}, we only estimate the former
term since the later term can be proved in the same way. By Lemma
\ref{lem:tsln:decomp}, we have for sufficiently small $\alpha_0$ and
sufficiently large $L_0$
\begin{align}
\notag \frac{d}{dt}\left(\overline{x}^{0} + \sigma t - \sqrt{t+a} -
x_{1}\sts{t}\right)
    =&\;
    \sigma - \frac{1}{ 2\sqrt{t+a}}-\dot{x}_{1}\sts{t}
    \\
    \notag
    \geq
    &
    \sts{ \sigma-c_{1}^{0} } - \frac{4}{L} - \sts{\; \dot{x}_{1}\sts{t} - c_{1}\sts{t} \;} + \sts{\; c_{1}^{0} - c_{1}\sts{t}\; }
    \\
    \geq
    &
    \frac{\sigma-c_{1}^{0}}{4},
\end{align}
By integrating with respect to $t,$ we obtain
\begin{align}\label{h-rate}
    \notag
\overline{x}^{0} + \sigma t - \sqrt{t+a} - x_{1}\sts{t}
    \geq
    & \;
    \overline{x}^{0} - \frac{L}{8} - x_{1}\sts{0} + \frac{\sigma-c_{1}^{0}}{4}\; t
    \\
    \geq
    &
    \frac{L}{4} + 4\theta_{2}t.
\end{align}
This implies that
\begin{align*}
    \int
        \abs{ \mathcal{R}_{1}\sts{t\,,\,x}~\mathfrak{h}\sts{t\,,\,x} }
    \;\dx
    \leq
    &\;
    C
    \int_{x > \overline{x}^{0} + \sigma t -  \sqrt{t+a}}
        e^{ -\frac{ \sqrt{ 4\omega_{1}\sts{t}-c_{1}^{2}\sts{t} } }{2}\abs{ x-x_{1}\sts{t} } }
    \;\dx
\\
    \leq
    &\;
    C
    \int_{x > \overline{x}^{0} + \sigma t -  \sqrt{t+a}}
        e^{ -\frac{ \sqrt{ 4\omega_{1}^{0}-\sts{c_{1}^{0}}^{2} } }{4}\abs{ x-x_{1}\sts{t} } }
    \;\dx
\\
    \leq & \;
    C
    e^{ -16\theta_{2}
        \sts{ \frac{L}{4} + 4\theta_{2}t }
    }.
\end{align*}
This ends the proof.
\end{proof}
\begin{proof}[Proof of Lemma \ref{lem:tsln:acfln}]
We now expand $\mathfrak{E}\sts{u\sts{t}}$ one by one.

\noindent{\textbf{The term: $\displaystyle\int\abs{ \partial_{x}
u\sts{t} }^{2}\;\dx.$ }} By \eqref{w-act} and integration by parts,
we have
\begin{align*}
    \int\abs{ \partial_{x} u\sts{t} }^{2}\;\dx
= &
    \sum_{k=1}^{2}\int\abs{ \partial_{x} R_{k}\sts{t} }^{2}\;\dx
     -
    \sum_{k=1}^{2}
    2
    \Re\int \partial_{x}^{2}R_{k}\sts{t}\; \overline{\varepsilon}\sts{t}\;\dx
 +
    \int\abs{ \partial_{x} \varepsilon\sts{t} }^{2}\;\dx
    \\
&
      +
    \bigo{ e^{ -8\theta_{2}\sts{ \frac{L}{2} + 8\theta_{2}\;t } } }.
\end{align*}

\noindent{\textbf{The term: $\displaystyle\int
\abs{u\sts{t}}^{6}\;\dx.$ }} By \eqref{w-act} and
Gagliardo-Nirenberg inequality, we have
\begin{align*}
    \int \abs{u\sts{t}}^{6}\;\dx
= &
    \sum_{k=1}^{2}\int
        \abs{ R_{k}\sts{t} }^{6}
    \;\dx
    +
    \sum_{k=1}^{2}
    \int
        6
        \abs{ R_{k}\sts{t} }^{4}\Re\sts{ R_{k} \;\overline{ \varepsilon }}\sts{t}
    \;\dx
\\
&
    +
    \sum_{k=1}^{2}
    \int
        3
        \abs{ R_{k}\sts{t} }^{4} \abs{ \varepsilon\sts{t} }^{2}
        +
        12
        \abs{ R_{k}\sts{t} }^{2} \left[ \Re\sts{ R_{k} \;\overline{ \varepsilon }}\sts{t} \right]^{2}
    \;\dx
\\
&
  +
    \normhone{\varepsilon\sts{t}}^{2}\beta\sts{ \normhone{\varepsilon\sts{t}}
    }
  +
    \bigo{
        e^{ -8\theta_{2}\sts{ \frac{L}{2} + 8\theta_{2}\;t } }
        }
   .
\end{align*}

\noindent{\textbf{The term:
$\displaystyle\frac{\omega_{1}\sts{0}}{2}\int
\abs{u\sts{t}}^{2}\mathfrak{g}\sts{t}\;\dx$ and $\displaystyle
\frac{\omega_{2}\sts{0}}{2}\int
\abs{u\sts{t}}^{2}\mathfrak{h}\sts{t}\;\dx.$ }} By
\eqref{cutoff-decay} and the Cauchy-Schwarz inequality, we have
\begin{align*}
&
    \frac{\omega_{1}\sts{0}}{2}\int \abs{u\sts{t}}^{2}\mathfrak{g}\sts{t}\;\dx
\\
= &
    \frac{\omega_{1}\sts{0}}{2}
    \int
        \abs{\sum_{k=1}^{2} R_{k}\sts{t} + \varepsilon\sts{t}}^{2}\mathfrak{g}\sts{t}
    \;\dx
\\
= &
    \frac{\omega_{1}\sts{0}}{2}
    \int
        \abs{ R_{1}\sts{t} }^{2}
        -
        \abs{ R_{1}\sts{t} }^{2} \mathfrak{h}\sts{t}
        +
        \abs{ R_{2}\sts{t} }^{2} \mathfrak{g}\sts{t}
        +
        2
        \Re\sts{ R_{1}\;\overline{ R_{2} } }\sts{t}\mathfrak{g}\sts{t}
    \;\dx
\\
&
+
    \frac{\omega_{1}\sts{0}}{2}
    \int
        \Re\sts{ R_{1}\;\overline{ \varepsilon } }
        -
        \Re\sts{ R_{1}\;\overline{ \varepsilon } }\sts{t}\mathfrak{h}\sts{t}
        +
        \Re\sts{ R_{2}\;\overline{ \varepsilon } }\sts{t}\mathfrak{g}\sts{t}
    \;\dx
\\
&
+
    \frac{\omega_{1}\sts{0}}{2}
    \int
        \abs{ \varepsilon\sts{t} }^{2}\mathfrak{g}\sts{t}
    \;\dx
\\
=
&
    \frac{\omega_{1}\sts{0}}{2}
    \int
        \abs{ R_{1}\sts{t} }^{2}
    \;\dx
+
    \frac{\omega_{1}\sts{0}}{2}
    \int
        \Re\sts{ R_{1}\;\overline{ \varepsilon } }
    \;\dx
+
    \frac{\omega_{1}\sts{0}}{2}
    \int
        \abs{ \varepsilon\sts{t} }^{2}\mathfrak{g}\sts{t}
    \;\dx
\\
&
+
    \bigo{
        \int
            \abs{ R_{1}\sts{t} }^{2} \mathfrak{h}\sts{t}
            +
            \abs{ R_{2}\sts{t} }^{2} \mathfrak{g}\sts{t}
            +
            \abs{ R_{1}\sts{t}}\;\abs{ R_{2}\sts{t}}\mathfrak{g}\sts{t}
        \;\dx
    }
\\
&
    +
    \bigo{
        \int
            \abs{ R_{1}\sts{t}}\;\abs{ \varepsilon\sts{t}}\mathfrak{h}\sts{t}
            +
            \abs{ R_{2}\sts{t}}\;\abs{ \varepsilon\sts{t}}\mathfrak{g}\sts{t}
        \;\dx
    }
\\
=
&
    \frac{\omega_{1}\sts{0}}{2}
    \int
        \abs{ R_{1}\sts{t} }^{2}
    \;\dx
    +
    \frac{\omega_{1}\sts{0}}{2}
    \int
        \Re\sts{ R_{1}\;\overline{ \varepsilon } }
    \;\dx
    +
    \frac{\omega_{1}\sts{t}}{2}
    \int
        \abs{ \varepsilon\sts{t} }^{2}\mathfrak{g}\sts{t}
    \;\dx
\\
&
    +
    \frac{ \omega_{1}\sts{0} - \omega_{1}\sts{t} }{ 2 }
    \int
        \abs{ \varepsilon\sts{t} }^{2}\mathfrak{g}\sts{t}
    \;\dx
\\
&
+
    \bigo{
        \int
            \abs{ R_{1}\sts{t} }^{2} \mathfrak{h}\sts{t}
            +
            \abs{ R_{2}\sts{t} }^{2} \mathfrak{g}\sts{t}
            +
            \abs{ R_{1}\sts{t}}\;\abs{ R_{2}\sts{t}}\mathfrak{g}\sts{t}
        \;\dx
    }
\\
&
    +
    \bigo{
        \int
            \abs{ R_{1}\sts{t}}\;\abs{ \varepsilon\sts{t}}\mathfrak{h}\sts{t}
            +
            \abs{ R_{2}\sts{t}}\;\abs{ \varepsilon\sts{t}}\mathfrak{g}\sts{t}
        \;\dx
    }
\\
= &
    \frac{\omega_{1}\sts{0}}{2}
    \int
        \abs{ R_{1}\sts{t} }^{2}
    \;\dx
    +
    \frac{\omega_{1}\sts{0}}{2}
    \int
        \Re\sts{ R_{1}\;\overline{ \varepsilon } }
    \;\dx
    +
    \frac{\omega_{1}\sts{t}}{2}
    \int
        \abs{ \varepsilon\sts{t} }^{2}\mathfrak{g}\sts{t}
    \;\dx
\\
&
    +
    \normhone{\varepsilon\sts{t}}^{2} \beta\sts{ \normhone{\varepsilon\sts{t}} }
    +
    \bigo{ \abs{ \omega_{1}\sts{t} - \omega_{1}\sts{0} }^{2} }
    +
    \bigo{
        e^{ -8\theta_{2}
        \sts{ \frac{L}{2} + 8\theta_{2}t }
        }
    },
\end{align*}
and
\begin{align*}
&
    \frac{\omega_{2}\sts{0}}{2}\int \abs{u\sts{t}}^{2}\mathfrak{h}\sts{t}\;\dx
\\
= &
    \frac{\omega_{2}\sts{0}}{2}
    \int
        \abs{ R_{2}\sts{t} }^{2}
    \;\dx
    +
    \frac{\omega_{2}\sts{0}}{2}
    \int
        \Re\sts{ R_{2}\;\overline{ \varepsilon } }
    \;\dx
       +
    \frac{\omega_{2}\sts{t}}{2}
    \int
        \abs{ \varepsilon\sts{t} }^{2}\mathfrak{g}\sts{t}
    \;\dx
 \\
&   +
    \normhone{\varepsilon\sts{t}}^{2} \beta\sts{ \normhone{\varepsilon\sts{t}} }
    +
    \bigo{ \abs{ \omega_{2}\sts{t} - \omega_{2}\sts{0} }^{2} }
    +
    \bigo{
        e^{ -8\theta_{2}
        \sts{ \frac{L}{2} + 8\theta_{2}t }
        }
    }.
\end{align*}

\noindent \textbf{The term $\displaystyle - \frac{c_{1}\sts{0}}{2}\;
        \Im\int\sts{ \overline{u}\partial_{x}u }\sts{t}\mathfrak{g}\sts{t}\;\dx$ and $\displaystyle-   \frac{c_{2}\sts{0}}{2}\;
        \Im\int\sts{ \overline{u}\partial_{x}u
        }\sts{t}\mathfrak{h}\sts{t}\;\dx.$} Similarly, we have
\begin{align*}
&
    -   \frac{c_{1}\sts{0}}{2}\;
        \Im\int\sts{ \overline{u}\partial_{x}u }\sts{t}\mathfrak{g}\sts{t}\;\dx
\\
 = &
    -   \frac{c_{1}\sts{0}}{2}\;
        \Im\int \overline{R}_{1}\; \partial_{x} R_{1}\sts{t}    \;\dx
    -   c_{1}\sts{0}\;
        \Im\int \partial_{x} R_{1}\sts{t}\;\overline{\varepsilon} \;\dx
-   \frac{c_{1}\sts{t}}{2}\;
        \Im\int \overline{\varepsilon}\; \partial_{x} \varepsilon\sts{t}\mathfrak{g}\sts{t}\;\dx
  \\
&
      +
    \normhone{\varepsilon\sts{t}}^{2} \beta\sts{ \normhone{\varepsilon\sts{t}} }
 +
    \bigo{ \abs{ c_{1}\sts{t} - c_{1}\sts{0} }^{2} }
+
    \bigo{
        e^{ -8\theta_{2}
        \sts{ \frac{L}{2} + 8\theta_{2}t }
        }
    },
\end{align*}
and
\begin{align*}
&
    -   \frac{c_{2}\sts{0}}{2}\;
        \Im\int\sts{ \overline{u}\partial_{x}u }\sts{t}\sts{t}\mathfrak{h}\sts{t}\;\dx
\\
= &
    -   \frac{c_{2}\sts{0}}{2}\;
        \Im\int \overline{R}_{2}\; \partial_{x} R_{2}\sts{t}    \;\dx
    -   c_{2}\sts{0}\;
        \Im\int \partial_{x} R_{2}\sts{t}\;\overline{\varepsilon} \;\dx
 -   \frac{c_{2}\sts{t}}{2}\;
        \Im\int \overline{\varepsilon}\; \partial_{x} \varepsilon\sts{t}\mathfrak{h}\sts{t}\;\dx
  \\
&
     +
        \normhone{\varepsilon\sts{t}}^{2} \beta\sts{ \normhone{\varepsilon\sts{t}} }
 +
    \bigo{  \abs{ c_{2}\sts{t} - c_{2}\sts{0}}^{2} }
+
    \bigo{
        e^{ -8\theta_{2}
        \sts{ \frac{L}{2} + 8\theta_{2}t }
        }
    }.
\end{align*}

\noindent \textbf{The term $\displaystyle \frac{c_{1}\sts{0}}{8}\int
\abs{u\sts{t}}^{4}\mathfrak{g}\sts{t}\;\dx$ and
$\displaystyle\frac{c_{2}\sts{0}}{8}\int
\abs{u\sts{t}}^{4}\mathfrak{h}\sts{t}\;\dx.$}
\begin{align*}
&
    \frac{c_{1}\sts{0}}{8}\int \abs{u\sts{t}}^{4}\mathfrak{g}\sts{t}\;\dx
\\
= &
    \frac{c_{1}\sts{0}}{8}
    \int
            \abs{R_{1}\sts{t}}^{4}
    \;\dx
    +
    \frac{c_{1}\sts{0}}{4}
    \int
            \abs{R_{1}\sts{t}}^{2}\Re\sts{R_{1}\;\overline{\varepsilon}}\sts{t}
    \;\dx
    \\
&
    +
    \frac{c_{1}\sts{t}}{8}
    \int
            2
            \abs{R_{1}\sts{t}}^{2}\abs{\varepsilon\sts{t}}^{2}
            +
            \left[
                 \Re\sts{R_{1}\;\overline{\varepsilon}}\sts{t}
            \right]^{2}
    \;\dx
    \\
& +
    \normhone{\varepsilon\sts{t}}^{2} \beta\sts{ \normhone{\varepsilon\sts{t}} }
+
    \bigo{  \abs{ c_{1}\sts{t} - c_{1}\sts{0}}^{2} }
+
    \bigo{
        e^{ -8\theta_{2}
        \sts{ \frac{L}{2} + 8\theta_{2}t }
        }
    },
\end{align*}
and
\begin{align*}
&
    \frac{c_{2}\sts{0}}{8}\int \abs{u\sts{t}}^{4}\mathfrak{h}\sts{t}\;\dx
\\
= &
    \frac{c_{2}\sts{0}}{8}
    \int
            \abs{R_{2}\sts{t}}^{4}
    \;\dx
    +
    \frac{c_{2}\sts{0}}{4}
    \int
            \abs{R_{2}\sts{t}}^{2}\Re\sts{R_{2}\;\overline{\varepsilon}}\sts{t}
    \;\dx
\\
&
    +
    \frac{c_{2}\sts{t}}{8}
    \int
            2
            \abs{R_{2}\sts{t}}^{2}\abs{\varepsilon\sts{t}}^{2}
            +
            \left[
                 \Re\sts{R_{2}\;\overline{\varepsilon}}\sts{t}
            \right]^{2}
    \;\dx
   \\
&  +
    \normhone{\varepsilon\sts{t}}^{2} \beta\sts{ \normhone{\varepsilon\sts{t}} }
+
    \bigo{  \abs{ c_{2}\sts{t} - c_{2}\sts{0}}^{2} }
+
    \bigo{
        e^{ -8\theta_{2}
        \sts{ \frac{L}{2} + 8\theta_{2}t }
        }
    }.
\end{align*}
Summing up the above terms, we can conclude the proof by
\eqref{eq:tw:tvar} for $k=1, 2$ and  the orthogonal conditions
\eqref{tsl:orth:nondeg}.
\end{proof}


\section{the coercivity of the localized quadratic term}\label{app:tsln:coer} Let $L$ be
large enough, $x_{1}$ and $x_{2}\in\R$ with
$x_{2}-x_{1}>\frac{L}{2}.$ Now, define
\begin{align}
\label{eq:app:temp:g}
    g\sts{x}=
    \left\{
        \begin{array}{ll}
          1,         &   x\leq x_{1}+\frac{L}{8},                \\
          0<\cdot<1, &   x_{1}+\frac{L}{8} < x < x_{2}-\frac{L}{8},   \\
          0,         &   x\geq x_{2}-\frac{L}{8},
        \end{array}
    \right. \quad\text{and}\; h\sts{x} = 1 - g\sts{x}.
\end{align}

In order to prove Lemma \ref{lem:tsln:coer}, it suffices to show the
following result.
\begin{lemma}
\label{lem:static:2coercivity} Let $L>1$ be large enough, $g,$ $h$
be given by \eqref{eq:app:temp:g}. Then there exists $C_1>0$ such
that
     \begin{align*}
       \mathcal{H}_{2}\sts{~\varepsilon~,~\varepsilon~}\geq C_1 \normhone{ \varepsilon },
     \end{align*}
     where
     \begin{align*}
       \mathcal{H}_{2}\sts{~\varepsilon~,~\varepsilon~}
       =
       &
       \frac{1}{2}\int \abs{\varepsilon_{x}}^{2}
       -
       \frac{1}{32}
       \left(
        3
        \int \abs{R_{1}}^{4}\abs{\varepsilon}^{2}
        +
        12
        \int\abs{R_{1}}^{2}\left[ \Re\sts{ \overline{R}_{1}\varepsilon } \right]^{2}
       \right)
     \\
     &
     -
       \frac{1}{32}
       \left(
        3
        \int \abs{R_{2}}^{4}\abs{\varepsilon}^{2}
        +
        12
        \int\abs{R_{2}}^{2}\left[ \Re\sts{ \overline{R}_{2}\varepsilon } \right]^{2}
       \right)
     \\
     &
       +
       \frac{ \omega_{1} }{ 2 }\int\abs{ \varepsilon }^{2}g
       -
       \frac{c_{1}}{2}\Im\int \overline{\varepsilon}\varepsilon_{x}g
       +
       \frac{c_{1}}{8}
        \sts{
            2
            \int \abs{R_{1}}^{2}\abs{\varepsilon}^{2}
            +
            4
            \int\left[ \Re\sts{ \overline{R}_{1}\varepsilon } \right]^{2}
        }
     \\
     &
       +
       \frac{ \omega_{2} }{ 2 }\int\abs{ \varepsilon }^{2}h
       -
       \frac{c_{2}}{2}\Im\int \overline{\varepsilon}\varepsilon_{x}h
       +
       \frac{c_{2}}{8}
        \sts{
            2
            \int \abs{R_{2}}^{2}\abs{\varepsilon}^{2}
            +
            4
            \int\left[ \Re\sts{ \overline{R}_{2}\varepsilon } \right]^{2}
        },
     \end{align*}
     with $R_{k}\sts{x}=\varphi_{\omega_{k},c_{k}}\sts{x - x_{k}} e^{\i \gamma_{k}}$\,$\sts{ k=1,\,2}.$
\end{lemma}
First, we give a localized version of the `single solitary' coercive
result. For the convenience of notation, we denote
   \begin{align*}
     R\sts{x} = \varphi_{\omega,c}\sts{x-y_0}e^{\i\gamma},
   \end{align*}
   with $4\omega>c^{2},$ $y_{0},\,\theta\in\R.$
Let $\Phi\,:\,\R\mapsto \R$ be an even $C^{2}$ function with
   \begin{align*}
        \Phi\sts{x}=
        \left\{
            \begin{array}{ll}
                1,         &   \abs{x}\leq 1,                                \\
                e^{-|x|}\leq\cdot\leq 3 e^{-|x|}, &   1<\abs{x}<2,   \\
                e^{-|x|},         &   \abs{x}\geq 2,
            \end{array}
        \right.
   \end{align*}
   and $\Phi'\sts{x}\leq 0$ for $x>0.$
\begin{lemma}
\label{claim:app:loc:coercive} Let $B>1$ be large enough. If
$\varepsilon\in H^{1}\sts{\R}$ satisfies the following orthogonality
condition,
\begin{align*}
&
        \Re\int
 R(x)\; \overline{\varepsilon(x)}
        \;\dx =0,
\quad
        \Re\int\sts{\i \partial_x R + \frac12
          \abs{R}^2R} (x)\; \overline{\varepsilon(x)}
        \;\dx =0,
\\
 &   \Re\int \partial_x R(x)\; \overline{\varepsilon(x)}
        \;\dx =0,
 \quad
   \Re\int\i R(x)\;
           \overline{ \varepsilon(x)}
        \;\dx =0.
\end{align*}
   Then, we have
   \begin{align*}
        \mathcal{H}_{ B, y_{0} }\bilin{ \varepsilon }{ \varepsilon }
        \geq
            \frac{C_{0}}{4}\int \sts{ \abs{\varepsilon_{x}}^{2} + \abs{\varepsilon}^{2} } \Phi_{B,y_{0}}\, \dx,
   \end{align*}
   where
   \begin{align*}
        \mathcal{H}_{ B, y_{0} }\bilin{ \varepsilon }{ \varepsilon }
        =
        &
        \frac{1}{2}\int\abs{\varepsilon_{x}}^{2}\Phi_{B,y_{0}}\,\dx
        +
        \frac{\omega}{2}\int\abs{\varepsilon}^{2}\Phi_{B,y_{0}}\,\dx
        -
        \frac{c}{2}\Im\int\overline{\varepsilon}\varepsilon_{x}\Phi_{B,y_{0}}\,\dx
        \\
        &
        +
        \frac{c}{8}
        \sts{
            2   \int\abs{ R
            }^{2}\abs{\varepsilon}^{2}\Phi_{B,y_{0}}\,\dx
            +
            4   \int \left[ \Re\sts{\overline{R}\varepsilon}
            \right]^{2}\Phi_{B,y_{0}}\,\dx
        }
        \\
        &
        -
        \frac{1}{32}
        \sts{
            3 \int\abs{ R
            }^{4}\abs{\varepsilon}^{2}\Phi_{B,y_{0}}\,\dx
            +
            12
            \int \abs{ R }^{2}\left[ \Re\sts{\overline{R}\varepsilon}
            \right]^{2}\Phi_{B,y_{0}}\,\dx
        }.
   \end{align*}
\end{lemma}
\begin{proof}
By setting
$\zeta\sts{x}=\sqrt{\Phi_{B,y_{0}}\sts{x}}\varepsilon\sts{x},$  we
have
\begin{align*}
&
    \abs{\varepsilon_{x}}^{2}\Phi_{B,y_{0}}
    =
    \abs{\zeta_{x}}^{2} - \frac{ \Phi_{B,y_{0}}' }{ \Phi_{B,y_{0}} }\Re\sts{ \overline{\zeta}\zeta_{x} }
    +
    \frac{1}{4} \sts{ \frac{ \Phi_{B,y_{0}}' }{ \Phi_{B,y_{0}} } }^{2}\abs{\zeta}^{2}.
\\
&
    \Im\sts{ \overline{\varepsilon}\varepsilon_{x} } \Phi_{B,y_{0}}
    =
    \Im\sts{ \overline{\zeta}\zeta_{x} },
    \quad\text{ and }\quad
    \abs{\varepsilon}^{2}\Phi_{B,y_{0}} = \abs{\zeta}^{2}.
\end{align*}
Now, we rewrite the quadratic form $\mathcal{H}_{ B, y_{0} }$ as a
quadratic form with respect to $\zeta,$ which means
\begin{align*}
    \mathcal{H}_{ B, y_{0} }\bilin{\varepsilon}{\varepsilon}
    =
        &
    \mathcal{H}_{\omega,c}\bilin{ \zeta }{ \zeta }
    -
    \frac{1}{2}\Re\int \frac{ \Phi_{B,y_{0}}' }{ \Phi_{B,y_{0}}
    }\overline{\zeta}\zeta_{x}\,\dx
    -
    \frac{1}{4}
    \int\sts{ \frac{ \Phi_{B,y_{0}}' }{ \Phi_{B,y_{0}} } }^{2}\abs{\zeta}^{2}\,\dx.
\end{align*}
On the one hand, as a consequence of Lemma \ref{lem:asln:quadcoer},
we obtain
\begin{multline*}
    \mathcal{H}_{\omega,c}\bilin{ \zeta }{ \zeta }
    \geq
    \frac{C_{0}}{2}\normhone{\zeta}^{2}
    -
    \frac{ 2 }{ C_{0} }
    \left[
    \sts{ \Re\int R\; \overline{\zeta} \,\dx }^{2}
    +
    \sts{ \Re\int\sts{\i \partial_x R + \frac12
          \abs{R}^2R}\; \overline{\zeta} \,\dx
    }^{2}
    \right]
\\
    -
    \frac{ 2 }{C_{0} }
    \left[
    \sts{
        \Re\int \partial_x R\; \overline{\zeta} \,\dx
    }^{2}
    +
    \sts{
        \Re\int\i R\;
           \overline{ \zeta} \,\dx
    }^{2}
    \right].
\end{multline*}
On the other hand, a straight calculation implies that
\begin{align*}
    \abs{ \Re\int R\; \overline{\zeta} \,\dx}
    =
    &
    \abs{ \Re\int R\; \overline{\varepsilon}\sts{ 1-\sqrt{\Phi_{B,y_{0}}} } \,\dx}
    \\
    =
    &
    \abs{ \Re\int_{ \abs{ x-y_{0} }>B }R\; \overline{\varepsilon}\sts{ 1-\sqrt{\Phi_{B,y_{0}}} } \,\dx}
    \\
    \leq
    &
    \normto{\varepsilon}\sts{ \int_{ \abs{ x-y_{0} }>B }\abs{R}^{2} \,\dx}^{\frac{1}{2}}
    \\
    \leq
    &
    \frac{ C  }{ e^{ \sqrt{ 4\omega-c^{2} }\frac{B}{2} } }\normto{\varepsilon},
\end{align*}
moreover, applying the similar argument to $\Re\int\sts{\i
\partial_x R + \frac12
          \abs{R}^2R}\; \overline{\zeta},
$ $\Re\int \partial_x R\; \overline{\zeta}$ and $\Re\int i
R\;\overline{ \zeta}$ gives us that
\begin{align*}
  \mathcal{H}_{\omega,c}\bilin{ \zeta }{ \zeta }
  \geq
  \frac{C_{0}}{2}\normhone{\zeta}^{2}
  -
  \frac{ C  }{ e^{ \sqrt{ 4\omega-c^{2} }B } }\normto{\zeta}^{2}.
\end{align*}
Now it follows from $\abs{ \Phi_{B,y_{0}}' }\leq \frac{C
}{B}\Phi_{B,y_{0}}$ that, for $B>1$ large enough,
\begin{align*}
   \mathcal{H}_{ B,y_{0} }\bilin{\varepsilon}{\varepsilon}
    \geq
    \frac{C_{0}}{2}\normhone{\zeta}^{2}
    -
    \frac{ C  }{ e^{ \sqrt{ 4\omega-c^{2} }B } }\normto{\zeta}^{2}
    -
    \frac{C}{B^{2}} \normhone{\zeta}^{2}
    \geq \frac{3C_{0}}{8}\normhone{\zeta}^{2}.
\end{align*}
Since
\begin{align*}
    \abs{\zeta_{x}}^{2}
    =
    &
    \abs{\varepsilon_{x}}^{2}\Phi_{B,y_{0}}
    +
    \frac{ \Phi_{B,y_{0}}' }{ \Phi_{B,y_{0}} }\Re\sts{ \overline{\varepsilon}\varepsilon_{x} }\Phi_{B,y_{0}}
    +
    \frac{1}{4} \sts{ \frac{ \Phi_{B,y_{0}}' }{ \Phi_{B,y_{0}} } }^{2}
    \abs{\varepsilon}^{2}\Phi_{B,y_{0}}
    \\
    \geq
    &
    \sts{ 1-\frac{C}{B^{2}} } \abs{\varepsilon_{x}}^{2}\Phi_{B,y_{0}}
    -
    \frac{C}{B^{2}}\abs{\varepsilon}^{2}\Phi_{B,y_{0}}
\end{align*}
we obtain, for $B$ large enough,
\begin{align*}
    \mathcal{H}_{ B, y_{0} }\bilin{\varepsilon}{\varepsilon}
    \geq
    \frac{C_{0}}{4}\int \sts{ \abs{\varepsilon_{x}}^{2} + \abs{\varepsilon}^{2} } \Phi_{B,y_{0}} \,\dx.
\end{align*}
This ends the proof.
\end{proof}
\begin{proof}[Proof of Lemma \ref{lem:static:2coercivity}]
Since $L>1$ is sufficiently large enough, we can take
$B\in\sts{1,~\frac{L}{4}~}$ such that Claim
\ref{claim:app:loc:coercive} holds.
     \begin{align*}
       \mathcal{H}_{2}\sts{~\varepsilon~,~\varepsilon~}
       =
       &
       \mathcal{H}_{ B,x_{1} }\bilin{ \varepsilon }{ \varepsilon }
       +
       \mathcal{H}_{ B,x_{2} }\bilin{ \varepsilon }{ \varepsilon }
       \\
       &
       +
       \frac{1}{2}
       \int
        \left[  \abs{\varepsilon_{x}}^{2}
                +
                \omega_{1}\abs{\varepsilon}
                -
                c_{1}\Im\sts{ \overline{\varepsilon}\varepsilon_{x} }
        \right]\sts{ g - \Phi_{B,x_{1}} }\,\dx
       \\
       &
       +
       \frac{1}{2}
       \int
        \left[  \abs{\varepsilon_{x}}^{2}
                +
                \omega_{2}\abs{\varepsilon}
                -
                c_{2}\Im\sts{ \overline{\varepsilon}\varepsilon_{x} }
        \right]\sts{ h - \Phi_{B,x_{2}} }\,\dx
        \\
        &
        +
        \frac{c_{1}}{8}
        \sts{
            2
            \int \abs{R_{1}}^{2}\abs{\varepsilon}^{2}
            \sts{ 1 - \Phi_{B,x_{1}} }\,\dx
            +
            4
            \int\left[ \Re\sts{ \overline{R}_{2}\varepsilon } \right]^{2}
            \sts{ 1 - \Phi_{B,x_{1}} }\,\dx
        }
       \\
        &
        +
        \frac{c_{2}}{8}
        \sts{
            2
            \int \abs{R_{2}}^{2}\abs{\varepsilon}^{2}
            \sts{ 1 - \Phi_{B,x_{2}} }\,\dx
            +
            4
            \int\left[ \Re\sts{ \overline{R}_{1}\varepsilon } \right]^{2}
            \sts{ 1 - \Phi_{B,x_{2}} }\,\dx
        }
       \\
       &
       -
       \frac{1}{32}
       \left(
        3
        \int \abs{R_{1}}^{4}\abs{\varepsilon}^{2}
        \sts{ 1 - \Phi_{B,x_{1}} }\,\dx
        +
        12
        \int\abs{R_{1}}^{2}\left[ \Re\sts{ \overline{R}_{1}\varepsilon } \right]^{2}
        \sts{ 1 - \Phi_{B,x_{1}}}\,\dx
       \right)
       \\
       &
       -
       \frac{1}{32}
       \left(
        3
        \int \abs{R_{2}}^{4}\abs{\varepsilon}^{2}
        \sts{ 1 - \Phi_{B,x_{2}} }\,\dx
        +
        12
        \int\abs{R_{2}}^{2}\left[ \Re\sts{ \overline{R}_{2}\varepsilon } \right]^{2}
        \sts{ 1 - \Phi_{B,x_{2}} }\,\dx
       \right).
     \end{align*}
It follows from a direct computation that
\begin{align*}
  g\sts{x}-\Phi_{B,x_{1}}\sts{x}
  \left\{
    \begin{array}{ll}
      = 0, & \abs{ x-x_{1} }<\frac{L}{8}, \\
      \geq -e^{-\frac{L}{8B}}, & \text{ else },
    \end{array}
  \right.
\\
  h\sts{x}-\Phi_{B,x_{2}}\sts{x}
  \left\{
    \begin{array}{ll}
      = 0, & \abs{ x-x_{2} }<\frac{L}{8}, \\
      \geq -e^{-\frac{L}{8B}}, & \text{ else },
    \end{array}
  \right.
\\
  1-\Phi_{B,x_{1}}\sts{x}
  \left\{
    \begin{array}{ll}
      = 0, & \abs{ x-x_{1} }<\frac{L}{8}, \\
      \geq -e^{-\frac{L}{8B}}, & \text{ else },
    \end{array}
  \right.
\end{align*}
and
\begin{align*}
  1-\Phi_{B,x_{2}}\sts{x}
  \left\{
    \begin{array}{ll}
      = 0, & \abs{ x-x_{2} }<\frac{L}{8}, \\
      \geq -e^{-\frac{L}{8B}}, & \text{ else }.
    \end{array}
  \right.
\end{align*}
Moreover, since, for $k=1,2,$ $c_{k}^{2}<4\omega_{k},$ there exists
$\delta_{k}>0$ such that
\begin{align*}
    \abs{\varepsilon_{x}}^{2}
    +
    \omega_{k}\abs{\varepsilon}^{2}
    -
    c_{k}\Im\sts{ \overline{\varepsilon}\varepsilon_{x} }
    \geq
    \delta_{k}\sts{ \abs{\varepsilon_{x}}^{2} + \abs{\varepsilon}^{2}   }.
\end{align*}
Thus, taking $L$ large enough, we obtain
\begin{align*}
    \mathcal{H}_{2}\sts{~\varepsilon~,~\varepsilon~}
    \geq
    &\;
    \frac{C_{0}}{4}
    \int \sts{ \abs{\varepsilon_{x}}^{2} + \abs{\varepsilon}^{2} }
    \Phi_{B,x_{1}}\,\dx
    +
    \frac{C_{0}}{4}
    \int \sts{ \abs{\varepsilon_{x}}^{2} + \abs{\varepsilon}^{2} }
    \Phi_{B,x_{2}}\,\dx
    \\
    &
    +
    \delta_{1}
    \int \sts{ \abs{\varepsilon_{x}}^{2} + \abs{\varepsilon}^{2}   }\sts{ g - \Phi_{B,x_{1}}
    }\,\dx
    +
    \delta_{2}
    \int \sts{ \abs{\varepsilon_{x}}^{2} + \abs{\varepsilon}^{2}   }\sts{ h - \Phi_{B,x_{2}}
    }\,\dx
    \\
    &
    -
    C  e^{-\frac{L}{4B}}
    \int    \sts{ \abs{\varepsilon_{x}}^{2} + \abs{\varepsilon}^{2}
    }\,\dx
    -
    C e^{ -\sqrt{ 4\omega_{1}-c_{1}^{2} }\frac{L}{4} }\int
    \abs{\varepsilon}^{2}\,\dx
    -
    C e^{ -\sqrt{ 4\omega_{2}-c_{2}^{2} }\frac{L}{4} }\int
    \abs{\varepsilon}^{2}\,\dx
    \\
     \geq   &\;
   C_1
    \normhone{ \varepsilon }^{2},
\end{align*}
where $C_1 =
\frac{1}{2}\min\ltl{\frac{C_{0}}{4},\delta_{1},\delta_{2}}.$
This concludes the proof.
\end{proof}
\end{appendices}

\subsection*{Acknowledgements.} The authors would like to thank the referee for his/her
valuable comments and suggestions to help us to improve this paper. A few months after we submitted our paper, S. Le Coz and Y. Wu  obtained the
stability of a $k$-soliton solution of (DNLS) independently in \cite{LeWu:DNLS}.  The authors were partially supported by the NSF grant of China and partially supported by Beijing Center of Mathematics and Information Interdisciplinary Science.


\end{document}